\documentclass[12pt,a4paper]{article}

\usepackage[bbgreekl]{mathbbol}
\usepackage{amsfonts}

\DeclareSymbolFontAlphabet{\mathbb}{AMSb}
\DeclareSymbolFontAlphabet{\mathbbl}{bbold}

\usepackage{amsmath,amssymb,nccmath}
\usepackage{theorem}
\usepackage{graphicx,float}
\usepackage{tcolorbox}
\usepackage{epstopdf}
\usepackage{subfig}
\usepackage{array}
\usepackage{xspace}
\usepackage{footmisc}
\usepackage{lipsum}
\usepackage{afterpage}
\usepackage{dcolumn,booktabs}
\usepackage{siunitx}
\usepackage{etoolbox,multirow}
\usepackage{soul}
\usepackage{pdflscape}
\usepackage{setspace}
\usepackage{algorithm,algcompatible,amsmath,algpseudocode}
\usepackage{adjustbox}
\usepackage{color}
\usepackage{diagbox,geometry}
\usepackage{makecell,tabularx,array}
\usepackage{lmodern}
\usepackage{thmtools}
\usepackage{hyperref}
\usepackage{enumerate}
\usepackage{color,colortbl}
\usepackage{subeqnarray}
\usepackage{multirow}
\usepackage{dsfont}
\usepackage{hhline}
\usepackage{rotating}
\usepackage{natbib}

\usepackage[toc,page]{appendix} 
\usepackage{verbatim} 
\usepackage{siunitx} 
\usepackage{diagbox,float} 
\usepackage{makecell,booktabs,tabularx,array,nccmath} 
\usepackage{etoolbox,multirow} 
\usepackage{setspace} 
\usepackage{thmtools} 
\usepackage{verbatim} 

\allowdisplaybreaks

\algrenewcommand\algorithmicrequire{\textbf{Input:}}
\algnewcommand\algorithmicinput{\textbf{Initialize:}}
\algnewcommand{\Initialize}{\item[\algorithmicinput]}
\algrenewcommand\algorithmicensure{\textbf{Output:}}

\renewrobustcmd{\bfseries}{\fontseries{b}\selectfont}
\renewrobustcmd{\boldmath}{}

\newrobustcmd{\BB}{\bfseries}

{\theorembodyfont{\upshape} }
\newtheorem{proposition}{Proposition}
\newtheorem{lemma}{Lemma}

\newtheorem{theorem}{Theorem}

\DeclareMathAlphabet\mathbfcal{OMS}{cmsy}{b}{n}

\newcommand{\keywords}[1]{\small \textbf{\textit{Keywords:}} #1}

\newenvironment{proof}[1]{\textbf{Proof of {#1}}}{ \hfill$\square$}

\newcommand{\vx}{\vec x}
\newcommand{\vu}{\vec u}
\newcommand{\vc}{\vec c}
\newcommand{\vzero}{\vec 0}
\newcommand{\vbeta}{\vec \beta}

\newcommand{\ccdot}{\cdot }

\begin{document}


\title{\bfseries {\Large Dynamic Basis Function Generation for Network Revenue Management}
}

\author{\small Daniel Adelman, Christiane Barz, Alba V. Olivares-Nadal\thanks{The University of Chicago, Booth School of Business, \textit{Dan.Adelman@chicagobooth.edu}, University of Zurich, \textit{christiane.barz@uzh.ch}, UNSW Business School, \textit{a.olivares\_nadal@unsw.edu.au}}}

\date{\small \today}
\maketitle
\thispagestyle{empty}

\begin{abstract}
    This paper introduces an algorithm that dynamically generates basis functions to approximate the value function in Network Revenue Management. Unlike existing algorithms sampling the parameters of new basis functions, this  Nonlinear Incremental Algorithm (NLIAlg)  iteratively refines the value function approximation by optimizing these parameters.
For larger instances, the Two-Phase Incremental Algorithm (2PIAlg) modifies NLIAlg to leverage the efficiency of LP solvers. It reduces the size of a large-dimensional nonlinear problem and transforms it into an LP by fixing the basis function parameters, which are then optimized in a second phase using the flow imbalance ideas from \citet{adelmanklabjan2012}. This marks the first application of these techniques in a stochastic setting. The algorithms can operate in two modes: (1) {\it Standalone mode}, to construct a value function approximation from scratch, and (2) {\it Add-on mode}, to refine an existing approximation. Our numerical experiments indicate that while NLIAlg and 2PIAlg in standalone mode are only feasible for small-scale problems, the heuristic version of 2PIAlg (H-2PIAlg) in add-on mode, using the Affine Approximation and exponential ridge basis functions, can handle extremely large instances that may cause benchmark network revenue management methods to run out of memory. In these scenarios, H-2PIAlg delivers substantially better policies and upper bounds than the Affine Approximation. Furthermore, H-2PIAlg achieves higher average revenues in policy simulations compared to network revenue management benchmarks in instances with limited capacity.

\end{abstract}


\maketitle

\section{Introduction}

Dynamic programming \citep{bellman} is an optimization tool \textcolor{black}{that is typically} used to find the optimal policy for a controlled stochastic process. Usually, it involves assigning a value to each state of the process through the so-called value function. Under certain conditions, both the value function and the optimal policy can be retrieved by solving a linear program \textcolor{black}{(LP) with one variable for each state in the state space and one constraint for each feasible state-action pair. In applications with high-dimensional state spaces, solving this LP may become computationally intractable for two reasons. First, there might be too many variables.} Second, the number of constraints may become impractically large. This challenge calls for an extension of traditional dynamic programming methods, leading to the emergence of Approximate Dynamic Programming (ADP).

To address dimensionality challenges via ADP, value function approximations seek to capture the underlying structure of the optimal value function without explicitly enumerating values for each state. Frequently, a linear combination of basis functions is chosen to approximate the true value function. This technique effectively reduces the number of parameters to be estimated (representing variables in the LP) by focusing on characterizing the basis functions and their weights within the proposed approximation.

The estimation of these new parameters can be addressed by plugging the value function approximation into the LP  formulation \citep{fariasroy2003}\textcolor{black}{; we call this problem the {\it approximate} LP}. However, a lingering issue remains: the number of constraints may be too large for any solver to handle. This challenge is often addressed by using row generation, an iterative optimization technique employed in mathematical programming. The algorithm starts with a relaxed version of the \textcolor{black}{approximate LP}, called the {\it master program}, that only considers a subset of the constraints. In the context of value function approximation, this means that constraints are imposed only for a subset of feasible state-action pairs. At each step, new rows (i.e., constraints corresponding to some state-action pairs) are generated by fixing the parameters of the approximation and finding state-action pairs that induce violated constraints. These newly generated rows are then added to the master program, with the objective of progressively refining the estimate of the approximation parameters over time. Once the row generation algorithm concludes, it provides a value function approximation.


While this procedure has been widely employed in ADP (see, for instance, \cite{adelman2007,topaloglu}), the set and number of basis functions in the value function approximation are typically predefined. Both are crucial for the quality and tractability of the approximation. For instance, a larger number of basis functions  \( K \) may provide a better approximation due to increased degrees of freedom, but it also introduces more parameters, potentially leading to dimensionality issues. Thus, selecting the appropriate values for
$K$ requires balancing the quality of the approximation with the tractability of the problem. 

The purpose of this paper is to propose an algorithm that iteratively adds basis functions from a given set by solving an optimization problem. This algorithm allows the user to either stop the algorithm based on her satisfaction with the current approximation or to further improve it at an additional computational cost. In particular, the suggested Nonlinear Incremental Algorithm (NLIAlg) 
\textcolor{black}{iteratively follows these two steps: (I) estimating the weights and parameters of the current basis functions using a row generation algorithm, and (II) increasing $K$ until the improved approximation meets a user-specified criterion.} \textcolor{black}{For large instances, we propose to estimate the parameters of the approximation in two different phases within Step (I). We first transform the master problem into an {LP}  by fixing the parameters that define the basis functions, which would otherwise make the problem nonlinear. 
In Phase (i), we follow the {\it flow imbalance} ideas introduced in \citet{adelmanklabjan2012} to estimate those parameters. Given these fixed parameters, the other parameters are then optimized in Phase (ii). This constitutes the Two-Phase Algorithm (2PIAlg).} 

\color{black}
The proposed approaches can be used to either (1) find a suitable value function approximation from scratch in a {\it standalone mode} (when initialized with an empty set of initial basis functions), or (2) enhance a given value function approximation in an {\it add-on mode}. 
Despite the computational improvement of \textcolor{black}{the two-phase-estimation approach}, our numerical results highlight that \textcolor{black}{the standalone mode} is practical only for toy problems. This is why for problems of realistic size, we will \textcolor{black}{use the add-on mode with an heuristic version of 2PIAlg}.

Summarizing, the main contributions of this paper are the following. 

\begin{itemize}
\item[(a)]\textit{We develop the first algorithm in the NRM context that iteratively refines the approximation by sequentially optimizing and adding new basis functions.} These algorithms endow the user with control over the approximation process. Unlike traditional methods where number and shape of the basis functions are predefined, our approach allows users to either stop the algorithm if the current approximation meets their criteria or continue refining it further by increasing the number of basis functions \( K \). This flexibility balances quality and computational cost, addressing the critical issue of the selection of \( K \).


    \item[(b)]   \textit{We extend for the first time the flow-balance gap ideas of \cite{adelmanklabjan2012} to a stochastic context.} Originally developed for deterministic settings, we adapt this approach to the NRM problem. While \citet{pakimanetal2019} and \citet{bhat} apply sampling-based approaches to generate basis functions, our method leverages mathematical programming to construct basis functions that maximize flow imbalances in the NRM problem. Under certain conditions we prove that if the family of basis functions has enough fidelity to perfectly fit the optimal value function, and there is still a gap between the optimal objectives of the LP and the ALP problem, then one can find a new basis function that improves the current solution. 


   \item[(c)] \textit{We carry out a performance analysis against strong NRM benchmarks.} Our computational study compares the heuristic version of 2PIAlg (H-2PIAlg) in add-on mode against well-established methods in the NRM literature, namely AA, SPLA, and NSEP, using two sets of benchmark instances. Our numerical results demonstrate that H-2PIAlg significantly outperforms the policies and upper bounds provided by AA. Additionally, it achieves superior policies and upper bounds compared to SPLA and NSEP in scenarios where capacity is scarce, although these improvements come at a greater computational cost. However, thanks to its iterative nature, H-2PIAlg is able to provide solutions for instances where SPLA and NSEP run out of memory.

\end{itemize}

The structure of the paper is as follows. Section \ref{sec:literature} reviews the relevant literature. In Section \ref{sec:approxNRM}, we introduce the dynamic programming formulation for the NRM problem and present the ideas behind our iterative approximation approach. We also introduce the concept of flow imbalance and discuss how it can be implemented. Section \ref{sec:NLIAlg} details the NLIAlg, explaining how row generation is employed to implement the approximation and discussing criteria for deciding when to stop adding new basis functions. Section \ref{sec:2phase} explains how to modify NLIAlg to yield the more tractable 2PIAlg. In Section \ref{sec:heuristics}, we discuss heuristic modifications for the algorithms, while in Section \ref{sec:basis_choice} we explore the choice of basis functions. Finally, Section \ref{sec:experiments} presents our numerical analysis, and Section \ref{sec:conclusion} concludes with remarks and potential extensions. All proofs can be found in the Appendix.

\color{black}

\section{Literature Survey \label{sec:literature}}
\textcolor{black}{In this section, we point at overviews on NRM and ADP before we summarize 
particular classes of basis functions suggested when using approximate linear programming in Section \ref{sec:lit:basis}. In Section \ref{sec:lit:methods}, we then provide an overview of the related literature applying different methods within approximate linear programming. Given the vast literature in this area, we focus on papers applying such methods for NRM problems and highlight the novelty of our contributions.}

Chapter 3 of \citet{talluriryzin2004book} provides an introduction to NRM problems, formulating \textcolor{black}{the decision problem as a Markov decision process with} states $\vx$ representing the remaining number of seats on all legs.  \textcolor{black}{In this context, the curse of dimensionality emerges due to exponential growth of the state space in the number of legs, prompting the application of ADP. }

\textcolor{black}{As mentioned above, ADP provides an array of methods to approximate Markov decision processes in the face of the curse of dimensionality. A summary of simulation-based approaches is given in \citet{powell2007}. On the other hand, approximate linear programming approximates the value function $v_t(\cdot)$ of a Markov decision process as a linear combination of basis functions and integrates this approximation into the linear programming formulation \citep{schweitzerseidmann1985,fariasroy2003,adelman2007}.}

\subsection{\textcolor{black}{Basis Functions in Approximate Linear Programming}\label{sec:lit:basis}}


In the context of NRM, \cite{adelman2007} was the first to suggest an additively separable  value function approximation within the linear programming formulation of the corresponding Markov decision process. By approximating the value function with an affine function with time-dependent slopes, \citet{adelman2007} demonstrated the efficiency of column generation in solving the approximated LP.
 
This affine approximation implies two major assumptions: (i) the different resources are separable, i.e., the marginal value of one unit of resource $i$ is independent of the number of units available for resources $j\neq i$; and (ii) the marginal value of one unit of resource $i$ is constant and hence independent of the number of units available for this resource. 

To remove Assumption (ii), different basis functions have been suggested. In the general separable framework, the values of the value function in different states can be directly viewed as variables. This broader approximation is often referred to as the Separable Piecewise Linear Approximation (SPLA). \citet{fariasroy2007} use constraint sampling to analyze this case, \citet{zhangvossen2015} develop a reduction and \citet{kunnumkaltalluri2014} prove the equivalence between this approximation and the Lagrangian relaxation introduced by
\citet{topaloglu2009}.

Assumption (i), separability, is pervasive in the literature. However, in the context of NRM non-separable approximations have been explored by \citet{zhang2011}, who combines the solutions of decomposed single-leg revenue management problems using a minimum-operator, leading to a non-separable approximation. Additionally, \citet{Simon} propose basis functions explicitly dependent on more than one resource, suggesting reductions for non-separable approximations.

In a different setting, \citet{guestrinkoller2002} consider non-separable basis functions depending on a small number of variables each.  \citet{linetal2019} advocate for constraint-violation learning for non-separable approximations for inventory control and energy storage. In contrast to sampling-based approaches such as \citet{pakimanetal2019} and \cite{bhat}, we generate basis functions analytically using row generation. Nevertheless, both works share with ours the use of basis functions with universal approximability.

Regarding our chosen approximation structure, the ideas presented in \citet{adelmanklabjan2007} align closely with our work.  They suggest using ridge functions to capture interactions between different components of the state variable. However, unlike our work, they apply this approximation in a deterministic setting without improving a given approximation.

\subsection{\label{sec:lit:methods}\textcolor{black}{Solution Methods for Approximate Linear Programming }}

\textcolor{black}{For any choice of basis functions, the use of a linear combination of them within the linear programming formulation typically yields a linear problem with few variables and many constraints.} The resulting problem can then be solved by row or column generation \citep{trickzin1997,adelman2004,adelman2007,zhangadelman2009,kunnumkaltalluri2014}. Sometimes, the resulting problem can also be reduced analytically to allow for a computationally more efficient solution; see, e.g. \citet{fariasroy2007}, \citet{topaloglu},  \citet{zhangvossen2015}, \citet{kunnumkaltalluri2015b} and \citet{Simon}. Although unconventional in traditional NRM, alternative approaches to solving the approximated LP include constraint sampling \citep{fariasroy2004} and constraint violation learning \citep{linetal2019}. 

In this paper, we suggest a row generation method that iteratively adds basis functions to a given approximation (which can also be an empty set of basis functions). Methodologically, our algorithm mirrors early ideas from \citet{adelman2007}, but for a non-linear setting and deciding on the number of basis functions to be added.

\section{\textcolor{black}{Iterative Approximation Refinement in ADP for NRM} \label{sec:approxNRM}}

\textcolor{black}{We begin by presenting the dynamic programming formulation of the NRM problem in Section \ref{sec:NRM}. Section \ref{sec:NLA} introduces the nonlinear approximation we propose. Estimating the parameters for this approximation involves solving a large-dimensional nonlinear program, which we outline in Section \ref{sec:ALP}. In this section, we also discuss how the approximation can be iteratively refined by increasing the number of basis functions. Finally, in Section \ref{sec:refining}, we present the dual formulation of the linearized counterpart of such a large-dimensional program, which allows us to introduce the flow-balance constraints. }

\subsection{The NRM Problem\label{sec:NRM}}

We consider a finite discrete-time horizon with time units $t \in \{1,\dots \tau\}$ to sell products $j \in \{1,\dots J\}$. At each time $t$, at most one customer arrives and requests a particular product $j$ with probability $p_{t,j}$.  No customer arrives with probability $1-\sum_{j}p_{t,j}$. The sale of product $j$ at time $t$ generates revenue $f_{t,j}$ and reduces the number of units of available resources $i\in \{1,\dots I\}$ by $a_{ij}\in\{0,1\}$.  The values of $a_{ij}$ are summarized in the consumption matrix $A \in \{0,1\}^{I \times J}$. At the beginning of the horizon, i.e. when $t=1$, the seller has $c_i$ units of resource $i$, that are referred to as the capacities of the resources and are denoted jointly as $\vec{c} = (c_1,\dots c_I)$. The remaining capacity of resource $i$ at any given point in time is denoted by $x_i$ with $\vec{x} = (x_1,\dots x_I)$.   We assume that at the end of the selling period, i.e. when $t=\tau+1$, remaining units of resources are worthless.
	
    At every period $t$, the seller must determine whether a customer request for product $j$ is accepted  ($u_j=1$) or rejected ($u_j=0$) with  $\vec{u}=(u_1,\dots,u_J)$. Given a customer request for product $j$, this binary variable allows us to write the revenue obtained as $f_{j}u_j$ and  the remaining resources as $\vx_t -u_j\vec a_j$, where $\vec a_j$ denotes the $j$-th column of matrix $A$. At every state $\vx$ we assume that we cannot sell more resources than we have, thus our set of feasible actions is
$$\mathcal U_{\vec{x}} := \left\{ \vec{u} \in \{0,1\}^{{{J}}} \mid u_j \vec{a}_{j} \leq \vec{x} \ \ \forall j\right\}.$$
The set of feasible remaining resources is given by
$$\mathcal X_{t} :=  \begin{cases}
\{(c_1,\dots,c_I)\} & t = 1 \\
\{0,\dots c_1\}\times\dots\times \{0,\dots c_I\}, & t \geq 2. 
\end{cases}. $$
We denote the hypercube convex hull of $\mathcal X_{t}$ by $conv(\mathcal X_{t} )$. Furthermore, we denote the set of feasible state-action pairs $(\vec x, \vec u)$ at time $0\le t\le\tau$  by 
	$$
	(\vec x, \vec u)\in  \mathcal{F}_t = \{ (\vx,\vu):  \vx\in\mathcal X_t , \vu\in \mathcal U_{\vx} \}=\{ (\vx,\vu):  x_i \in \{0,1,.., c_i\} ,  \quad u_j\in \{0, 1\}, \quad 
	u_ja_{i,j}\le x_i \ \forall i,j \}.
	$$

The objective is to maximize the total expected revenue that can be gained given initial capacity $\vec c$ over the selling horizon $t=1,\dots,\tau$. In this paper, we do not consider cancellations. Then, the maximum expected revenue that can be obtained over the selling horizon $t = 1,\dots,\tau$ given remaining resources $\vec x$ can be computed using the optimality equations
\begin{equation}\label{OE}
v_t^*(\vec x)=\max_{\vec u\in\mathcal U_{\vec x}}\left\{\sum_{j=1}^Jp_{t,j}\left[f_ju_j+v_{t+1}^*(\vec x-\vec a_ju_j)\right]+(1-\sum_{j=1}^Jp_{t,j})v_{t+1}^*(\vec x)\right\}, \quad \forall t=1,\dots,\tau-1,\vec x\in\mathcal X_t
\end{equation}
and $v^*_{\tau}(\vec x)=\max_{\vec u\in\mathcal U_{\vec x}}\left\{\sum_{j=1}^Jp_{\tau,j}f_ju_j\right\}$ for all $\vec x\in\mathcal X_{\tau}$.  Instead of tackling the  above optimality equations directly, we formulate the following LP 

\begin{samepage} 
\begin{flalign}
\min_{v_t(\vec x)} &\quad v_1(\vec c)& \nonumber\\ 
s.t.: &\quad v_t(\vec x) \geq \left\{\sum_{j=1}^Jp_{t,j}\left[f_ju_j+v_{t+1}(\vec x-\vec a_ju_j)\right]+(1-\sum_{j=1}^Jp_{t,j})v_{t+1}(\vec x)\right\} \label{LP}\tag{LP}\\
&\phantom{v_{\tau}(\vec x)\geq \sum_{j=1}^Jp_{t,j}f_ju_j} \hspace{4.2cm} \forall (\vec x,\vec u)\in\mathcal F_{t},t=1,\dots,\tau-1 \nonumber\\
&\quad v_{\tau}(\vec x)\geq \sum_{j=1}^Jp_{\tau,j}f_ju_j \hspace{4.2cm} \forall (\vec x,\vec u)\in\mathcal F_{\tau}. \nonumber
\end{flalign}
\end{samepage}
As pointed out in Proposition 1 of \cite{adelman2007}, for every $\vx$ any feasible solution $\hat v_t(\vx)$ to \eqref{LP} gives an upper bound of the optimal value function $v_t^*(\vx)$ solving \eqref{OE}. According to complementary slackness, an optimal solution to \eqref{LP} fits the optimal value function exactly for states $\vx$ with non-zero visit probability. In particular, the optimal solution to \eqref{LP} evaluated at period 1 and at full capacity, $\hat v_1^*(\vec c)$, coincides with the optimal value function evaluated at that state; i.e., $\hat v_1^*(\vec c)= v_1^*(\vec c)$.

However, this NRM problem is hardly tackled directly because it suffers from the well-known curse of dimensionality. Indeed, Problem \eqref{LP} has as many constraints as feasible state-action pairs and functions $v_t(\cdot)$ may lie on a highly dimensional space. As a consequence, this problem can only be solved exactly for very small instances. To solve larger instances, typically the value function is approximated and a row generation algorithm for \eqref{LP} is implemented or equivalent reductions are developed. 

\subsection{The Value Function Approximation\label{sec:NLA}}

We denote any given approximation of the value function by $\psi_t(\cdot)$, and suggest an extension that accommodates non-separable terms and nonlinearities. More specifically, for any given $\psi_t(\cdot)$ we approximate the value function $v_t(\cdot)$ by the Nonlinear (and \textcolor{black}{potentially} non-separable) Approximation (NLA)
\begin{align}
v_t^*(\vec x)\approx \psi_t(\vx) +\xi_t -\sum_{k=1}^K V_{t,k} \phi(\vx;\vbeta_{k}), \label{NLA} \tag{NLA}
\end{align}
with offset $\xi_{t}$, number of additional basis functions $K\in\mathbb N$, weights $V_{t,k}\in\mathbb R$ for $t=1,\dots,T, k=1,\dots K$ 
and nonlinear, {nonseparable} functions $\phi(\vx;\vbeta_{k})$ 
in $\vx$ with parameters $\beta_{i,k}\in\mathbb R$. 
Even though the same basis functions are used in every period $t$, the weights are time-dependent, allowing for very differently shaped approximations over time. While the value function may not be jointly concave, numerical approximations of the value functions \eqref{NLA} often exhibit a concave structure. To predominantly obtain non-negative weights with convex basis functions, we subtract the weighted sum of basis functions in \eqref{NLA}.



\textcolor{black}{The quality of the optimal approximation in the form \eqref{NLA} depends on the selection of the baseline approximation $\psi_t(\cdot)$, the class of basis functions $\phi(\cdot)$, and the number of basis functions $K$.} When $\psi_t(\vx)=0$ for all $t=1,\dots,T, \ \vx \in \mathcal{X}_t$, \eqref{NLA} generates a new approximation that does not fit the residuals of any baseline approximation \textcolor{black}{({\it standalone mode})}. Nonetheless, the fundamental concept behind \eqref{NLA} lies in its capacity to improve a given approximation $\psi_t(\cdot)\neq 0$ \textcolor{black}{({\it add-on mode})}. 
%
%

By substituting the exact value function with \eqref{NLA} on the right-hand side of \eqref{OE}, the proposed value function approximation yields the following policy for a given $\psi_t(\cdot)$

\begin{equation}\label{policy}
u_{t,j}(\vx)=\left\{ 
\begin{array}{ll}
1 & \mbox{ if }  f_j \geq  \displaystyle (\psi_{t+1}(\vx)-\psi_{t+1}(\vx-\vec a_j u_j))- \sum_{k=1}^K V_{t+1,k} \left(\phi(\vx;\vbeta_k) -\phi(\vx-\vec a_j,\vbeta_k)   \right)\\&\mbox{ and }  \vx \geq \vec a_j,\\
&\\
0 & \mbox{ otherwise.}\\
\end{array} \right.
\end{equation}
In the subsequent sections, we will refer to this policy as the policy based on (NLA). {To evaluate this policy, we first choose a class of basis functions $\phi(\cdot)$ and a baseline approximation $\psi_t(\cdot)$. If the current policy is not satisfactory, we can increase $K$ to improve the current approximation. An iterative refinement of \eqref{NLA} can be performed as follows:

\begin{itemize}
    \item[\bf Step (I):] Estimate $\xi_t, V_{t,k}, \vbeta_k$ for a given number of basis functions $K$
    \item[\bf Step (II):] Decide if $K$ should be increased. If yes, increase $K$ and return to Step (I).
\end{itemize}
The resulting value function approximation can then be used in \eqref{policy}.

\color{black}

\subsection{The Approximate LP  \label{sec:ALP}}

To implement Step (I), we leverage the nonlinear problem that arises from plugging \eqref{NLA} into the linear programming formulation \eqref{LP} for a fixed $K\in\mathbb N$

\begin{samepage}
\begin{align}
    \min_{\xi_t,V_{t,k},\vbeta_k}& \quad \psi_1(\vc) +\xi_1 -\sum_{k=1}^K V_{1,k} \phi(\vc;\vbeta_{k})& \nonumber\\
s.t.:&\quad  \psi_t(\vx) +\xi_t -\sum_{k=1}^K V_{t,k} \phi(\vx;\vbeta_{k}) \geq &&\label{AP} \tag{AP}\\
& \qquad \qquad  \sum_{j=1}^Jp_{t,j} \left[ f_ju_j+\psi_{t+1}(\vx-\vec a_ju_j) +\xi_{t+1}-\sum_{k=1}^K V_{t+1,k} \phi(\vx-\vec a_ju_j;\vbeta_{k}) \right] &&\nonumber \\
& \qquad \qquad+(1-\sum_{j=1}^Jp_{t,j})\left( \psi_{t+1}(\vx) +\xi_{t+1}-\sum_{k=1}^K V_{t+1,k} \phi(\vx;\vbeta_{k})\right)\ \forall  (\vec x,\vec u)\in\mathcal F_{t},t=1,\dots,\tau-1 \nonumber\\
&\qquad  \psi_{\tau}(\vx) +\xi_{\tau}-\sum_{k=1}^K V_{\tau,k} \phi(\vx;\vbeta_{k}) \geq \sum_{j=1}^Jp_{\tau,j} f_ju_j \qquad\qquad\ \forall (\vec x,\vec u)\in\mathcal F_{\tau}. & &\nonumber
\end{align}
\end{samepage}

The solution of \eqref{AP} provides the parameters \textcolor{black}{for Step (I)}. As mentioned in \cite{adelman2007}, the approximate program \eqref{AP} finds the lowest upper bound on $v_1(\vec c)$ of the form \eqref{NLA}. This means that choosing basis functions that are dense on the space where the value function lies may yield an approximation with the potential to fit $v_1(\vec c)$ arbitrarily close. Since an optimal solution to \eqref{LP} also fits the optimal value function exactly for states with non-zero visit probability, \eqref{AP} may also fit the optimal value function exactly at such states under a strong approximation. In sum, the ability of the proposed approximation \eqref{NLA} to fit the value function seems to be crucially dependent on the choice of the basis functions $\phi(\cdot)$. \textcolor{black}{In addition, the shape of $\phi(\cdot)$ in $\vbeta_k$ also affects the computational burden of solving \eqref{AP}. More specifically, nonlinear basis functions in $\vbeta_k$ lead to a nonlinear problem \eqref{AP}.} We provide guidance on the choice of basis functions in Section \ref{sec:basis_choice}.

\textcolor{black}{A straightforward implementation of Step (II) is to add variables $V_{t,K+1}$ and $\vbeta_{K+1}$ to the {\it nonlinear} problem \eqref{AP} in every iteration. By increasing the degrees of freedom with every set of newly added variables, this approach might have the potential to refine the current approximation with increasing $K$.}

\subsection{Using Flow-Balance Constraints to Refine the Value Function Approximation \label{sec:refining}}
\textcolor{black}{Since Problem \eqref{AP} might be nonlinear in $\vbeta_k$, $k=1,...,K$, \textcolor{black}{it might be} difficult to solve even for a moderate number of basis functions.  However,   for fixed values $\vbeta_k$, $k=1,...,K$, \eqref{AP} is a finite-dimensional LP with} 
variables $\xi_t$ and $V_{t,k}$. We call this resulting  problem {\it linearized} \eqref{AP}, and strong duality holds with the following dual formulation:

\begin{alignat}{2}
  \max_{\lambda_{t,(\vx,\vu)} \ge 0} \quad & \sum_{t=1}^{\tau-1} \sum_{(\vx,\vu) \in \mathcal{F}_t}  \lambda_{t,(\vx,\vu)}\left\{ \sum_{j=1}^J p_{j,t} \left[f_j u_j+\psi_{t+1}(\vx-\vec{a}_j u_j)\right] +\Big(1- \sum_{j=1}^J p_{j,t}\Big)\psi_{t+1}(\vx)-\psi_t(\vx) \right\} \nonumber \\
& \quad + \sum_{(\vx,\vu) \in \mathcal{F}_\tau} \lambda_{\tau,(\vx,\vu) } \Big( \sum_{j=1}^J p_{j,\tau} f_j u_j-\psi_{\tau}(\vx) \Big)  
  \label{D}\tag{D} \\
\text{s.t.} \quad & \sum_{(\vx,\vu) \in \mathcal{F}_t} \lambda_{t,(\vx,\vu)} 
= \begin{cases}
1 & \text{if } t=1 \\
\displaystyle \sum_{(\vx,\vu) \in \mathcal{F}_{t-1} } \lambda_{t-1,(\vx,\vu)} & \text{if } t=2,\dots,\tau
\end{cases} 
\nonumber\\
& \sum_{(\vx,\vu) \in \mathcal{F}_t} \lambda_{t,(\vx,\vu)}  \phi(\vx;\vbeta_k) 
= \begin{cases} 
\phi(\vc;\vbeta_k) & \forall k,t=1 \\
\displaystyle\sum_{(\vx,\vu) \in \mathcal{F}_{t-1}} \lambda_{t-1,(\vx,\vu)} \Big( \sum_{j=1}^J p_{j,t-1} \phi(\vx-\vec{a}_j u_j;\vbeta_k)  \\
 \quad + (1 - \sum_{j=1}^J p_{j,t-1}) \phi(\vx;\vbeta_k) \Big) & \forall k, t=2,\dots,\tau. 
\end{cases} \label{constraint_beta} \tag{D.FB}
\end{alignat}

In the problem above, $\lambda_{t,(\vec x,\vec u)}$ denotes the dual variables for the constraints  $(\vx,\vu)\in \mathcal{F}_{t}$. The so-called Flow-Balance constraint \eqref{constraint_beta} is associated with primal variables $V_{k,t}$. It is the only constraint that includes the parameters $\vbeta_1,\dots,\vbeta_K$ of the basis function. In any feasible solution to this dual, \eqref{constraint_beta} will be \textcolor{black}{satisfied} for the values $\vbeta_1,\dots,\vbeta_K$. If \eqref{constraint_beta} does not hold for some $\vbeta$, adding a basis function with these parameters could change the solution and refine the approximation. If there are no values $\vbeta$  within our selected set of basis functions \textcolor{black}{violating} \eqref{constraint_beta}, the \textcolor{black}{current approximation cannot be further improved using any additional $\phi(\cdot;\vbeta)$ from this selected set of basis functions .}


\textcolor{black}{This idea to improve value function approximations via flow-balance conditions was originally suggested by \cite{adelmanklabjan2012} in a deterministic context  with flow-balance equations similar to \eqref{constraint_beta}.} For deterministic MDPs and piecewise linear basis functions, \cite{adelmanklabjan2007} can even prove the convergence of such an algorithm. To mathematically formulate this idea in the context of NRM, we \textcolor{black}{subtract the right-hand side of \eqref{constraint_beta} on both sides of  \eqref{constraint_beta}  for all $t$ and view the resulting left-hand side} as functions in $\vbeta\in\mathbb R^I$. For given $t$, the result quantifies the flow-imbalance of \eqref{constraint_beta} given $\vbeta$:
\begin{flalign*}
\ell_1(\vbeta)= &\displaystyle \sum_{(\vx,\vu) \in \mathcal{F}_1} \phi(\vx;\vbeta) \lambda_{1,(\vx,\vu)}+ \phi(\vc;\vbeta)  & 
\end{flalign*}
\begin{flalign*}
\ell_t(\vbeta)=  &\displaystyle \sum_{(\vx,\vu) \in \mathcal{F}_t} \phi(\vx;\vbeta) \lambda_{t,(\vx,\vu)}\\
&\quad - \displaystyle \sum_{(\vx,\vu) \in \mathcal{F}_{t-1} }\left( \sum_{j=1}^J p_{j,t-1}\phi(\vx-\vec a_j u_j;\vbeta)+(1-\sum_{j=1}^J p_{j,t-1})\phi(\vx;\vbeta) \right) \lambda_{t-1,(\vx,\vu)} & \nonumber  \qquad  t=2,\dots,\tau\nonumber.
\end{flalign*}

Adding another basis function $\phi(\vx,\vbeta_{K+1})$ with parameters $\vbeta_{K+1}$ and new variables $V_{t,K+1}$ in the primal generates a new flow-balance constraint in the dual. Adding a basis function with parameters $\vbeta$ that result in no flow imbalance $\ell_t(\vbeta)=0$ for all $t=1,\dots,\tau$, hence will not alter the solution of the dual linear master problem. In contrast, adding a basis function violating constraint \eqref{constraint_beta} will improve the bound given by \eqref{NLA} under certain conditions.

\begin{proposition} \label{lemma:reducebound}
If there exists a unique optimal solution to \eqref{D}, adding a new basis function $\phi(\vx,\vbeta_{K+1})$ with $|\ell_t(\vbeta_{K+1})|>0$ for some $t\in \{1,...,\tau\}$ to the nonlinear approximation \eqref{NLA} reduces the upper bound given by {\it linearized} \eqref{AP}.
\end{proposition}

The following theorem demonstrates that if a finite set of basis functions can perfectly represent the optimal value function, and there remains a gap between the current approximation and the true optimal objective, then at least one of these basis functions must violate flow balance. 
\begin{theorem} \label{lemma:column_generation}
Assume there exists a finite number of basis functions, $K^* < \infty$, and parameters $\xi_t^*$, $V^*_{t,k} \in \mathbb{R}$, and $\vbeta^*_k \in \mathbb{R}^I$, for $k = 1, \dots, K^*$, such that the optimal value function can be expressed as
$$v_t^*(\vx) = \psi_t(\vx) + \xi^*_t - \sum_{k=1}^{K^*} V^*_{t,k} \phi(\vx; \vbeta^*_k) \quad \forall t=1, \dots, \tau, \ \vx.$$
Consider a given collection $\{\hat{\vbeta}_k\}_{k=1}^{\hat{K}}$, with $\hat{K} < \infty$. Let $(\hat{\lambda}, (\hat{V}, \hat{\xi}))$ be an optimal primal-dual solution pair to the corresponding approximation problem, such that $\hat{v}_1(\vc) > v_1^*(\vc)$, where the approximate value function $\hat{v}_t(\vx)$ is given by $$\hat{v}_t(\vx) = \psi_t(\vx) + \hat{\xi}_t - \sum_{k=1}^{\hat{K}} \hat{V}_{t,k} \phi(\vx; \hat{\vbeta}_k) \quad \forall t=1, \dots, \tau, \ \vx.$$ Define $\ell_t(\vbeta; \hat{\lambda})$ as the flow-balance functions parametrized by $\hat{\lambda}$. Then, there must exist some $k \in \{1, \dots, K^*\}$ and $t \in \{1, \dots, \tau\}$ such that $|\ell_t(\vbeta_k^*; \hat{\lambda})| > 0.$

\end{theorem}

\section{The Nonlinear Incremental Algorithm (NLIAlg) \label{sec:NLIAlg}}

We begin in Section \ref{sec:rowgen} discussing how we can use row generation in Step (I) to solve Problem (\ref{AP}) for a fixed number of basis functions. Based on the quality of the current approximation,  Section \ref{sec:increase} then discusses whether the number of basis functions $K$ should be increased in Step (II).

In a nutshell, NLIAlg involves the following two steps:\begin{itemize}
    \item[\bf Step (I):] Estimate $\xi_t, V_{t,k}, \vbeta_k$ for a given 
    $K$ solving \eqref{AP} {\it via row generation}.
    \item[\bf Step (II):] Decide if $K$ should be increased. If yes, increase $K$ and return to Step (I).
\end{itemize}
\color{black}

\subsection{\textcolor{black}{Solving \eqref{AP} for Given $K$ via Row Generation}\label{sec:rowgen}}

Due to the extensive number of constraints in Problem (\ref{AP}), solving it directly for realistically sized problems is impractical. Instead of addressing (\ref{AP}) with constraints for all $(\vec x, \vec u)\in  \mathcal{F}_t$ and $\vx\in\mathcal X_{\tau}$, we employ row generation. That is, we minimize (\ref{AP}) only for a subset of constraints $(\vec x, \vec u)\in \mathcal B_t^{\lambda}\subseteq  \mathcal{F}_t$ for all $t=1,\dots,\tau$.  Additional constraints are added sequentially until a predefined stopping criterion $\mathcal{C}_R$ is met.

Constraints for period $t$ are generated by identifying a state-action pair $(\vec x,\vec u)\in \mathcal{F}_t$ that violates a constraint in \eqref{AP}. We refer to the optimization problem that emerges when seeking these state-action pairs as {\it subproblem} $t$. For fixed $\vec \xi, \vec V, \mbox{and } \vbeta$, and for $t=1,\dots,\tau-1$, subproblem $t$ minimizes the difference between the left-hand side and the right-hand side of the constraints for that time,
\begin{flalign}
\min_{( \vec x,\vec u) \in\mathcal{F}_t} &\quad \medmath{\left\{\psi_t(\vx) +\xi_t -\sum_{k=1}^K V_{t,k} \phi(\vx;\vbeta_{k})   - \sum_{j=1}^Jp_{t,j} \left[ f_ju_j+\psi_{t+1}(\vx-\vec a_ju_j) +\xi_{t+1}-\sum_{k=1}^K V_{t+1,k} \phi(\vx-\vec a_ju_j;\vbeta_{k}) \right] \right.} &&\nonumber \\
& \qquad \left. -(1-\sum_{j=1}^Jp_{t,j})\left( \psi_{t+1}(\vx) +\xi_{t+1}-\sum_{k=1}^K V_{t+1,k} \phi(\vx;\vbeta_{k})\right) \right\}, \qquad \quad  t=1,...,\tau-1 \label{subproblem_t_primal}.  
\end{flalign}

If, for a given $t$, the objective value of the optimal solution to (\ref{subproblem_t_primal}) is positive, the current solution  {$\vec \xi, \vec V,\vbeta$} fulfills all constraints  for time $t$, not just the ones that were explicitly considered in the master problem. If the objective value is negative, the optimal solution   $(\vec x^*_t, \vec u^*_t)$  of the subproblem $t$, with $t=1,...,\tau-1$, is added to the subset of the constraints $\mathcal B_{t}^{\lambda}$ considered in the master problem. The subproblem corresponding to the last time period $\tau$ is
  \begin{flalign}
 \min_{( \vec x,\vec u) \in\mathcal{F}_\tau} & \quad \left\{ \psi_\tau(\vx)+\xi_{\tau}-\sum_{k} V_{\tau,k}  \phi(\vx;\vbeta_k){-} \sum_{j=1}^Jp_{\tau,j} f_ju_ j\right\} , \label{subproblem_Tau_primal}
 \end{flalign}

\noindent {and treated accordingly. } Given a stopping criterion $\mathcal{C}_R$ (e.g., stopping if the objective value of all subproblems is greater than or equal to 0) we can hence solve {(\ref{AP}) using Row Generation.} This algorithm iteratively solves the master problem for given sets of constraints $B^{\lambda}_{t}, t=1, \dots,\tau$, solves the subproblems for all $t=1,\dots,\tau$  and expands the sets $B^{\lambda}_{t}$ until the stopping criterion is met. \textcolor{black}{The pseudocode for the Row Generation Algorithm can be found in the Appendix.}



\subsection{Determining \textcolor{black}{Whether to Increase} the Number of Basis Functions $K$ \label{sec:increase}}

While there are a few exceptions (see, for instance, \citet{pakimanetal2019}), the literature frequently overlooks the question of finding the appropriate number of basis functions to approximate the value function. Leveraging our adaptable approximation design, $K$ can simply be increased until a sufficiently good approximation is reached according to a user-specified stopping criterion $\mathcal{C}_K$. Every time $K$ is increased, a new set of variables $V_{t,K}$ and $\beta_{i,K}$ for all $t=1,...,\tau$ and $i=1,...,I$ is added to the master problem.  \textcolor{black}{A more detailed pseudocode of this procedure is given in the Appendix.}

To assess whether the quality of the approximation is sufficient, the stopping  criterion $\mathcal{C}_K$ could consider the difference between average revenue obtained from a simulation of the current policy (\ref{policy}) and the approximated objective value (\ref{AP}), which provide a lower and upper bound on $v_1(\vc)$,  respectively. If that difference is smaller than some predefined value, the algorithm is terminated.


\section{A Two-Phase Approach \textcolor{black}{to Reduce the Dimensionality of Nonlinear \eqref{AP}} \label{sec:2phase}}

A successful implementation of NLIAlg hinges on the tractability of both the master problem and the subproblems. The master problem is a large-dimensional nonlinear program with many constraints, whose tractability is deeply influenced by the shape of the basis functions in $\vbeta$. 
\textcolor{black}{In this section, we leverage LP solvers' efficiency by fixing $\vbeta_k$ for $k=1,\dots,K$ in \eqref{AP}. This transforms the master problem into an LP, regardless of the shape of $\phi(\cdot)$ in $\vbeta$.} In addition, it reduces the dimensionality of \eqref{AP}. We then integrate the flow-balance idea discussed in Section \ref{sec:refining} to optimize the parameter values of the new basis function $ \vbeta_{K+1}$. Thus the estimation estimation of the parameters in \eqref{NLA} is computed in two phases:

\begin{itemize}
    \item[\bf Step (I)] Use the following two phases to find a solution to (AP):
    \begin{itemize}
        \item[\bf Phase (i)]Estimate $\vbeta_{K}$ by maximizing flow imbalances.
        \item[\bf Phase (ii)]Estimate $\xi_t, V_{t,k}$ for a given number of basis functions $K$ solving {\it linearized} \eqref{AP} via row generation for fixed $\vbeta_{1},\dots,\vbeta_{K}$.
    \end{itemize} 
    \item[\bf Step (II):]  Decide if $K$ should be increased. If yes, increase $K$ and go to Step (I).
\end{itemize}

Although the subproblems remain nonlinear integer optimization programs, their dimensionality is smaller and more manageable with appropriate choices of basis functions and baseline approximation. Section \ref{sec:basis_choice} provides  more details on this choice.


\color{black}
\textcolor{black}{Based on the rationale discussed in Section \ref{sec:refining}} we expect that adding a basis function with parameters $\vbeta_{K+1}$ maximizing $|\ell_t(\vbeta)|$ will result in a closer fit to the real value function. Instead of maximizing the flow imbalance for a specific time period, we suggest finding parameters that maximize the overall weighted flow imbalance across the entire time horizon:

\begin{equation}
\vbeta_{K+1}=\arg\max_{\vbeta} \sum_{t=1}^{\tau} \quad (T-t+1) |\ell_t(\vbeta)|.
\label{Wflowimbalance}
\end{equation}				
Because flow imbalances early on propagate to future periods, intuition suggests that they should be prioritized for closure. This is why we emphasize the importance of the first periods to provide a good value function approximation.  Once the flow-imbalances in early periods are removed, the focus moves towards later periods. \textcolor{black}{A more detailed pseudocode of 2PIAlg is given in the Appendix.}




\color{black}




\subsection{The Heuristic Two-Phase Algorithm (H-2PIAlg) \label{sec:heuristics}}

In this section we propose a Heuristic Two-Phase Incremental Algorithm (H-2PIAlg), which includes modifications to the algorithm that intend to improve tractability. 
\textcolor{black}{First, we propose heuristic approaches to generating rows and basis functions.}
 \color{black}
To generate a new basis function we do not need to find the optimal solution of the flow imbalance problem \eqref{Wflowimbalance}. Parameters $\vec\beta_{K+1}$ yielding an objective value larger than zero still cut off flow imbalances. Therefore, H-2PIAlg uses local solvers and  enforces time limits on \eqref{Wflowimbalance}.

Similarly, solutions yielding negative objective values in \eqref{subproblem_t_primal} or \eqref{subproblem_Tau_primal} constitute violated constraints in \eqref{AP}, and their inclusion in the master problem is likely to bring improvement. As a consequence, H-2PIAlg
also uses local solvers and imposes time limits when solving these subproblems. In addition, instead of solving one subproblem for each time period $t=1,\dots,\tau$, H-2PIAlg reuses state-action pairs found for a given $t$ in other time periods if this state-action pair leads to a negative objective value in the corresponding subproblem. \textcolor{black}{This might potentially avoid solving $\tau$ subproblems at each iteration of the row generation algorithm.}


%


\textcolor{black}{Second, we explain how to exploit the monotonicity of the NRM value function.}
\textcolor{black}{In NRM, the optimal value function is non-negative and decreases over time for a given \(\vec{x}\); i.e., \(v_t(\vec{x}) \geq v_{t+1}(\vec{x})\) for all \(t = 1, \dots, \tau-1\). Additionally, it is increasing in each component \(x_i\) when the resources \(x_s\) for all other legs \(s \neq i\) and time \(t\) are held fixed; i.e., \(v_t(\vec{x}^{+i}) \geq v_t(\vec{x})\) for all $t$, where \(\vec{x}^{+i} = (x_1, \dots, x_{i-1}, x_i + 1, x_{i+1}, \dots, x_I)\)}. If we required \eqref{NLA} to fullfil these properties, these would translate into the following constraints
\begin{flalign}
&0\le \psi_t(\vx) +\xi_t -\sum_{k=1}^K V_{t,k} \phi{_k}(\vx; \vbeta_k)\quad \forall  t=1,...,\tau,\vec x\in\mathcal X_{t} \label{u02}\\
&0\le(\theta_t-\theta_{t+1})+ \sum_{i=1}^I (W_{t,i}-W_{t+1,i})x_i -( \sum_{k=1}^K V_{t,k}-\sum_{k=1}^K V_{t+1,k} \phi_k(\vx: \vbeta_k))\quad  \forall  t=1,...,\tau-1,\vec x\in\mathcal X_{t}  \label{u01}\\
&0\leq W_{t,i}+\sum_{k}V_{t,k}(\phi_k(\vx;\vbeta_k)-\phi_k(\vx^{+i};\vbeta_k) ))  \quad \forall  t=1,...,\tau,i=1,\dots,I,\vec x\in \mathcal{X}_{t} \setminus\mathcal{X}_{t}^{c_i} \label{constraint_extra}
\end{flalign}

\noindent where $\vx^{+i}=(x_1,...,x_{i-1},x_{i}+1,x_{i+1},...,x_I)$ and  $\mathcal{X}_{t}^{c_i}=\{\vx: \ x_i=c_i \}$. Constraints \eqref{u02}-\eqref{u01} are redundant as they are given by constraints with $\vu=\vec 0$ in \eqref{AP}. However, adding constraints \eqref{constraint_extra}  additionally  enforces monotonicity in  $x_j$, thus we may consider our master problem to be (\ref{AP}) with feasible region restricted by constraints of type \eqref{constraint_extra}. We call these problems with additional constraints {\it tightened} \eqref{AP}. Abusing notation, we refer to this kind of tightened problem as 
(\ref{AP})+\eqref{constraint_extra}. Nonetheless, considering all the monotonicity constraints for all $i = 1,...,I$ and $t=1,...,\tau$ would significantly increase the dimension of this problem, and our numerical experience showed that they do not improve much the current approximation when enough basis functions and rows have already been generated so there is already enough information about the structure of the value function. Hence H-2PIAlg generates monotonicity
constraints \eqref{constraint_extra} only for $K=1$. We denote
by $\mathcal{B}_t^\mu$ for all $t$ the subsets of generated constraint indices.


To generate monotonicity constraints, the right-hand side of \eqref{constraint_extra} is minimized for a particular $i$ and $t$. This subproblem is a nonlinear integer problem, solved with local solvers and imposing time limits. H-2PIAlg also uses solutions obtained by solving subproblem $t$ and $i$ to generate rows at other periods of time and other legs. \textcolor{black}{A pseudocode can be found in the Appendix.} The master program of this modified Row Generation algorithm is \eqref{AP}+\eqref{constraint_extra}, thus when finding a new basis function 
with H-2PIAlg one needs to maximize flow imbalances taking into account dual variables associated with monotonicity constraints \eqref{constraint_extra}. \textcolor{black}{The dual formulation of the tightened LP giving rise to {\it adjusted} flow-imbalance functions $\ell_t^\mu(\cdot)$. These formulations can be found in the Appendix.}

\color{black}

}


\section{Choosing the Set of Basis Functions \label{sec:basis_choice}}

In this section, we explore the selection of basis functions. Section \ref{sec:basis_properties} outlines the properties we consider desirable for $\phi(\cdot)$. \textcolor{black}{ Section \ref{sec:hat_vs_exponential} introduces piecewise linear and exponential ridge basis functions and provides a numerical comparison.}

\subsection{Desirable Properties for Basis Functions \label{sec:basis_properties}}

Although the algorithms introduced \textcolor{black}{so far can be used for} any choice of \(\phi(\cdot)\), in this paper we recommend using basis functions that 
\begin{itemize}
    \item[(a)] Have the potential to fit the value function arbitrarily closely, 
    \item[(b)] Are smooth and/or tractable to solve in $\vx$ and $\vbeta$, and \item[(c)] Exhibit structural properties akin to the real optimal value function and/or good interpolation properties. 
\end{itemize}
The rationale behind (a) is to support our goal of providing an accurate approximation of the value function. Specifically, our goal is to select functions that ensure that parameters of our approximation can lead to a small difference between the value function and the approximation for large \(K\). \textcolor{black}{To state this objective more precisely, we want the following:} \(\forall \epsilon>0, \ \exists K\in\mathbb{N}, \ \xi_{t},\ V_{t,k}\), and \(\beta_{i,k}\in\mathbb{R}\) such that \(|\psi_t(\vx) +\xi_t -\sum_{k=1}^K V_{t,k} \phi(\vx;\vbeta_{k})-v_t(\vx)|<\epsilon\) \(\forall \vx\). To achieve this, it suffices to choose basis functions from a set of functions that is fundamental in the space of continuous functions \(C(\mathbb{R}^I)\), i.e., with a linear span that is dense in \(C(\mathbb{R}^I)\). Intuitively, \textcolor{black}{this should ensure a desirable} behavior of our two-phase heuristic (see Proposition \ref{lemma:column_generation}).

Property (b) aligns with previous discussions about how the choice of \(\phi(\cdot)\) impacts both the speed of row generation and the overall quality of the approximation. The subproblems solved to generate rows are integer programs in \(\vx\) and \(\vu\). As the number of basis functions \(K\) increases, the complexity of these subproblems also increases, potentially slowing down the algorithm. Therefore, it is essential to have basis functions that can be easily modeled and solved in \(\vx\). Similarly, the choice of baseline approximation \(\psi_t(\cdot)\) should provide a good starting point while having a manageable shape in \(\vx\). Finally, H-2PIAlg must solve a low-dimensional continuous problem to generate a \(\vbeta\) that violates the flow-balance constraints. Thus, having basis functions that are also easy to model and solve in \(\vbeta\) is crucial.

Imposing (c) can reduce the number of basis functions needed by the algorithm. Intuitively, the more state values a single basis function can closely fit, the fewer basis functions will be required to meet the user's accuracy threshold. Additionally, imposing (c) is beneficial if structural properties of the real value function are known. In our NRM problem, in particular, the real value function is known to be monotonically increasing and asymptotically concave in  \(\vx\) \citep{talluriryzin1998}.  By selecting convex and decreasing \(\phi(\cdot)\) in \(\vx\), we could easily impose the approximation to be concave and increasing by requiring \(V_{t,k}>0\) for all \(t\) and \(k\). This may allow our approximation to capture the true structure of the value function, potentially providing well-behaved policies even for a small value of \(K\).

\subsection{Piece-Wise Linear Hat versus Exponential Ridge Basis functions \label{sec:hat_vs_exponential}}
\textcolor{black}{A function $f: \mathcal{X} \to \mathbb{R}$ is called a ridge function if it has the form $f(\vx)= g(\vbeta^\top \vx)$, where $g$ is a function in $\mathbb{R}$ and $\vbeta$ is a fixed vector \citep{ridge}.} Ridge functions form a fundamental set in $C(\mathbb{R}^I)$ (see Theorem 4 in \cite{ridge}) and hence naturally fulfill (a). 
The following result shows that as long as $g$ is a convex function, the ridge function is convex, ensuring Property (c).
\begin{proposition}\label{lemma:ridge}
\textcolor{black}{Let  $\phi(\vx;\vbeta)=g(\vbeta^\top \vx)$  be a ridge function with direction $\vbeta$ and $g(\cdot)$ convex. Then, $\phi(\cdot;\vbeta)$ is convex.} 
\end{proposition}

To ensure stable estimation and robust numerical performance, it is often advantageous to constrain or penalize $\|\vbeta\|$. Theorem 5 in \cite{ridge} confirms that we can add this type of constraints while still preserving Property (a). \textcolor{black}{In the Appendix, we show that norm-constrained convex ridge basis functions give rise to DC optimization problems.}

\textcolor{black}{Inspired by \cite{adelmanklabjan2012}, one might consider using continuous piecewise linear hat functions in approximation \eqref{NLA}, but they fail to meet Conditions (b) and (c). Specifically, hat functions are nondifferentiable and encoded via MINLP, potentially making efficient row generation less tractable. As noted in \cite{adelmanklabjan2012}, the flow-balance problem \eqref{Wflowimbalance} is a hard nonlinear integer problem, complicating the generation of new basis functions. Furthermore, \citet{adelmanklabjan2007} discuss the challenges of enforcing convexity over the entire domain when using these functions. Lastly, since each hat function is composed of two linear segments, they may be able fit a very limited number of points closely. Given the stochastic nature of the NRM problem, one would expect to need almost as many functions as there are states in the state space.}

\textcolor{black}{The set of exponential ridge functions $\mathcal{E}:=\{e^{-\sum_{i=1}^I \beta_{i,k}x_i} : \ \vbeta_k\in\mathbb{R}^I, \|\vbeta_k\|\le 1, k=1,\dots,K \}$ is another set of fundamental functions. The corresponding basis functions}
\begin{equation}
\phi(\vx;\vbeta_{k})=e^{-\sum_{i=1}^I \beta_{i,k}x_i}\text{ for all }k=1,\dots,K.
\label{exponential_basis}
\end{equation} 
\textcolor{black}{are differentiable, monotone decreasing, and convex. If all weights in \eqref{NLA} are chosen positive, the value function approximation will hence be monotone increasing and concave.} 
\textcolor{black}{Furthermore, basis \textcolor{black}{functions} of the form \eqref{exponential_basis} are smooth, easy to code and solve. The flexibility of these functions may also allow to closely fit a larger number of points at once, possibly reducing the number of basis functions needed. Summarizing, exponential ridge basis functions fulfill Conditions (a), (b) and (c). }

For numerical purposes, we restrict ourselves to  norm-constrained exponential ridge-basis functions with 
\begin{align}
&||\vbeta_k||=\sum_i c_i |\beta_i| = 1,\label{Wnorm1} 
\end{align}
where $c_i$ is the capacity of leg $i$.
The following proposition shows that adding the norm constraint \eqref{Wnorm1} allows us to bound the basis function away from zero and away from large values. This might avoid numerical issues in practice, further ensuring Property (b).
\begin{proposition} \label{prop:bounds}
	If $\vbeta_k$ is constrained by \eqref{Wnorm1}, then $\phi(\vx;\vbeta_k) \in [\frac{1}{e},e]$
\end{proposition}

Instead of solving (\ref{AP}), the Row Generation Algorithm can also be used to solve (\ref{AP})+\eqref{Wnorm1} by adding \eqref{Wnorm1} to the master problem. In addition, the resulting minimum objective value of the tightened problem (\ref{AP})+\eqref{Wnorm1} provides an upper bound on $v_1(\vc)$.


\textcolor{black}{To numerically compare the performance of H-2PIAlg with hat and exponential ridge basis functions, we use a toy example } We consider the standalone version of the algorithm; i.e., \(\psi_t(\vx)=0\). Although NLIAlg provides similar results, it is less tractable, so its results are relegated to the Appendix.
\color{black}
 We consider a small network with two flights, Flight 1 connecting cities $A$ and $B$ and Flight 2 connecting $B$ and $C$. Customers can travel from $A$ to $C$ by taking both flights. Hence, there are three itineraries $\{AB, BC, AC\}$ each offered at a low and high fare, generating six products in total. There are 10 periods of time and the demand process is assumed to be stationary, i.e. $p_{t,j}=p_j$ for all $j$. Fares and arrival probabilities for all products are specified in \textcolor{black}{the Appendix.}

%


After generating 60 time-dependent hat basis functions (6 per period of time), H-2PIAlg struggled to generate new rows. H-2PIAlg ran for more than 12 hours without converging, whereas the algorithm converged in 28 seconds after generating only 10 exponential ridge basis functions. \textcolor{black}{More specifically, H-2PIAlg achieved an average revenue of 397.2 when simulating policy (\ref{policy}), which is very close to the optimal value function \(v_1^*(\vc) \approx 397.507\), computed using value iteration.}  As a result, we selected exponential ridge basis functions for subsequent numerical experiments. \textcolor{black}{More details on the performance of H-2PIAlg on this toy example can be found in the Appendix. There we also provide graphical illustration showing that the more exponential ridge basis functions H-2PIAlg adds, the better the approximated value function fits the optimal value function.}

%

 \section{Computational study \label{sec:experiments}}

The objective of this section is to assess the performance of H-2PIAlg against strong benchmark methods in NRM. Section \ref{sec:setup} details the computational specifications for H-2PIAlg, while Section \ref{sec:methods} describes the methods we compare. Section \ref{sec:datasets} provides an overview of the datasets used in our experiments, and Section \ref{sec:performance} examines the performance of H-2PIAlg. 

\subsection{Computational Specifications for H-2PIAlg \label{sec:setup}}

To implement H-2PIAlg we need to specify \textcolor{black}{the following six inputs:}

\paragraph{\textcolor{black}{\it 1. The baseline approximation $\psi_t(\cdot)$:}} When choosing the baseline approximation, recall that H-2PIAlg has two modes: the {\it standalone mode} and the {\it add-on mode}.
To run the add-on mode, we first have to specify a baseline approximation and determine its parameters. We choose the Affine Approximation 
\begin{equation} \tag{AA} \label{AA}
\psi_t(\vec x)=\theta_t+\sum_i^I W_{t,i} x_i
\end{equation} 
as proposed in \cite{adelman2007}. This baseline approximation adds little computational burden (see \cite{topaloglu}). In addition, 
 this choice of $\psi_t(\vx)$ is linear in $\vx$, so this basis function fulfills Property (b) (see Section \ref{sec:basis_choice}), allowing an efficient generation of rows. 
 Additionally, it provides a reasonable approximation of the real value function, addressing Property (c).
 Since the standalone version is not tractable for larger instances, we only report the results for the add-on mode here. Appendix \ref{sec:compare_baselines} compares the convergence of the standalone and add-on modes for smaller hub-and-spoke problem instances.

\paragraph{\textcolor{black}{\it 2. The class of basis functions $\phi(\cdot)$:}} As discussed in Section \ref{sec:basis_choice}, we select $\phi(\cdot)$ from the set $\mathcal{E}$ of exponential ridge basis functions.  \textcolor{black}{This choice allows H-2PIAlg to exploit nonlinearities and resource interactions.}

 \paragraph{\it 3. The initial value of $\vbeta_{1}$:} To fulfill constraint \eqref{Wnorm1}, we let $\beta_{i,1}=1/(c_iI)$ for all $i$.

 \paragraph{\it 4.The initial subset of constraints $\mathcal{B}=\left\{\mathcal{B}_t^\lambda\right\}_{t=1,\ldots \tau}$:}The initial subsets of constraints are $\mathcal{B}_1^\lambda =\{(\vc,\vzero) \}$, $\mathcal{B}_t^\lambda =\{(\vc,\vzero),(\vzero,\vzero) \}$ for all $t=2,\dots,\tau$. 
    
\paragraph{\it 5. The stopping criterion $\mathcal{C}_R$ for the Row Generation Algorithm:}Our stopping criterion $\mathcal{C}_R$ is based on a bound for the optimality gap. To specify this, let $Z_\phi$ denote the optimal objective value of \textcolor{black}{{\it linearized} \eqref{AP}. Similarly, let $Z_{\mathcal{B}}$  denote the minimum objective value of the tightened problem \eqref{AP}+\eqref{constraint_extra} considering the reduced set of constraints $\mathcal{B}$.} Let $\pi_{\mathcal{B},t}^*$ be the minus the objective value of \eqref{subproblem_t_primal} or \eqref{subproblem_Tau_primal} when evaluated at their global optima.
Using Proposition 3 in \cite{adelman2007}, we obtain the bound ${Z_\phi} \leq Z_{\mathcal{B}}+ \sum_t \pi_{\mathcal{B},t}^*:=Z^*_{\phi}$. For a given value of $\Omega^{gap}>0$, stopping the Row Generation Algorithm when  $ \sum_t \pi_{\mathcal{B},t}^*\leq \Omega^{gap} Z_{\mathcal{B}} $
ensures that the estimate $Z^*_{\phi}$ is less than $\Omega^{gap}\%$ larger than the true value ${Z_\phi}$.

We have defined $\pi_{\mathcal{B},t}^*$ as minus the (globally) optimal solutions of subproblems \eqref{subproblem_t_primal} and \eqref{subproblem_Tau_primal}. If only locally optimal reduced costs $\hat{\pi}_{\mathcal B,t}$  are available, an alternative stopping criterion $\mathcal{C}_R$ is to stop generating additional rows when $ \sum_t \hat{\pi}_{\mathcal{B},t} \leq \Omega^{gap} Z_{\mathcal{B}}$. Since $\hat{\pi}_{\mathcal{B},t} \leq \pi^*_{\mathcal B,t}$, we expect an algorithm with stopping criterion $\mathcal{C}_R$ based on $\hat{\pi}_{\mathcal{B},t}$ to run for shorter times.
However, note that $\hat{Z}_{\phi}:=Z_\mathcal{B}+\sum_t \hat{\pi}_{\mathcal{B},t}$ is not an upper bound for ${Z}_\phi$. 

If the stopping criterion should be based on an upper bound $\overline{Z}_\phi\geq Z^*_\phi$, the dual bounds $\overline{\pi}_{\mathcal{B},t} \geq \pi_{\mathcal{B},t}^*$  given by {\tt Knitro} when solving subproblem $t$ can be used. Since H-2PIAlg might not solve all the subproblems for each period of time when generating rows (see Section \ref{sec:heuristics}), one would have to enforce the solution of subproblems for all time periods. However, such a stopping criterion would  make H-2PIAlg intractable for large instances. Therefore, we calculate the upper bound only after the algorithm terminates according to a stopping criterion based on $\hat{Z}_{\phi}$.

\paragraph{\it 6. The stopping criterion $\mathcal{C}_K$ for adding new basis functions:} Recall that every solution of the master problem with the reduced set of constraints $\mathcal{B}$ provides a policy as specified in \eqref{policy}. For given values $\Omega^{policy}$ and $\Omega^{p-gap}$, we simulate instances of this policy until the relative standard error $S_e$ of the average policy revenue $\overline{R}$ is smaller than a given value $\Omega^{policy}$; i.e., until $S_e/\overline{R} \leq
\Omega^{policy}$. We then stop generating additional basis functions and terminate the incremental algorithms if this average policy revenue is within $\Omega^{p-gap}\%$ from $\hat{Z}_\phi$; i.e., we stop adding basis functions if $\left |1-\overline{R}/\hat{Z}_{\phi} \right |<\Omega^{p-gap}$. Since $Z_{\phi}<\hat{Z}_{\phi}$, this criterion provides an optimality gap guarantee on the policy.  In practice the user may decide not to wait until $\mathcal{C}_K$ is fulfilled, but to stop earlier if a reasonably-good improvement (either in terms of estimated bound or simulated policy) is achieved. In our numerical experiments, we terminate the H-2PIAlg when either $\mathcal{C}_K$ is met or when a CPU time of 168h is reached. More computational details can be found in Appendix \ref{appendix_comp}


\subsection{Methods Under Comparison \label{sec:methods}}

The numerical experiments compare four methods. The {\it Affine Approximation (AA)} is used to evaluate the extent to which the proposed algorithm enhances the quality of an initial approximation. The {\it Separable Piecewise Linear Approximation (SPLA)} serves as a strong benchmark, as it is known for producing small optimality gaps. It can be computed via a linear program \citep{zhangvossen2015}. The {\it Non-Separable Approximation} (NSEP)\footnote{To run NSEP, we used the codes available in the Github repository \url{https://github.com/slaume/ Reductions-of-Non-Separable-ALPs-for-NRM}. We chose what seemed the “best” performer in the hub-and-spoke instances of \cite{Simon}. That is to say, in the Github code we set {\tt task = 2}, i.e., we determined the partition for which the network measure is largest, which can then be used for NSEP4(I/q,q,gl.).  As recommended by the authors, we then saved the partition set {\tt Icaln} and reset the rest of the environment before running the code with {\tt task=1}. } lets us evaluate how much H-2PIAlg can improve upon an approximation with nonlinearities and resource interactions. Finally, {\it H-2PIAlg} with the specifications discussed in Section \ref{sec:setup}.

\subsection{Instances, Code, and Hardware \label{sec:datasets}}
We consider two types of network instances: the experimental setup specified in \cite{adelman2007} and the bus networks considered in \cite{Simon}. The Hub-and-Spoke (H\&S) instances of \cite{adelman2007} include up to 20 non-hub locations and 2 fares, resulting in up to $40$ single-leg itineraries and up to $380$ two-leg itineraries. Time horizons of up to 1,000 periods are considered. The capacities, \textcolor{black}{specified in Section B of the  Online Appendix}, are chosen so that the load factor is roughly 1.6.

In the bus line examples provided in \cite{Simon}, there are three types of instances: the Simple Bus Line (SBL), which consists of $I$ consecutive legs; the Consecutive Bus Lines (CBL), which combines multiple simple networks; and the Realistic Bus Line (RBL), a real-world example from a large European bus company. More details can be found in \cite{Simon}, and the data for these instances are available on {\it Github}.

\color{black}
All experiments were conducted on the Mercury high-performance computing cluster at The University of Chicago Booth School of Business, with 50 GB of memory and 1 core utilized. The code is available in GitHub.
\color{black}


\subsection{Numerical Results and Discussion \label{sec:performance}}

In this section, we compare the performance of the methods outlined in Section \ref{sec:methods} in terms of Upper Bounds \textcolor{black}{(UBs), Lower Bounds (LBs) corresponding to the average revenue when simulating policy \eqref{policy}, and computational times}. For H-2PIAlg, we simulate the policy and calculate its average revenue for each $K$. Although there is a tendency towards improvement with the increment of $K$, the policy does not necessarily improve monotonically in $K$. As a consequence, we report the highest average revenue found in that process. In other words, it might be the case that H-2PIAlg stopped at $K$, and there was a $K'<K$ such that $\overline{R}^{(K')}>\overline{R}^{(K)}$, where $\overline{R}^{(k)}$ is the average reward of the simulated policy for $k$ basis functions. So the LB corresponds to the {\it best policy found}, which is $\max_{k \in \{1,...,K\}} \overline{R}^{(k)}$. \textcolor{black}{Table \ref{table:SPLA_comparison} display UBs and LBs for the H\&S and bus instances. More details about CPU times and the number of basis functions added by H-2PIAlg can be found in Section C of the Appendix.}




\begin{table}[h!]																	
\centering																	
\begin{small}																	
\begin{tabular}{ @{} *{3}{S[table-format=3.0]} | *{4}{S[table-format=5.2]}  | *{4}{S[table-format=5.2]} @{}}														
\multicolumn{3}{c|}{}	&			\multicolumn{4}{c|}{ Upper bound (UB)}	&			\multicolumn{4}{c}{Policy (LB)}			\\
	&		&		&	{ SPLA}	& {NSEP} &{ AA}	&	{ H-2PIAlg}	&	{ SPLA}	&  {NSEP}	& { AA}	&	{ H-2PIAlg}	\\
\cline{4-11}	
{ $L$}	&	{ $\tau$}	&	 {$c$} & \multicolumn{8}{c}{\bf Hub-and-Spoke instances (H\&S)}\\
\hline\hline
{\multirow{6}{*}{2}}	&	20	&	3	&	821.71	&	 \BB 810.87	&	874.44	&	817.33	&	809.18	&	809.99	&	773.61	&	\BB 810.01	\\
	&	50	&	8	&	2269.7	&	\BB 2250.54	&	2832.01	&	2292.21	&	2244.20	&	\BB 2249.15	&	1701.84	&	2244.70	\\
	&	100	&	17	&	4781.13	&	\BB 4744.31	&	5623.28	&	4866.92	&	4738.27	&	\BB 4745.98{$^*$}	&	3642.22	&	4723.35	\\
	&	200	&	33	&	9707.19	&	$^\dagger$  	&	11269.31	&	9882.71	&	\BB 9645.95	&	$^\dagger$  	&	7137.15	&	9461.31	\\
	&	500	&	83	&	24581.25	&	$^\dagger$  	&	28083.08	&	25034.2	&	\BB 24469.78	&	$^\dagger$  	&	17712.68	&	23869.81	\\
	&	1000	&	165	&	49251.79	&	$^\dagger$  	&	56020.68	&	49785.31	&	\BB 49182.99	&	$^\dagger$  	&	35142.63	&	48147.09	\\
\hline														{\multirow{6}{*}{3}}	&	20	&	2	&	685.75	&	673.83	&	741.13	&	\BB 668.08	&	659.58	&	660.45	&	641.16	&	\BB 661.43	\\
	&	50	&	6	&	2063.02	&	\BB 2043.13	&	2634.23	&	2090.53	&	2016.68	&	\BB 2018.28	&	1561.28	&	2012.33	\\
	&	100	&	12	&	 4327.88	&	\BB 4296.01 	&	5291.44	&	4432.55	&	 4262.21	&	\BB 4266.14 	&	3219.88	&	4206.24	\\
	&	200	&	25	&	\BB 8977.32	&	$^\dagger$  	&	10569.67	&	9101.56	&	\BB 8870.66	&	$^\dagger$  	&	6718.57	&	8696.39	\\
	&	500	&	62	&	\BB 22802.57	&	$^\dagger$  	&	26500.14	&	23013.13	&	\BB 22628.00	&	$^\dagger$  	&	16733.62	&	22173.58	\\
	&	1000	&	124	&	\BB 45751.93	&	$^\dagger$  	&	52911.9	&	46123.38	&	\BB 45494.00	&	$^\dagger$  	&	33344	&	43680.25	\\
\hline													
{\multirow{6}{*}{5}}	&	20	&	2	&	755.71	&	\BB 737.56	&	824.34	&	750.01	&	711.5	&	713.49	&	687.57	&	\BB 713.62	\\
	&	50	&	4	&	1909.38	&	\BB 1881.68	&	2562.17	&	1966.49	&	1831.61	&	\BB 1837.25	&	1428.88	&	1816.37	\\
	&	100	&	8	&	4043.79	&	\BB 4008.09	&	5122.63	&	4183.00	&	3935.72	&	\BB 3941.74	&	3003.64	&	3802.08	\\
	&	200	&	17	&	\BB 8502.53	&	$^\dagger$  	&	10213.77	&	8716.87	&	\BB 8356.41	&	$^\dagger$  	&	6487.64	&	7910.35	\\
	&	500	&	41	&	\BB 21515.49	&	$^\dagger$  	&	25599.28	&	21961.94	&	\BB 21287.37	&	$^\dagger$  	&	15978	&	20510.82	\\
	&	1000	&	83	&	$^\dagger$	&	$^\dagger$  	&	51130.43	&	\BB 44258.58	&	$^\dagger$	&	$^\dagger$  	&	32310.06	&	\BB 40992.70	\\
\hline																					
{\multirow{6}{*}{10}}	&	20	&	1	&	624.95	&	$^\dagger$  	&	624.95	&	\BB 600.10	&	547.92	&	$^\dagger$  	&	547.64	&	\BB 550.60	\\
	&	50	&	2	&	\BB 1686.68	&	$^\dagger$  	&	2527.9	&	1885.22	&	\BB 1544.59	&	$^\dagger$  	&	1210.54	&	1490.57	\\
	&	100	&	5	&	\BB 4071.60	&	$^\dagger$  	&	5149.62	&	4368.73	&	\BB 3876.38	&	$^\dagger$  	&	3151.55	&	3640.82	\\
	&	200	&	9	&	\BB 8252.27	&	$^\dagger$  	&	10320.05	&	8741.37	&	\BB 7973.12	&	$^\dagger$  	&	6113.26	&	7338.47	\\
	&	500	&	23	&	\BB 21613.57	&	$^\dagger$  	&	25814.11	&	22598.56	&	\BB 21193.99	&	$^\dagger$  	&	16184.17	&	19559.56	\\
	&	1000	&	45	&	$^\dagger$	&	$^\dagger$  	&	51629.25	&	\BB 45096.76	&	$^\dagger$	&	$^\dagger$  	&	32070.7	&	\BB 39633.85	\\
\hline																					
{\multirow{6}{*}{20}}	&	20	&	1	&	734.14	&	$^\dagger$  	&	734.31	&	\BB 732.18	&	632.64	&		&	632.25	&	\BB 636.03	\\
	&	50	&	2	&	\BB 1391.22	&	$^\dagger$  	&	2255.3	&	1555.55	&	\BB 1178.99	&	$^\dagger$  	&	986.47	&	1158.22	\\
	&	100	&	2	&	\BB 3241.21	&	$^\dagger$  	&	5075.13	&	3646.09	&	\BB 2919.48	&	$^\dagger$ 	&	2182.02	&	2917.04	\\
	&	200	&	5	&	\BB 7875.53	&	$^\dagger$  	&	10147.54	&	8627.74	&	\BB 7436.05	&	$^\dagger$  	&	5883.32	&	6884.33	\\
	&	500	&	12	&	$^\dagger$	&	$^\dagger$  	&	25357.61	&	\BB 22111.36	&	$^\dagger$	&	$^\dagger$  	&	15266.23	&	\BB 18286.55	\\
	&	1000	&	24	&	$^\dagger$	&	$^\dagger$  	&	50722.76	&	\BB 44696.62	&	$^\dagger$	&	$^\dagger$  	&	31380.98	&	\BB 38004.44	\\
\hline
\multicolumn{11}{c}{\bf Bus instances}\\
\hline
\hline
\multicolumn{3}{r|}{Simple (SBL)}					&	30.041	&	\BB 28.47	&	30.86	&	29.95	&	27.62	&	27.62	&	26.93	&	\BB 27.68	\\
\multicolumn{3}{r|}{Consecutive (CBL)}					&	11.179	&	\BB 10.91	&	11.83	&	11.01	&	10.59	&	10.61	&	10.60	&	\BB 10.63	\\
\multicolumn{3}{r|}{Realistic (RBL)}					&	\BB 683.86	&	$^\dagger$  	&	700.85	&	699.27	&	\BB 682.44	&	$^\dagger$  	&	635.82	&	652.91	\\
\end{tabular}																	
\end{small}																	
\caption{\small Performance of SPLA, NSEP, AA, and  H-2PIAlg with AA baseline for the hub-and-spoke and bus networks.  Columns show upper bounds and average reward of simulated policy; best performer has been highlighted in bold. Instances marked with $^\dagger$ ran out of memory when solving SPLA or NSEP. \textcolor{black}{The estimated average revenue marked with $^*$ exceeds the UB by 1.67 in absolute value. However, given the estimator's standard error of 4.7, this is attributed to variability.}  \label{table:SPLA_comparison}}	
\end{table}

\color{black}
AA provides a tractable baseline approximation, with the smallest instances being solved in mere seconds and the largest instance taking less than 12 hours. We observe relative optimality gaps (UB vs. LB) below 20\% for the smallest instances (H\&S with $L=2$, $\tau=20$, and Bus instances). For the remaining cases, the gap exceeds 54\%. For each $L$, the worst performance is observed for moderate time horizons ($\tau=50,100$), and performance generally deteriorates as $L$ increases. In particular, H\&S scenarios with $L=20$ and $\tau=50,100$ exhibit the worst performance, with gaps exceeding 125\% and thus leaving more room for improvement. In fact, although H-2PIAlg outperforms AA by improving both UBs and LBs across all instances, these latter H\&S instances show the largest UB improvement, exceeding 28\%. Interestingly, the most significant LB improvements occur in H\&S instances with $L=2$ and large time horizons, yielding up to 37\% more average revenue. More generally, for H\&S instances with $\tau\geq 50$, UB and LB improvements remain above 10\% and 15\%, respectively. These improvements, however, come at a computational cost, with CPU times approaching the time limit set in the cluster.

Despite SPLA being a strong NRM benchmark, it fails to solve the largest instances in our numerical experiments due to memory limitations; those instances are marked with $\dagger$.  For instances that do solve, computational times remain reasonable, always below 7253 seconds, and optimality gaps are often relatively small, particularly for H\&S instances with larger time horizons. However, the larger the $L$ and the smaller the $\tau$, the larger the gaps, exceeding 10\% in some cases. H-2PIAlg capitalizes on this room for improvement to enhance SPLA's UB and LB by up to 4\% and 0.54\%, respectively, for instances with short time horizons (SBL, CBL and H\&S instances with $\tau=20$). These improvements come at a computational cost, ranging from 1.8h to 156h. For other instances, no improvement is observed. In fact, for the largest instances that SPLA solves (e.g., $L=20$, $\tau=50,100,200$), SPLA produces a UB more than 9\% better than H-2PIAlg. NSEP also fails to handle larger instances, marked with $\dagger$ in Table \ref{table:SPLA_comparison}. Among the small instances it could solve, NSEP outperformed SPLA, providing UBs and LBs that are nearly tight and close to optimal, with the largest gap being 3.4\%. Despite this, for the smallest instances ($\tau=20$, SBL, and CBL), H-2PIAlg generated better policies, achieving up to a 0.22\% improvement in average revenue. While H-2PIAlg generally requires additional computational effort (e.g., 153 hours more for CBL), its runtime can also be shorter (e.g., 3.4 hours less for $L=2$, $\tau=20$). However, the key advantage of H-2PIAlg over NSEP and SPLA lies in scalability: H-2PIAlg remains computationally feasible for larger instances, providing reasonable UB and LB values where the benchmarks cannot.

\color{black}

\paragraph{\bf Discussion:} \textcolor{black}{Summarizing}, we find that H-2PIAlg is particularly competitive under two conditions: (1)  in large problem instances, when SPLA and NSEP run out of memory, and (2) when capacity is very scarce. The reason \textcolor{black}{behind (1) is its smaller complexity space}. For instance, SPLA has \( T \times \left( I + J + \sum_{i=1}^{I} (c_i + 1) \times (J_i + 1) \right) \) variables and 
$ T \left( \sum_{i=1}^{I} c_i + 3 \sum_{i=1}^{I} J_i + 4 \sum_{i=1}^{I} J_i c_i \right) $ constraints, with \( J_i=\sum_{j=1}^J a_{ij} \). For the largest instance, this yields $41,880,000$ variables and $159,360,000$ constraints. In contrast, H-2PIAlg has a more manageable complexity. Instead of solving one large problem, namely \eqref{AP}, our algorithm solves smaller problems (linear master program, the flow-balance problem \eqref{Wflowimbalance}, and subproblems \eqref{subproblem_t_primal} and \eqref{subproblem_Tau_primal}) multiple times. The linear master program consists of   $\tau(K+1)$ variables and $N$ constraints, where $N$ is iteratively increased through row generation. The number of basis functions, $K$, is relatively small compared to $N$ and is incremented by one each time \eqref{Wflowimbalance} is solved. This problem involves $I$
continuous variables and a single linear constraint. The nonlinear subproblems for row generation feature $I$ integer variables, $J$
binary variables, and 
$IJ$ linear constraints. \textcolor{black}{ Each iteration of the algorithm requires a exponential number of floating-point operations to solve the subproblems in the worst case. Although the dimension of the subproblems remain invariant across iterations, the shape of their objective functions become increasingly complex with $K$, thereby increasing solution times. However, in our numerical experience we observe that the solution times are mainly driven by the LPs, whose sizes increase with $N$.}
 By progressively adding basis functions, H-2PIAlg incrementally increases the complexity of the problem and provides an approximation at every stage. 
\textcolor{black}{If the problem becomes too complex to solve in later steps, solutions from earlier, simpler stages remain valid. As a result, this algorithm requires longer computational times compared to other approaches but allows to approximate significantly larger instances.}

\textcolor{black}{The reason for (2) is that} the value function is more challenging to approximate in these scenarios, where non-separability and non-linearities play a significant role. Since the exponential ridge basis functions are able to exploit nonlinearities and interactions \textcolor{black}{across all} resources, H-2PIAlg adds multiple basis functions so that \eqref{NLA} is able to capture these intricate aspects of the value function. \textcolor{black}{In fact, the shorter the time horizon, the larger the $K$ produced by H-2PIAlg (see Table \ref{table:times} of the Appendix).} This observation aligns with the findings in \cite{Simon}, which indicate that the benefits of non-separability and non-linearity diminish as capacity increases. As also noted in \cite{Simon}, NSEP is particularly effective at exploiting non-separability and non-linearity  to improve upper bounds. However, because NSEP only models \textcolor{black}{these} within a small subnetwork and relies on affine approximations for all other resources,  policies generated by NSEP do not show significant improvement compared to AA. Since H-2PIAlg allows for non-separability and non-linearity across all resources, its ability to enhance policies and obtain higher average revenues is superior.

\section{Concluding Remarks and Extensions \label{sec:conclusion}}

In this paper we proposed a Nonlinear Incremental Algorithm (NLIAlg), which dynamically generates and optimizes basis functions for approximating value functions in NRM. The algorithm progresses through two main steps: (I) Estimate the weights and parameters of the approximation \eqref{NLA} with fixed number of basis functions $K$ using a row generation algorithm, and (II) Increase $K$ until the approximation is deemed satisfactory according to user-specified criteria. This approach allows for iterative improvements, balancing dimensionality and accuracy. In addition, NLIALg stands out as the first algorithm to optimize new basis functions rather than sampling them (see \cite{bhat} and \cite{pakimanetal2019}). 

\textcolor{black}{To address large instances, the Two-Phase Incremental Algorithm (2PIAlg) estimates the parameters in two phases within Step (I).} In Phase (i), the parameters of the basis functions are estimated following the flow imbalance ideas introduced by \citet{adelmanklabjan2012}. In Phase (ii), the remaining parameters of the approximation are estimated by solving the master problem with fixed basis function parameters, which simplifies the problem into an LP. This constitutes the first application of the flow-balance methodology in a stochastic setting.

Both NLIAlg and 2PIAlg can be used either in (1) {\it Standalone mode} to find a suitable value function approximation from scratch, or in (2) {\it Add-on mode}, to enhance a given value function approximation.
Our numerical results show that NLIAlg and 2PIAlg in standalone mode are practical only for toy problems. However, the heuristic version of 2PIAlg (H-2PIAlg) on add-on mode with the Affine Approximation \eqref{AA} and exponential ridge basis functions \eqref{exponential_basis} can solve instances so large that the benchmark NRM methods SPLA \citep{vossenzhang2014} and NSEP \citep{Simon} run out of memory. In addition, in these instances H-2PIAlg provides significantly better policies and upper bounds than AA. Moreover, H-2PIAlg is able to improve the policies of SPLA, NSEP and AA for those instances where capacity is scarce.

A potential future direction is exploring hybrid policies that combine the strengths of SPLA and our proposed method. For instance, using SPLA early in the booking period and switching to our approach later could yield significant benefits. Optimizing this transition point through methods such as line search or advanced policy search approaches could be a valuable extension of this work. Moreover, adapting our methodology to other stochastic dynamic programming problems beyond NRM is a promising avenue for future research. Finally, investigating the theoretical convergence of our approach in a broader Markov decision process context would also be an interesting line of research.



\section*{Appendices}
\appendix

\section{Additional theoretical results and proofs\label{appendix_proofs}}

\subsection{Norm-Constrained Convex Ridge Basis Functions Give Rise to DC Optimization problems\label{sec:basis_ridge} }

Basis functions that are convex or concave in $\vbeta$ and $\vx$ give rise to optimization problems in our algorithms whose objective and constraints are difference-of-convex (DC) functions\footnote{ 
    We call a function $f:\mathbb R^n\to \mathbb R$ a \em{DC function} if it can be expressed as the difference of two convex functions, i.e., if $f(\vec{y})=f_1(\vec{y})-f_2(\vec{y})$, where $f_1,f_2: \mathbb R^n\to \mathbb R$ are convex functions. Moreover, we say an optimization problem $\min\{f(\vec y): g_i(\vec y)\le0, i=1,\dots,m\}$ is a \em{DC program} if $f,g_1,\dots,g_m:\mathbb R^n\mapsto \mathbb R$ are DC functions. 
}. Under such a choice of basis functions, running NLIAlg implies solving a large DC master problem and low-dimensional integer DC problems for row generation. 2PIAlg entails solving a large linear master program, low-dimensional integer DC problems for row generation, and low-dimensional continuous DC problems for basis function generation. For convenience, the following proposition gathers these statements and characterizes the structure of the value function approximation.

\begin{proposition}\label{lemma:DCgeneral}

Let $\phi(\cdot)$ be convex (resp. concave) in $\vbeta$ and $\vx$, and let the baseline approximation $\psi_t(\cdot)$ be concave on $\vx$. Then:

\begin{itemize}
\item[(i)] Problem \eqref{AP} is a DC program on $\xi_t$, $V_{t,k}$ and $\vbeta_k$.
\item[(ii)] The continuous relaxations of subproblems \eqref{subproblem_t_primal} and \eqref{subproblem_Tau_primal} are DC programs in $\vx$ and $\vu$.
\item[(iii)] The flow-balance problem \eqref{Wflowimbalance} is a DC program on $\vbeta$.
\item[(iv)] The continuous relaxations of subproblems \eqref{constraint_extra}, $t=1,...,\tau-1$, that need so be solved to generate monotonicity constraints are DC programs in $\vx$.
\end{itemize}
\end{proposition}

\begin{proof}{Proposition \ref{lemma:DCgeneral}:}
Proposition 2.1 in \cite{DC_overview} claims that any finite linear combination and product of DC functions is also DC. According to this statement, (i) and (ii) are fulfilled. In addition, Proposition 2.1 in \cite{DC_overview} also claims that the absolute value of a DC function is still DC, rendering \eqref{Wflowimbalance} a DC program. 
\end{proof}


\subsection{Supporting Lemma for Theorem 1}

The following lemma that is essential for proving Theorem \ref{lemma:column_generation}.

\begin{lemma}\label{lemma:decomposition}
    For any value function approximation $v(\vx)=\xi_t+\psi_t(\vx)+\sum_{k+1}^K V_{t,k}\phi(\vx;\vbeta)$ and any dual variables $\lambda_{t,(\vx,\vu)}\geq 0$ the following decomposition applies

\begin{flalign}
   \Xi(\xi;\lambda)+\Psi(\lambda)+\Phi( V,\hat\vbeta;\lambda)&= \sum_{t=1}^\tau \sum_{(\vx,\vu)\in \mathcal{F}_t} \lambda_{t,(\vx,\vu)} \left\{\hat{v}_t(\vx) - \left\{ \sum_{j=1}^J p_{t,j} \left[f_ju_j+\mathbf{1}_{\{t<\tau\}} \hat{v}_{t+1}(\vx-\vec a_j u_j)\right]\right.\right. \nonumber\\
   &\quad \left. \left. +\left( 1-\sum_{j=1}^J p_{t,j}\right) \mathbf{1}_{\{t<\tau\}}\hat{v}_{t+1}(\vx)\right\} \right\}, \label{eq:decomposition}
\end{flalign}
where

\begin{flalign*}
&\Xi(\xi;\lambda):=-\sum_{t=1}^\tau  \sum_{(\vx,\vu)\in \mathcal{Q}_t(\lambda)} \lambda_{t,(\vx,\vu)} \sum_{j=1}^J p_{t,j} f_ju_j+ \xi_1,  \nonumber\\
&\Psi(\lambda):= \sum_{t=1}^\tau  \sum_{(\vx,\vu)\in \mathcal{Q}_t(\lambda)} \lambda_{t,(\vx,\vu)} \left\{\psi_t(\vx) -  \left[  \sum_{j=1}^J p_{t,j} \mathbf{1}_{\{t<\tau\}} \psi_{t+1}(\vx-\vec a_j u_j) +\left( 1-\sum_{j=1}^J p_{t,j}\right) \mathbf{1}_{\{t<\tau\}}\psi_{t+1}(\vx)\right] \right\}, \nonumber\\
& \Phi(V,\vbeta;\lambda) := -\sum_{k=1}^{K}  V_{1,k}
 \phi(\vec c;\vbeta_k) 
-\sum_{t=2}^{\tau}\sum_{k=1}^{K}  V_{t,k} \ell_t(\vbeta_k;\lambda).
\end{flalign*}
and $\mathcal{Q}_t(\lambda):=\{(\vx,\vu)\in \mathcal{F}_t: \  \lambda_{t,\vx,\vu)}>0\}$.

\end{lemma}

\begin{proof}{Lemma \ref{lemma:decomposition}:}
Introducing the set  $\mathcal{Q}_t(\hat{\lambda}):=\{(\vx,\vu)\in \mathcal{F}_t: \ \hat \lambda_{t,\vx,\vu)}>0\}$, we can decompose the terms in the right hand side of \eqref{eq:decomposition} as
\begin{flalign}
\Xi(\xi;\lambda)&=\sum_{t=1}^\tau  \sum_{(\vx,\vu)\in \mathcal{Q}_t(\lambda)} \lambda_{t,(\vx,\vu)} \left\{\xi_t - \left\{ \sum_{j=1}^J p_{t,j} \left[f_ju_j+\mathbf{1}_{\{t<\tau\}} \xi_{t+1} \right]+\left( 1-\sum_{j=1}^J p_{t,j}\right) \mathbf{1}_{\{t<\tau\}}\xi_{t+1}\right\} \right\} \nonumber \\
&=-\sum_{t=1}^\tau  \sum_{(\vx,\vu)\in \mathcal{Q}_t(\lambda)} \lambda_{t,(\vx,\vu)} \sum_{j=1}^J p_{t,j} f_ju_j+\sum_{t=1}^\tau  \sum_{(\vx,\vu)\in \mathcal{Q}_t(\lambda)} \lambda_{t,(\vx,\vu)} \left( \xi_t - \mathbf{1}_{\{t<\tau\}} \xi_{t+1}\right) \nonumber\\
&=-\sum_{t=1}^\tau  \sum_{(\vx,\vu)\in \mathcal{Q}_t(\lambda)} \lambda_{t,(\vx,\vu)} \sum_{j=1}^J p_{t,j} f_ju_j+\sum_{t=1}^\tau  \left( \xi_t - \mathbf{1}_{\{t<\tau\}} \xi_{t+1}\right) \nonumber\\
&=-\sum_{t=1}^\tau  \sum_{(\vx,\vu)\in \mathcal{Q}_t(\lambda)} \lambda_{t,(\vx,\vu)} \sum_{j=1}^J p_{t,j} f_ju_j+ \xi_1,  \nonumber\\
\Psi(\lambda)&= \sum_{t=1}^\tau  \sum_{(\vx,\vu)\in \mathcal{Q}_t(\lambda)} \lambda_{t,(\vx,\vu)} \left\{\psi_t(\vx) -  \left[  \sum_{j=1}^J p_{t,j} \mathbf{1}_{\{t<\tau\}} \psi_{t+1}(\vx-\vec a_j u_j) +\left( 1-\sum_{j=1}^J p_{t,j}\right) \mathbf{1}_{\{t<\tau\}}\psi_{t+1}(\vx)\right] \right\}, \nonumber\\
\Phi( V,\vbeta;\lambda) &= \sum_{t=1}^\tau  \sum_{(\vx,\vu)\in \mathcal{Q}_t(\lambda)} \lambda_{t,(\vx,\vu)} \left\{-\sum_{k=1}^{K}V_{t,k}\phi(\vx;\vbeta_k) -\left[ - \sum_{j=1}^J p_{t,j} \mathbf{1}_{\{t<\tau\}} \sum_{k=1}^{K}V_{t+1,k}\phi(\vx-\vec a_j u_j;\vbeta_k)\right. \right.   \nonumber\\
& \left.\left. -\left( 1-\sum_{j=1}^J p_{t,j}\right) \mathbf{1}_{\{t<\tau\}}\sum_{k=1}^{K}V_{t+1,k}\phi(\vx;\vbeta_k)\right]  \right\} \nonumber\\
=& -\sum_{t=1}^{\tau}\sum_{k=1}^{K}  V_{t,k}
\sum_{(\vx,\vu)\in \mathcal{Q}_t(\lambda)} \lambda_{t,(\vx,\vu)} \phi(\vx;\vbeta_k)  \nonumber\\
&+\sum_{t=1}^{\tau-1}\sum_{k=1}^{K}  V_{t+1,k}\sum_{j=1}^J p_{t,j}
\sum_{(\vx,\vu)\in \mathcal{Q}_t(\lambda)} \lambda_{t,(\vx,\vu)} \phi(\vx-\vec a_j u_j;\vbeta_k) \nonumber\\
&+\sum_{t=1}^{\tau-1}\sum_{k=1}^{K}  V_{t+1,k}(1-\sum_{j=1}^J p_{t,j})
\sum_{(\vx,\vu)\in \mathcal{Q}_t(\lambda)} \lambda_{t,(\vx,\vu)} \phi(\vx;\vbeta_k) 
\nonumber\\
=& -\sum_{k=1}^{K}  V_{1,k}
\sum_{(\vx,\vu)\in \mathcal{Q}_1} \lambda_{1,(\vx,\vu)} \phi(\vx;\vbeta_k) 
-\sum_{t=2}^{\tau}\sum_{k=1}^{K^*}  V_{t,k}
\sum_{(\vx,\vu)\in \mathcal{Q}_t(\lambda)} \lambda_{t,(\vx,\vu)} \phi(\vx;\vbeta_k)  \nonumber\\
&+\sum_{t=2}^{\tau}\sum_{k=1}^{K}  V_{t,k}\sum_{j=1}^J p_{t-1,j}
\sum_{(\vx,\vu)\in \mathcal{Q}_{t-1}} \lambda_{t-1,(\vx,\vu)} \phi(\vx-\vec a_j u_j;\vbeta_k) \nonumber\\
&+\sum_{t=2}^{\tau}\sum_{k=1}^{K}  V_{t,k}(1-\sum_{j=1}^J p_{t-1,j})
\sum_{(\vx,\vu)\in \mathcal{Q}_{t-1}} \lambda_{t-1,(\vx,\vu)} \phi(\vx;\vbeta_k) 
\nonumber\\
=& -\sum_{k=1}^{K}  V_{1,k}
\sum_{(\vx,\vu)\in \mathcal{Q}_1} \lambda_{1,(\vx,\vu)} \phi(\vx;\vbeta_k) 
-\sum_{t=2}^{\tau}\sum_{k=1}^{K}  V_{t,k} \ell_t(\vbeta_k;\lambda)\nonumber\\
=& -\sum_{k=1}^{K}  V_{1,k}
 \phi(\vec c;\vbeta_k) 
-\sum_{t=2}^{\tau}\sum_{k=1}^{K}  V_{t,k} \ell_t(\vbeta_k;\lambda). \nonumber &\square
\end{flalign}

\end{proof}

\subsection{Proofs of all results in the manuscript}

\begin{proof}{Proposition \ref{lemma:reducebound}:}
Let $(D_K)$ denote the dual problem with $K$ basis functions. Let $(D_{K+1})$ be the dual problem where a violated flow-balance constraint has been added; i.e. $|\ell_t(\vbeta_{K+1})|>0$ and the following violated constraint for $t$ has been added. Then, \eqref{constraint_beta} reads:
\begin{equation} \label{FBnew} 
 \sum_{(\vx,\vu) \in \mathcal{F}_t} \phi(\vx;\vbeta_{K+1}) \lambda_{t,(\vx,\vu)} 
= \begin{cases} 
- \phi(\vc;\vbeta_{K+1}) &  \mbox{ if } t=1 \\
\sum_{(\vx,\vu) \in \mathcal{F}_{t-1}} \left( \sum_{j=1}^J p_{j,t-1} \phi(\vx-\vec{a}_j u_j;\vbeta_{K+1}) \right.\\
\left. \qquad+ (1 - \sum_{j=1}^J p_{j,t-1}) \phi(\vx;\vbeta_{K+1}) \right) \lambda_{t-1,(\vx,\vu)}   &  \mbox{ otherwise.}
\end{cases}
\end{equation}

Let $\Lambda^K$ and $\Lambda^{K+1}$ be the optimal solutions to $(D_K)$ and $(D_{K+1})$, respectively. The solution  $\Lambda^K$ leads to  $|\ell_t(\vbeta_{K+1})|>0$ and hence violates \eqref{FBnew}.
 Thus, $\Lambda^K$ is not a feasible solution for $(D_{K+1})$. In contrast, $\Lambda^{K+1}$ is feasible to $(D_K)$. Because  $\Lambda^K$ is the unique
 and optimal solution to $(D_K)$,

\begin{flalign*}
    &\sum_{t=1}^{\tau-1} \sum_{(\vx,\vu) \in \mathcal{F}_t}  \lambda^K_{t,(\vx,\vu)}\left\{ \sum_{j=1}^J p_{j,t} \left[f_j u_j+\psi_{t+1}(\vx-\vec{a}_j u_j)\right] +\Big(1- \sum_{j=1}^J p_{j,t}\Big)\psi_{t+1}(\vx)-\psi_t(\vx) \right\} \\
   & \qquad + \sum_{(\vx,\vu) \in \mathcal{F}_\tau} \lambda^K_{\tau,(\vx,\vu) } \Big( \sum_{j=1}^J p_{j,\tau} f_j u_j-\psi_{\tau}(\vx) \Big) \\
    &\not=\sum_{t=1}^{\tau-1} \sum_{(\vx,\vu) \in \mathcal{F}_t}  \lambda^{K+1}_{t,(\vx,\vu)}\left\{ \sum_{j=1}^J p_{j,t} \left[f_j u_j+\psi_{t+1}(\vx-\vec{a}_j u_j)\right] +\Big(1- \sum_{j=1}^J p_{j,t}\Big)\psi_{t+1}(\vx)-\psi_t(\vx) \right\}\\
    & \qquad + \sum_{(\vx,\vu) \in \mathcal{F}_\tau} \lambda^{K+1}_{\tau,(\vx,\vu) } \Big( \sum_{j=1}^J p_{j,\tau} f_j u_j-\psi_{\tau}(\vx) \Big) 
\end{flalign*}
 
%

Because the feasible region of $(D_{K+1})$ does not contain $\Lambda^K$, the optimal objective value of  $(D_{K+1})$  is (strictly) smaller than that of $(D_K)$. Strong duality then implies that linearized \eqref{AP} with the new variables $V_{t,K+1}$ and new basis function $\phi(\vx,\vbeta_{K+1})$ reduces the previous upper bound. \hfill $\square$
\end{proof}

\begin{proof}{Theorem \ref{lemma:column_generation}:}
Using Lemma \ref{lemma:decomposition}, complementary slackness yields
\begin{flalign}
    0&=\medmath{\sum_{t=1}^\tau \sum_{(\vx,\vu)\in \mathcal{F}_t} \hat{\lambda}_{t,(\vx,\vu)} \left\{\hat{v}_t(\vx) - \left\{ \sum_{j=1}^J p_{t,j} \left[f_ju_j+\mathbf{1}_{\{t<\tau\}} \hat{v}_{t+1}(\vx-\vec a_j u_j)\right] 
    +\left( 1-\sum_{j=1}^J p_{t,j}\right) \mathbf{1}_{\{t<\tau\}}\hat{v}_{t+1}(\vx)\right\} \right\}} \nonumber\\
&=\Xi(\hat \xi;\hat{\lambda})+\Psi(\hat{\lambda})+\Phi(\hat V,\hat\vbeta;\hat{\lambda}), \label{eq:decomposition_hat}
\end{flalign}

On the other hand, because $\hat{\lambda}_t,(\vx,\vu)\geq 0$ and $v^*_t(\vx) \geq  \sum_{j=1}^J p_{t,j} \left[f_ju_j+\mathbf{1}_{\{t<\tau\}} v^*_{t+1}(\vx-\vec a_j u_j) \right]+\left( 1-\sum_{j=1}^J p_{t,j}\right) \mathbf{1}_{\{t<\tau\}}v^*_{t+1}(\vx)$ for all $(\vx,\vu)\in \mathcal{F}_t$, we have that

\begin{flalign}
    0&\leq \medmath{\sum_{t=1}^\tau  \sum_{(\vx,\vu)\in \mathcal{F}_t} \hat{\lambda}_{t,(\vx,\vu)} \left\{v^*_t(\vx) - \left\{ \sum_{j=1}^J p_{t,j} \left[f_ju_j+\mathbf{1}_{\{t<\tau\}} v^*_{t+1}(\vx-\vec a_j u_j) \right]+\left( 1-\sum_{j=1}^J p_{t,j}\right) \mathbf{1}_{\{t<\tau\}}v^*_{t+1}(\vx)\right\} \right\}.}\nonumber\\    &=\Xi(\xi^*;\hat{\lambda})+\Psi(\hat{\lambda})+\Phi(V^*,\vbeta^*;\hat{\lambda}) \label{eq:decomposition*}
\end{flalign}
where the equality comes from Lemma \ref{lemma:decomposition}. Subtracting \eqref{eq:decomposition_hat} from \eqref{eq:decomposition*}, yields

\begin{flalign}
    0& \leq\left[\Xi(\xi^*;\hat{\lambda})+\Psi(\hat{\lambda})+\Phi(V^*,\vbeta^*;\hat{\lambda})\right]-\left[\Xi(\hat \xi;\hat{\lambda})+\Psi(\hat{\lambda})+\Phi(\hat V,\vbeta;\hat{\lambda})\right] \nonumber\\
    &= \Xi(\xi^*;\hat{\lambda})-\Xi(\hat \xi;\hat{\lambda})+\Phi(V^*,\vbeta^*;\hat{\lambda})-\Phi( \hat V,\hat \vbeta;\hat{\lambda}) \nonumber\\
&= \xi_1^*-\sum_{k=1}^{K^*}  V^*_{1,k}\phi(\vec c;\vbeta^*_k) 
-\Big(\hat \xi_1 -\sum_{k=1}^{\hat K}  \hat V_{1,k}\phi(\vec c;\hat \vbeta_k) \Big)
-\sum_{t=2}^{\tau}\sum_{k=1}^{K^*}  V^*_{t,k} \ell_t(\vbeta^*_k;\hat{\lambda})
+\sum_{t=2}^{\tau}\sum_{k=1}^{\hat K} \hat  V_{t,k} \ell_t(\hat \vbeta_k;\hat{\lambda}) \nonumber \\
&=\left[v^*_1(\vec c)-\hat v_1(\vec c)\right] -\sum_{t=2}^{\tau}\sum_{k=1}^{K^*}  V^*_{t,k} \ell_t(\vbeta^*_k;\hat{\lambda}).\nonumber
\end{flalign}

\noindent where the last equality holds due to $\ell_t(\hat \vbeta_k;\hat{\lambda})=0$. Since $v^*_1(\vec c)-\hat v_1(\vec c)<0$ by assumption, we cannot have $ \ell_t(\vbeta^*_k;\hat{\lambda})=0$ for all $t=1,...,\tau$, $k=1,...,K^*$. In fact, we need 

\begin{align*}
    \hspace{5cm} -\sum_{t=2}^{\tau}\sum_{k=1}^{K^*}  V^*_{t,k} \ell_t(\vbeta^*_k;\hat{\lambda})\geq v^*_1(\vec c)-\hat v_1(\vec c). \hspace{5cm} \square
\end{align*}

%

\end{proof}
\begin{proof}{Proposition \ref{lemma:ridge}:}
Using the convexity of $g$, we obtain
\begin{align*}
\phi(\lambda \vx +(1-\lambda) \vx')&=g(\vbeta^{\top}(\lambda \vx + (1-\lambda) \vx'))
=g(\lambda\vbeta^{\top}\vx+(1-\lambda)\vbeta^{\top}\vx')\\
&\leq \lambda g(\vbeta^{\top}\vx)+(1-\lambda)g(\vbeta^{\top}\vx')
=\lambda\phi( \vx)+(1-\lambda) \phi(\vx').& \square
\end{align*}
%


%
\end{proof}

\begin{proof}{Proposition \ref{prop:bounds}:}

For given $k$, a lower bound of the basis function \eqref{exponential_basis} can be found by solving 
\begin{align*}
\min_{\vbeta_k,\vx\in \{0,1,...,c_i\}^I} \qquad e^{-\sum_i \beta_{i,k} x_i } & \qquad 
\mbox{s.t.} \qquad \sum_i c_i |\beta_{i,k}| =1
\end{align*}
Since the objective function is monotone decreasing in $\sum_i \beta_{i,k} x_i$, we can write this  problem as
\begin{align*}
\max_{\vbeta_k,\vx\in \{0,1,...,c_i\}^I}  \qquad \sum_i \beta_{i,k} x_i  & \qquad 
\mbox{s.t.}  \qquad \sum_i c_i |\beta_{i,k}|=1  
\end{align*}
This objective function is monotone increasing in every $x_i$ with $\beta_{i,k}\geq 0$ and monotone decreasing in $x_i$ given $\beta_{i,k}<0$. Hence, the optimal value of the previous problem is upper bounded by the optimal objective of the following program
\begin{align*}
\max_{\vbeta_k,\vx\in \{0,1,...,c_i\}^I} \qquad \sum_i \beta_{i,k} x_i  & \qquad
\mbox{s.t.}  \qquad \sum_i c_i \beta_{i,k}=1, 
 \ \ \beta_{i,k} \geq 0\qquad \forall i=1,...,I 
\end{align*}

Given the non-negativity constraint on the $\vbeta_k$ and the norm constraint, the objective is upper bounded as follows $\sum_i \beta_{i,k} x_i \leq \sum_i \beta_{i,k} c_i =1$. Hence, a lower bound of the basis function \eqref{exponential_basis}  is given by $e^{-1}$ as long as $\vbeta_k$ is constrained by \eqref{Wnorm1}. Similarly, an upper bound of  $e^{1}$ can be obtained. 
 \hfill $\square$ 
\end{proof}

\section{Additional details on the algorithms \label{appendix:algorithms}}


\subsection{Flow-imbalances including monotonicity constraints}

We use dual variables $\lambda_{t,(\vec x,\vec u)}$ for the constraints defined in $(\vx,\vu)\in \mathcal B_t^{\lambda}$,  and $\mu_{t,i,\vec x}$ for the monotonicity constraints $(i,\vx)\in \mathcal B_t^{\mu}$. Then the dual formulation of the approximate LP including monotonicity constraints is

\begin{samepage}
\begin{alignat}{2}
  \max_{\lambda_{t,(\vx,\vu)} \ge 0} \quad & \sum_{t=1}^{\tau-1} \sum_{(\vx,\vu) \in \mathcal{B}_{t}^\lambda}  \lambda_{t,(\vx,\vu)}\left\{ \sum_{j=1}^J p_{j,t} \left[f_j u_j+\psi_{t+1}(\vx-\vec{a}_j u_j)\right] +\left(1- \sum_{j=1}^J p_{j,t}\right)\psi_{t+1}(\vx)-\psi_t(\vx) \right\} \nonumber \\
& \quad + \sum_{(\vx,\vu) \in \mathcal{B}_{\tau}^\lambda} \lambda_{\tau,(\vx,\vu) } \left( \sum_{j=1}^J p_{j,t} f_j u_j-\psi_{\tau}(\vx) \right)  
  \label{Dmu}\tag{D$^\mu$} \\
\text{s.t.} \quad & \sum_{(\vx,\vu) \in \mathcal{B}_{t}^\lambda} \lambda_{t,(\vx,\vu)} 
= \begin{cases}
1 & \text{if } t=1 \\
\displaystyle \sum_{(\vx,\vu) \in \mathcal{B}_{t-1}^\lambda } \lambda_{t-1,(\vx,\vu)} & \text{if } t=2,\dots,\tau
\end{cases} 
\nonumber\\
& \sum_{(\vx,\vu) \in \mathcal{B}^\lambda_t} \lambda_{t,(\vx,\vu)}  \phi(\vx;\vbeta_k)  -\displaystyle   \sum_{(i,\vx) \in \mathcal{B}_{t}^\mu} (\phi_k(\vx;\vbeta_k) -  \phi_k(\vx^{+i};\vbeta_k)) \mu_{t,i,\vec x}= \nonumber\\
&\quad \begin{cases} 
\phi(\vc;\vbeta_k) & \forall k,t=1 \\
\displaystyle\sum_{(\vx,\vu) \in \mathcal{B}_{t-1}^\lambda} \lambda_{t-1,(\vx,\vu)} \left( \sum_{j=1}^J p_{j,t-1} \phi(\vx-\vec{a}_j u_j;\vbeta_k) \right. \left. + (1 - \sum_{j=1}^J p_{j,t-1}) \phi(\vx;\vbeta_k) \right) & \forall k, t=2,\dots,\tau. 
\end{cases} \label{constraint_beta_mu} \tag{D.FB$^\mu$}
\end{alignat}
\end{samepage}

This yields the flow-imbalance functions

\begin{flalign*}
\ell^\mu_1(\vbeta)= &\displaystyle \sum_{(\vx,\vu) \in \mathcal{B}_1^\lambda} \phi{_k}(\vx;\vbeta) \lambda_{1,(\vx,\vu)}
-\displaystyle  \sum_{(i,\vx) \in \mathcal{B}_{1}^\mu} (\phi{_k}(\vx;\vbeta)-\phi{_k}(\vx;^{+i}\vbeta)) \mu_{1,i,\vec x}+ \phi{_k}(\vc;\vbeta)  & 
\end{flalign*}
\begin{flalign*}
\ell^\mu_t(\vbeta)=  &\displaystyle \sum_{(\vx,\vu) \in \mathcal{B}_t^\lambda} \phi{_k}(\vx;\vbeta) \lambda_{t,(\vx,\vu)}
-\displaystyle   \sum_{(i,\vx) \in \mathcal{B}_{t}^\mu} (\phi{_k}(\vx;\vbeta)- \phi{_k}(\vx^{+i};\vbeta))\mu_{t,i,\vec x}\\
&\quad - \displaystyle \sum_{(\vx,\vu) \in \mathcal{B}_{t-1}^\lambda} \left( \sum_{j=1}^J p_{j,t-1}\phi{_k}(\vx-\vec a_j u_j;\vbeta)+(1-\sum_{j=1}^J p_{j,t-1})\phi{_k}(\vx;\vbeta) \right) \lambda_{t-1,(\vx,\vu)} & \nonumber  \qquad  t=2,\dots,\tau\nonumber.
\end{flalign*}


\subsection{Pseudocodes}

\vspace{-0.25cm}
\begin{algorithm}[H]
    \caption{Row Generation Algorithm \label{row_generation}}
  \begin{algorithmic}[1] 
  \Require Initial subset of constraints $\mathcal B^{\lambda}_1, \dots, \mathcal B^{\lambda}_{\tau-1},  \mathcal B^{\lambda}_{\tau}$.
  \While{stopping criterion $\mathcal{C}_R$ is not met} \Comment{Solve the master problem}
  \State Solve the master problem \eqref{AP} given constraints $\mathcal{B}=\bigcup_{t=1}^{\tau}\mathcal{B}_t^\lambda$ to obtain $\hat \xi,\hat V, \hat  \vbeta$.
  \For{$t=1,\dots,$ $\tau$} \Comment{Generate rows}
  \If{ $t<\tau$} 
  \State Solve subproblem (\ref{subproblem_t_primal}) to obtain solution $(\vec x^*_t,\vec u^*_t)$
  \Else  
  \State Solve subproblem (\ref{subproblem_Tau_primal}) to obtain solution $(\vec x^*_{\tau}\,\vec u^*_{\tau})$
  \EndIf
        \If {objective value is negative}  
       \State $\mathcal B^\lambda_t  \leftarrow \mathcal B^\lambda_t \cup (\vec x^*_t,\vec u^*_t)$ 
        \EndIf
 \EndFor
   \EndWhile
      \Ensure Estimates $\hat \xi,\hat V, \hat  \vbeta$ for the improved approximation
  \end{algorithmic}
\end{algorithm}

\vspace{-0.75cm}

\begin{algorithm}[H]
    \caption{Nonlinear Incremental Algorithm (NLIAlg)  \label{DC_algorithm}}
  \begin{algorithmic}[1] 
  \Require Initial subset of constraints $\mathcal B^{\lambda}_1, \dots, \mathcal B^{\lambda}_{\tau-1},  \mathcal B^{\lambda}_{\tau}$
\Initialize $K \leftarrow 0$
  \While{stopping criterion $\mathcal{C}_K$ is not met}
  \State  $K \leftarrow K+1$ \Comment{Increment the number of basis functions}
  \State Start Row Generation Algorithm given $\mathcal B^{\lambda}_1, \dots, \mathcal B^{\lambda}_{\tau-1},  \mathcal B^{\lambda}_{\tau}$ and $K$ to obtain $\hat {\vec \xi},\hat {\vec V}, \hat  \vbeta_k, k=1,\dots,K$.
  \EndWhile
      \Ensure Number of basis functions $K$ and estimates $\hat {\vec \xi},\hat {\vec V}, \hat  \vbeta_k, k=1,\dots,K$ for the improved approximation
  \end{algorithmic}
\end{algorithm}

\begin{algorithm}[h!]
    \caption{Two-Phase Incremental Algorithm  (2PIAlg) \label{2phase_incremental}}
  \begin{algorithmic}[1] 
    \Require Initial subset of constraints $\mathcal B^{\lambda}_1, \dots, \mathcal B^{\lambda}_{\tau-1},  \mathcal B^{\lambda}_{\tau}$ and 
    \Initialize $K \leftarrow 0$ 
   \While{stopping criterion $\mathcal{C}_K$ is not met}
  \State  $K \leftarrow K+1$  \Comment{Generate a new basis function}
  \If{ $K=1$} 
  \State Pick any $\hat \vbeta_{1}$ fulfilling constraint \eqref{Wnorm1}
  \Else 
  \State Solve (\ref{Wflowimbalance}) to find $\hat \vbeta_{K}$ 
  \EndIf   \Comment{Row Generation Algorithm}
   \State Solve the LP \eqref{AP} for fixed $K$ and $\hat  \vbeta_1,...,\hat  \vbeta_K$ to estimate $\hat {\vec \xi},\hat {\vec V}$
     \EndWhile
      \Ensure Number of basis functions $K$ and estimates $\hat {\vec \xi},\hat {\vec V},\hat  \vbeta_1,...,\hat  \vbeta_K$ for the improved approximation
  \end{algorithmic}
\end{algorithm}


\begin{algorithm}[h!]
    \caption{Heuristic Row Generation Algorithm including monotonicity constraints \label{heuristic_row}}
  \begin{algorithmic}[1] 
  \Require $\vbeta$ and initial subset of constraints $\mathcal B^{\lambda}_1, \dots, \mathcal B^{\lambda}_{\tau-1},  \mathcal B^{\lambda}_{\tau}$ and $\mathcal{B}_1^\mu,\dots,\mathcal{B}_\tau^\mu$.
  \While{stopping criterion $\mathcal{C}_R$ is not met} \Comment{Solve the master problem}
  \State Solve the master problem \eqref{AP} given constraints $\mathcal{B}=\bigcup_{t=1}^{\tau}\mathcal{B}_t$ to obtain $\hat \xi,\hat V$.
  \State Initialize $\mathcal{T} \leftarrow \{1,...,\tau\}$
 \While{$\mathcal{T} \neq \emptyset$}\Comment{Generate rows}
 \State Randomly pick $t \in \mathcal{T}$
  \State Use a {\it local} solver to solve either subproblem (\ref{subproblem_Tau_primal}) or  (\ref{subproblem_t_primal}) and obtain solution $(\hat{\vec x}_t,\hat{\vec u}_t)$
  \State $\mathcal{T}^- \leftarrow \{t' \in \mathcal{T}: \ \hat{\pi}_{\mathcal B,t'}<0\}$
  \State $\mathcal B^\lambda_{t'}  \leftarrow \mathcal B^\lambda_{t'} \cup(\hat{\vec x}_t,\hat{\vec u}_t)$ for all $t'\in \mathcal{T}^- $ \Comment{Add all rows with negative reduced costs}
  \State $\mathcal{T}^- \leftarrow  \mathcal{T} \setminus \mathcal{T}^- $ \Comment{Update $\mathcal{T}$}
    \EndWhile
   \State Initialize $\mathcal{T} \leftarrow \{1,...,\tau\}$
 \While{$\mathcal{T} \neq \emptyset$} \Comment{Generate monotonicity constraints}
 \State Randomly pick $t \in \mathcal{T}$ and $i \in \{1,...,T\}$.
 \State Use a {\it local} solver to  mimize RHS of \eqref{constraint_extra}  to obtain solution $(i,\hat{\vec x}_t)$
 \For{$i'=1,...,I$}
 \State $\mathcal{T}^-_{i'} \leftarrow \{t' \in \mathcal{T}: \ \hat{\rho}_{\mathcal B,t',i'}<0\}$
 \State $\mathcal{B}^\mu_{t'}  \leftarrow \mathcal B^\mu_{t'} \cup (i',\vec x^*_t)$ for all $t'\in \mathcal{T}^-_{i'}$ \Comment{Add all rows with negative reduced costs}
 \EndFor
 \State $\mathcal{T}^- \leftarrow  \mathcal{T} \setminus \bigcup_{i'} \mathcal{T}^- _{i'}$ \Comment{Update $\mathcal{T}$}
\EndWhile
    \EndWhile
      \Ensure Estimates $\hat \xi,\hat V, \hat  \vbeta$ for the improved approximation
  \end{algorithmic}
\end{algorithm}

\section{Computational details\label{appendix_comp}}
At the end of each row generation cycle, we simulate the policy given by \eqref{policy} to determine if our estimate of $Z_{\phi}$ is close enough to the simulated average policy revenue and the algorithm is stopped, or if we should generate  another basis function. To obtain the average policy revenue for a given value function approximation we
\begin{itemize}
	\item[0.] let $N=1$;
	\item[1.] simulate the policy given by \eqref{policy} for one instance, i.e. one departure date, and save the total revenue obtained as $R^{(N)}$;
	\item[2.] compute the average policy revenue $	\overline{R}=\frac{1}{N} \sum_{n=1}^{N} R^{(n)}$ as well as the standard error of $\overline{R}$. 
	$S_e=\frac{1}{\sqrt{N}}\left(\sqrt{\frac{1}{N} \sum_{n=1}^{N} (R^{(n)})^2- \overline{R}^2}\right)$;
	\item[3.] If $S_e/\overline{R}<\Omega^{policy}$, the simulation is stopped, otherwise $N=N+1$ and go to 1.
\end{itemize} 
In our numerical experiments, we used $\Omega^{policy}= 0.1\%$.



 For the stopping criteria, we used $\Omega^{gap}=0.1\%, \ \Omega^{policy}=0.1\%$ and $\Omega^{p-gap}=1\%$. For the H-2PIAlg, {\tt Knitro} was used to solve the linear master program, the nonlinear integer subproblems, and the nonconvex continuous flow-balance maximization problem. To further improve computational performance, we chose time limits of 5 seconds to solve the subproblems. For the NLIAlg, the global solver {\tt Baron} was used to solve any DC problem.

Finally, Table \ref{table:SH_instances} summarizes the capacities of the hub-and-spoke instances generated for each number of non-hub locations $L$ and time periods $\tau$. 

\begin{table}[h!]
	\centering	\begin{small}
		\begin{tabular}{r|ccccc}
\backslashbox{$\tau$}{ L}	&	2	&	3	&	5	&	10	&	20	\\
 \hline
20	&	3	&	2	&	2	&	1	&	1	\\
50	&	8	&	6	&	4	&	2	&	1	\\
100	&	17	&	12	&	8	&	5	&	2	\\
200	&	33	&	25	&	17	&	9	&	5	\\
500	&	83	&	62	&	41	&	23	&	12	\\
1000	&	165	&	124	&	83	&	45	&	24	\\
\end{tabular}
\end{small}
\caption{Capacities of the hub-and-spoke instances generated. \label{table:SH_instances}}
\end{table}


\section{Additional Numerical Results}

\subsection{Comparing the performance of H-2PIAlg with hat versus exponential ridge basis functions}

Piece-wise linear 
hat functions, also known as hat functions or B-splines of degree 1, could take the following form in our our NRM setting
\begin{equation} \label{hat_basis}
    \phi_t(\vx;\vbeta_k,\vec{b}_k):=\left\{\begin{array}{ll}
    \frac{\vbeta^\top \vx -b_{k,t-1}}{b_{k,t}-b_{k,t-1}}     &  \mbox{ if } b_{k,t-1} \leq \vbeta^\top \vx \leq b_{k,t}   \\
     \frac{b_{k,t+1}-\vbeta^\top \vx }{b_{k,t+1}-b_{k,t}}     &  \mbox{ if } b_{k,t} \leq \vbeta^\top \vx \leq b_{k,t+1}   \\
     0 & \mbox{ otherwise,}
    \end{array} \right.
\end{equation}
where not only the $I$-dimensional parameter vector $\vbeta$ but also the $(\tau+2)$-dimensional vector $\vec{b}$ need to be estimated via maximizing flow imbalances. To generate a new basis function with $\phi(\cdot)$ of the form \eqref{hat_basis}, H-2PIAlg needs to solve a MINLP with approximately \(\sum_{t=1}^\tau (10+J)N_t\) constraints in the worst-case scenario, where \(N_t\) is the number of rows generated for period \(t\). When $N_t$ is large, as it will be for moderately to large instances, the use of hat functions impractical. In contrast, generating exponential basis functions entails solving a continuous problem with \(I\) variables and one constraint. 
We considered modifying \eqref{hat_basis} by parametrizing basis functions over time, specifically generating \(\tau\) basis functions with parameters \(\vbeta_{t,K+1}\) and breakpoints \(b_{t,-1}, b_{t,0}, b_{t,1}\) for \(t=1,...,\tau\). While this increases the dimensionality of our approximation, it allows for the decomposition of the flow-imbalance maximization problem.  

The toy example discussed in Section \ref{sec:hat_vs_exponential} uses the fares and arrival probabilities specified in Table \ref{parameters_toy}. Results reported in Section \ref{sec:hat_vs_exponential} were computed using {\tt Knitro}. To further assess the tractability of piecewise linear basis functions, we attempted to use {\tt Gurobi} instead of {\tt Knitro} for row generation in H-2PIAlg with hat basis functions. We chose {\tt Gurobi} because it can convert the subproblem, which involves products of bounded continuous variables and binary variables, into an LP. However, after generating 90 basis functions (9 per period of time), {\tt Gurobi} also struggled to produce new rows.

\begin{table}[h!]
\centering
\begin{tabular}{c|ccc|ccc}
 & \multicolumn{3}{c|}{low fare} &  \multicolumn{3}{c}{high fare} \\
 & AB &BC &AC &AB &BC &AC\\
  \hline
$ f_{j}$ & 20 & 30 &42 & 100 & 150 & 210\\
 $p_{j}$ & 0.3 &0.1125  & 0.1875 & 0.1 & 0.0375  & 0.0625\\
\end{tabular} 
\caption{Parameters of the toy example. \label{parameters_toy}}
\end{table}


Table \ref{table:toy} provides the upper bounds, average revenues \(\overline{R}\) obtained by simulating policy (\ref{policy}), and CPU times in seconds for H-2PIAlg using exponential ridge basis functions.
Figure \ref{fitting_toy} visualizes the different iterations of the H-2PIAlg in this small example. The black dots represent the true values $v_t(\vx)$, computed using value iteration, for some periods $t$. The surface corresponds to the approximated function 
$\xi_t -\sum_{k=1}^K V_{t,k} \phi(\vx;\vbeta_{k})$ after $K$ basis functions have been added. Plots in the first row show the results of the affine approximation for comparison. The following lines show the corresponding pictures of the H-2PIAlg for $K=1,4$ and 8. We can observe that the more basis functions the H-2PIAlg adds, the more the surface resembles the true value function, fitting the optimal value function at the discrete points representing feasible states. 



\begin{table}[h!]
	\centering	\begin{small}
		\begin{tabular}{l|rrr}
			&\multicolumn{3}{c}{H-2PIAlg}\\
			$K$ 	&  \multicolumn{1}{c}{$\hat{Z}_{\phi}$} &\multicolumn{1}{c}{$\overline{R}$} & Time\\
			\hline	
1	&	448.5	&	324.2	&	0	\\
2	&	430.4	&	386.2	&	1	\\
3	&	420.5	&	387.1	&	4	\\
4	&	419.6	&	387.1	&	6	\\
5	&	418.1	&	385.3	&	8	\\
6	&	412.4	&	393.7	&	11	\\
7	&	411.5	&	394.5	&	13	\\
8	&	404.1	&	395.7	&	21	\\
9	&	401.8	&	395.0	&	25	\\
10	&	399.5	&	397.2	&	28	\\
\end{tabular}
\end{small}
\caption{Upper bound on the value function and average simulated revenue for the H-2PIAlg. \label{table:toy}}
\end{table}

\begin{figure}[h!]
\centering
\begin{tabular}{ccc}
\includegraphics[width=170pt]{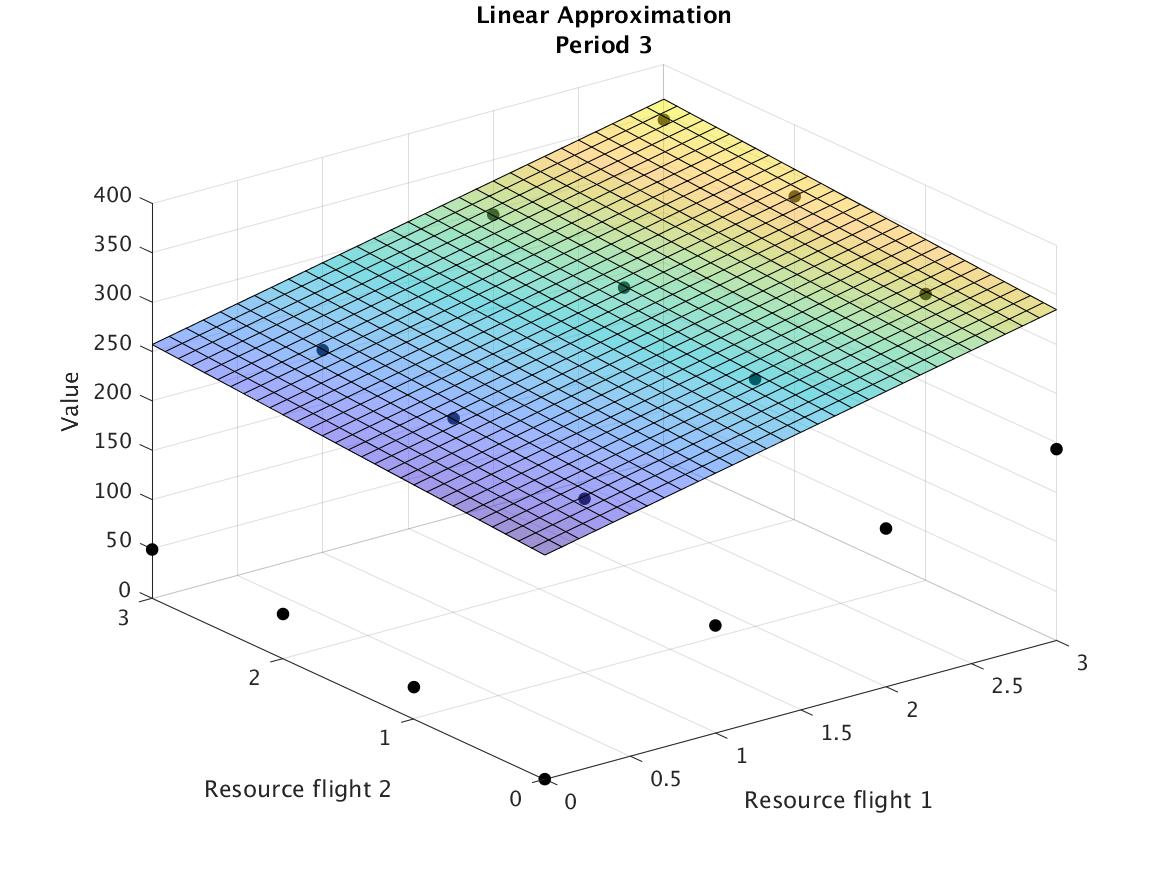} &  \includegraphics[width=170pt]{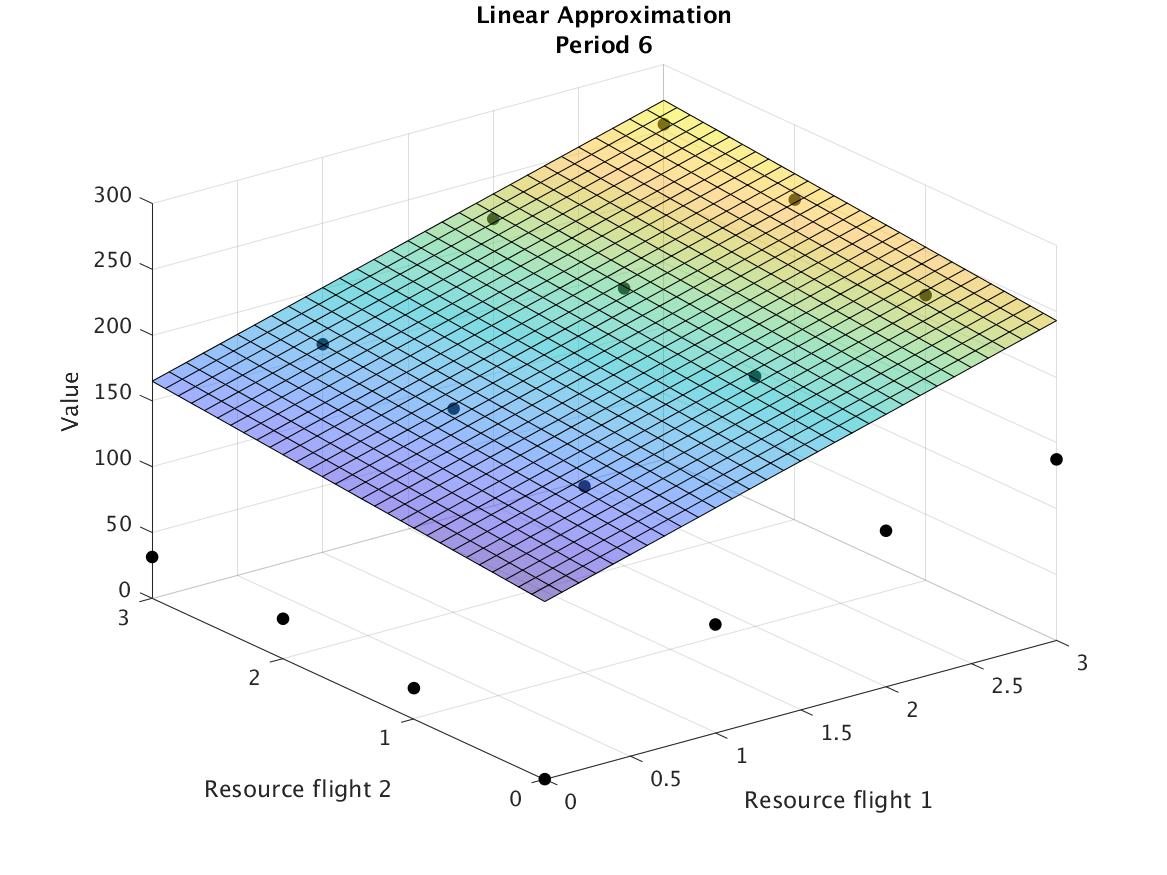} &\includegraphics[width=170pt]{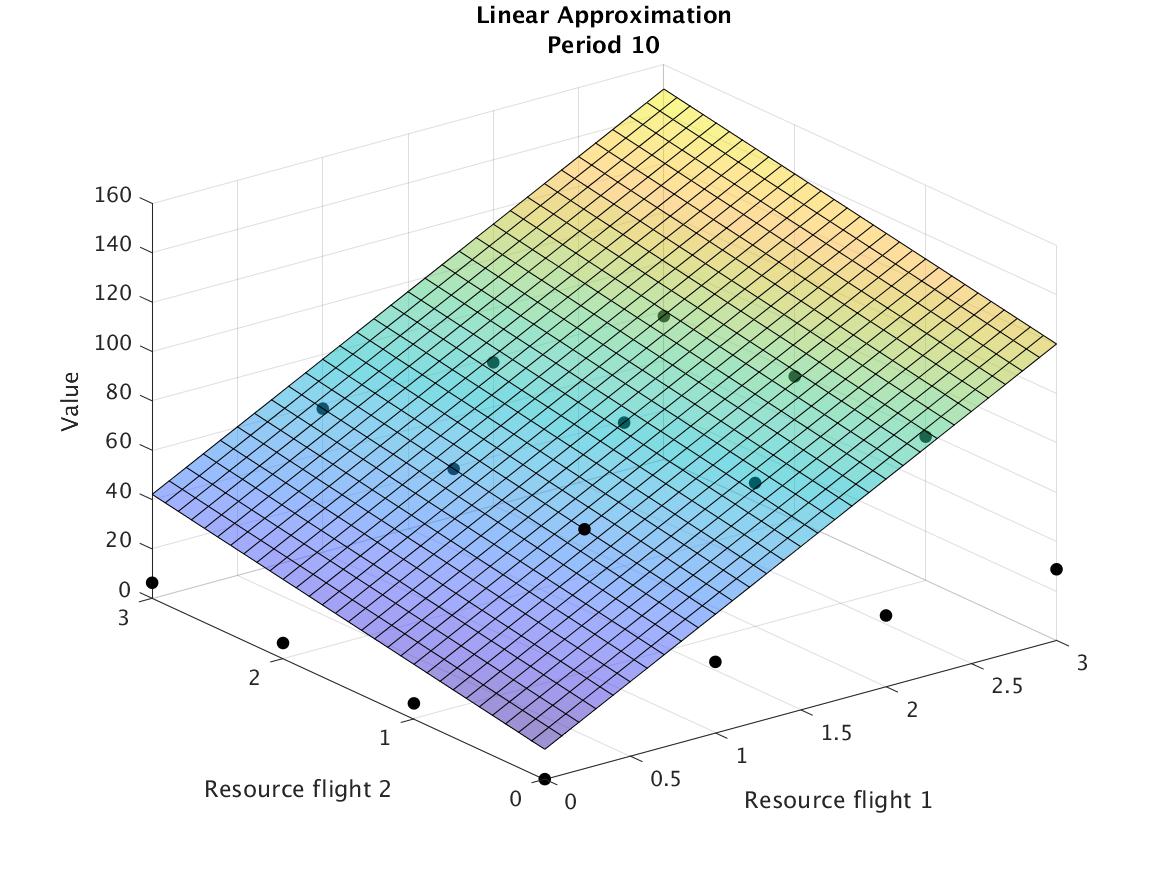} \\
\includegraphics[width=170pt]{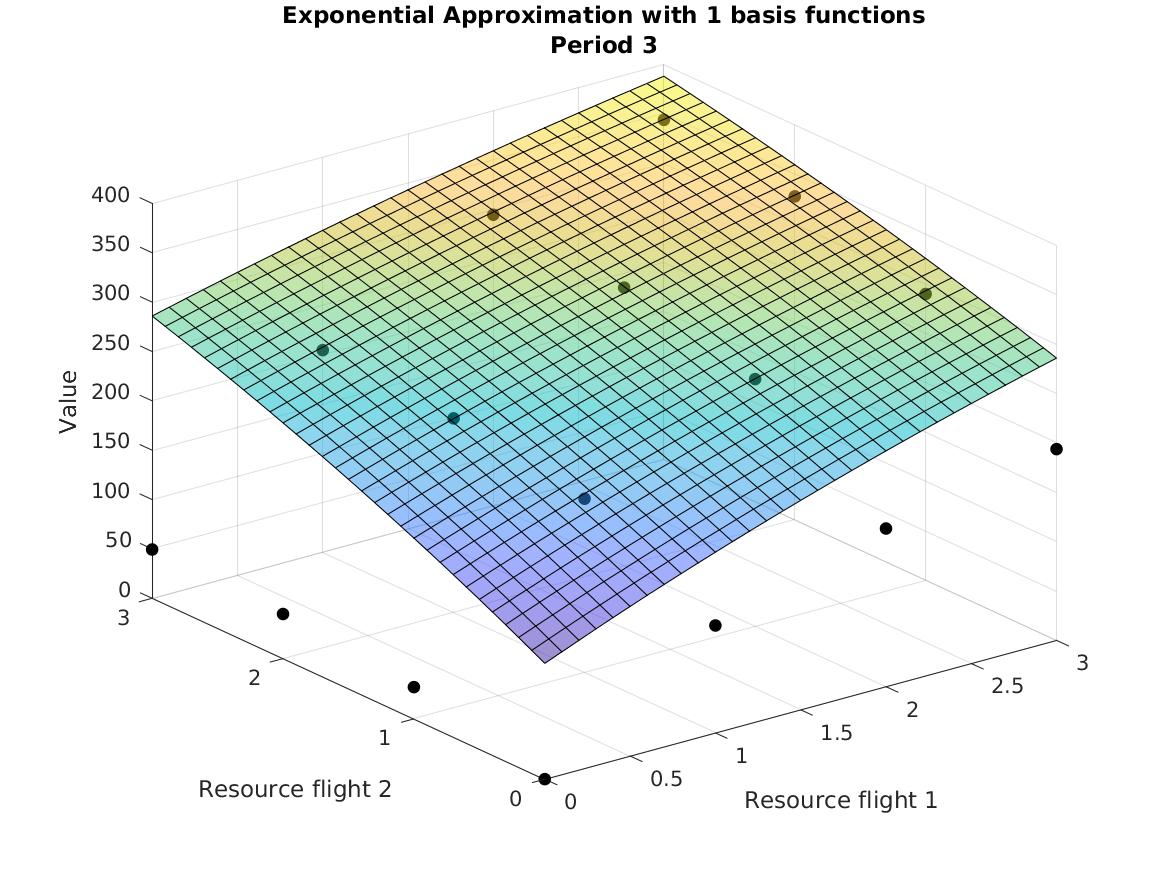} &  \includegraphics[width=170pt]{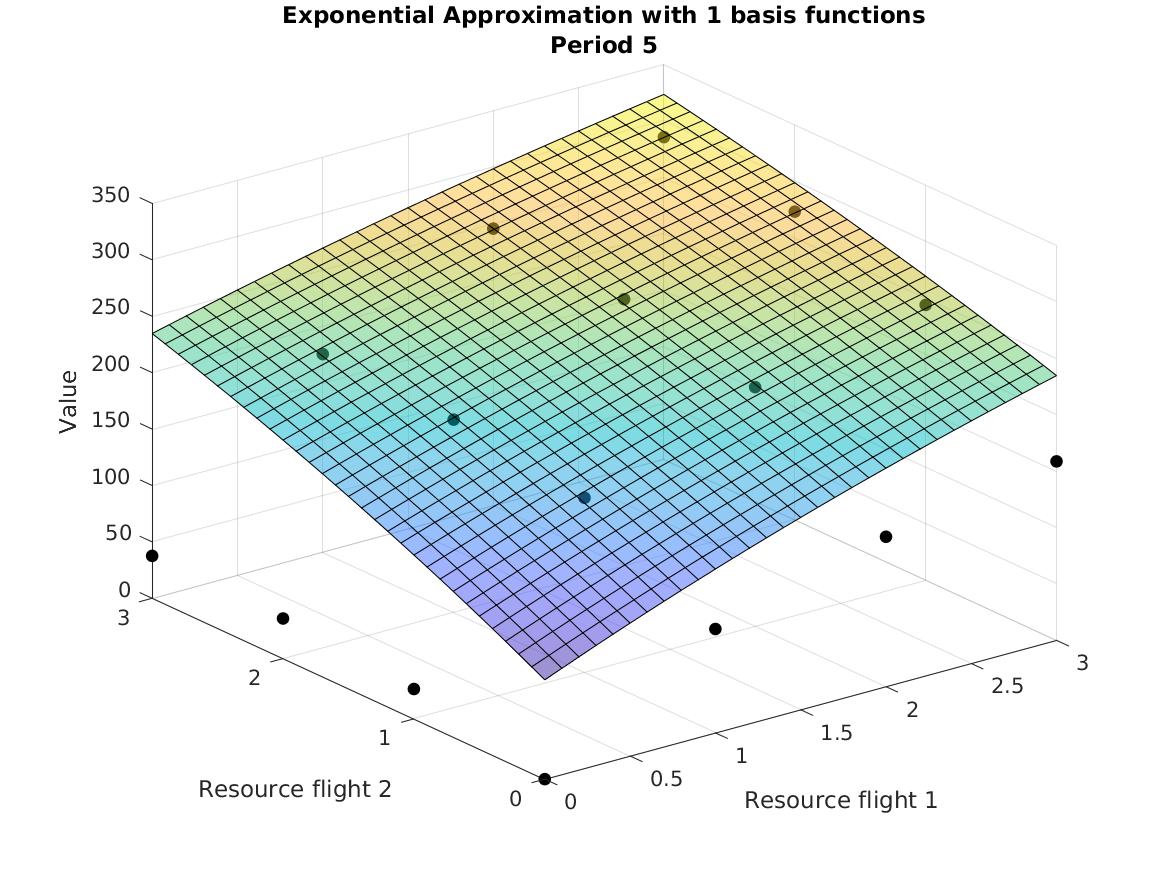} &\includegraphics[width=170pt]{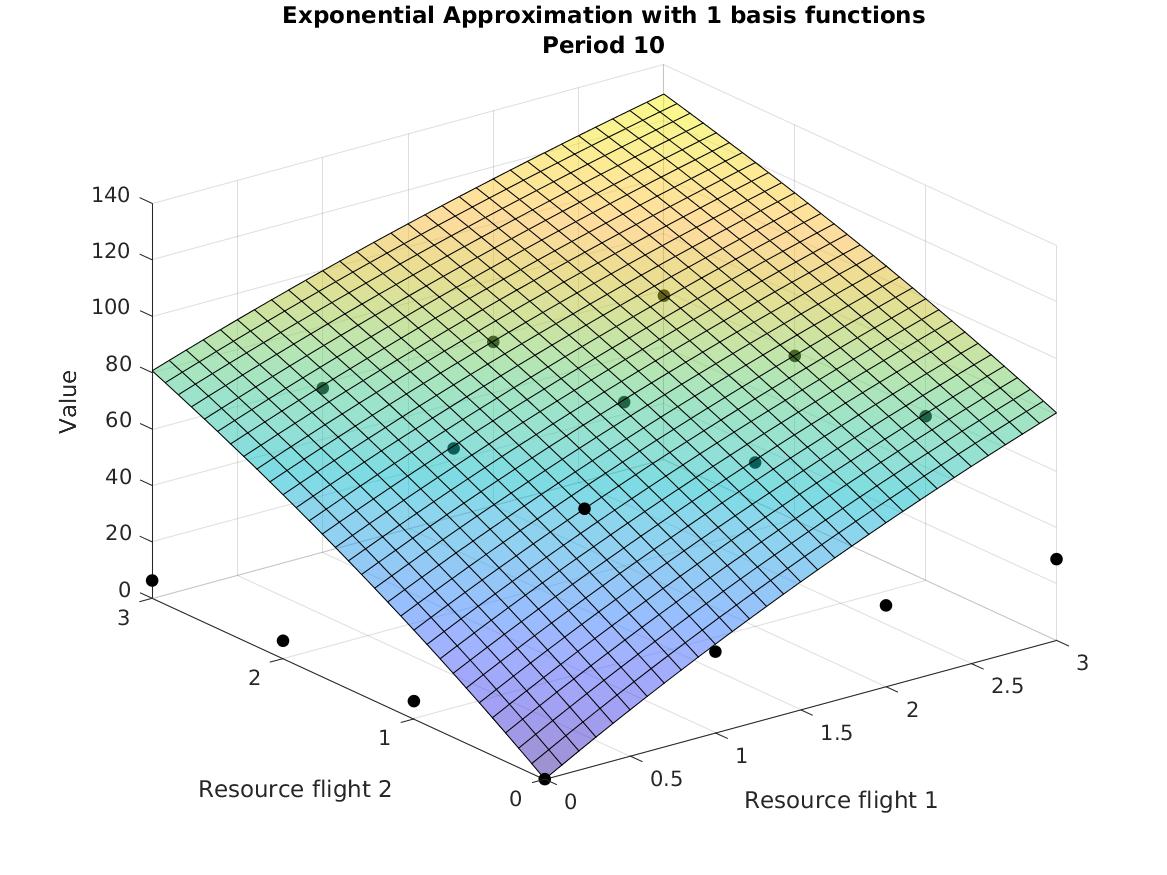} \\
\includegraphics[width=170pt]{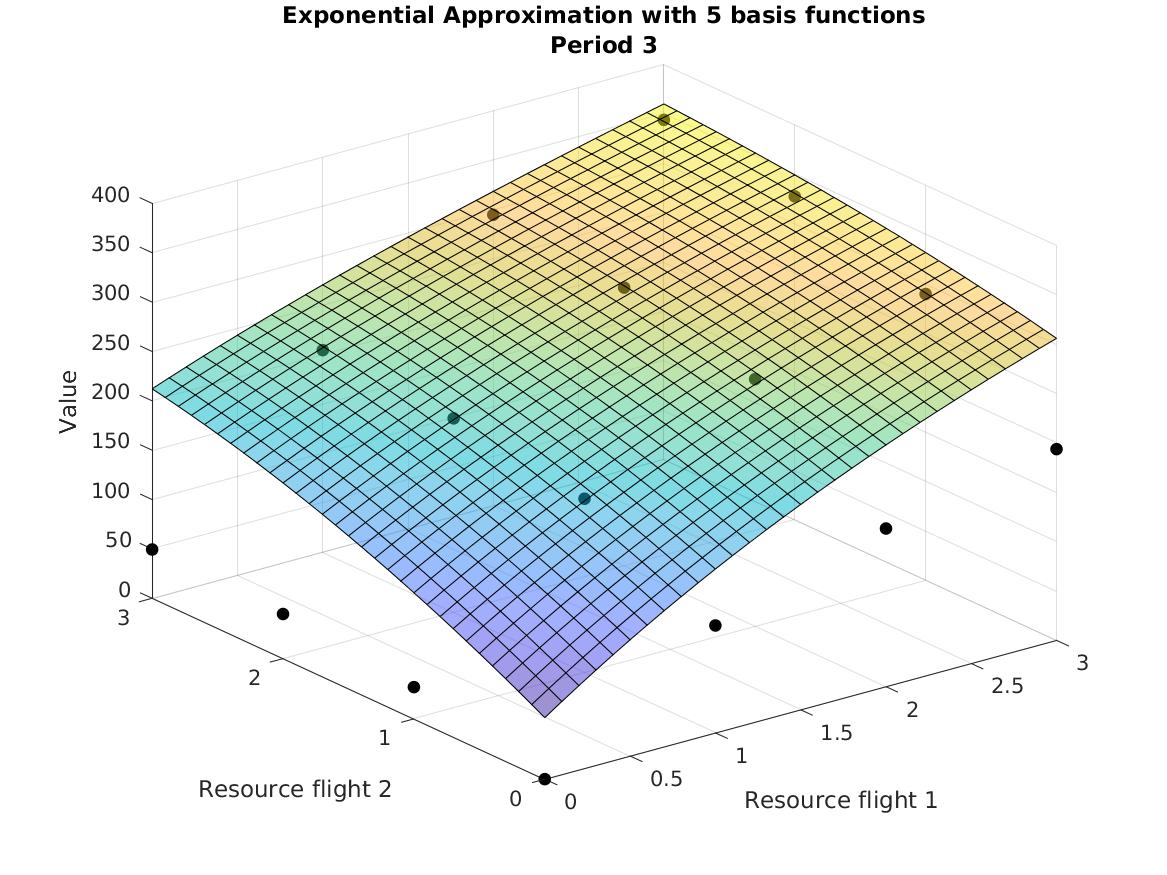} &  \includegraphics[width=170pt]{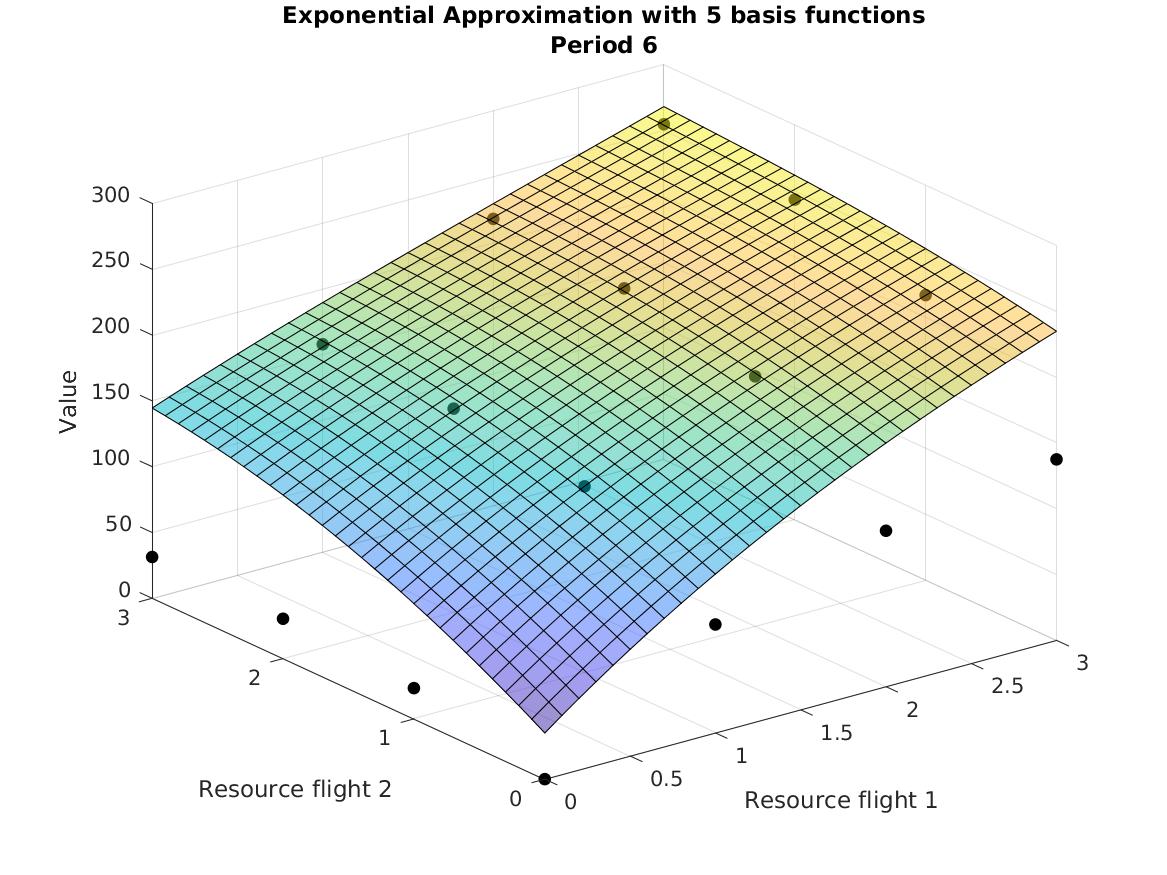} &\includegraphics[width=170pt]{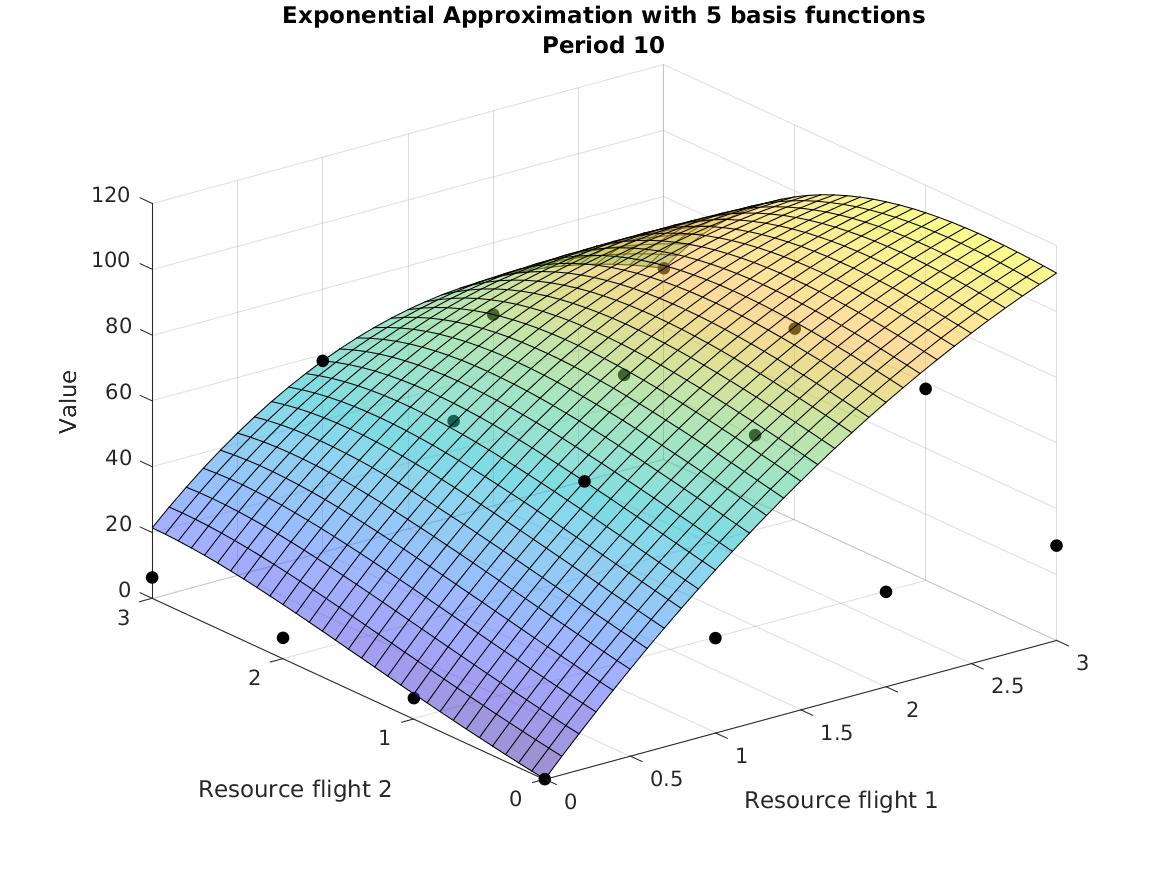} \\
\includegraphics[width=170pt]{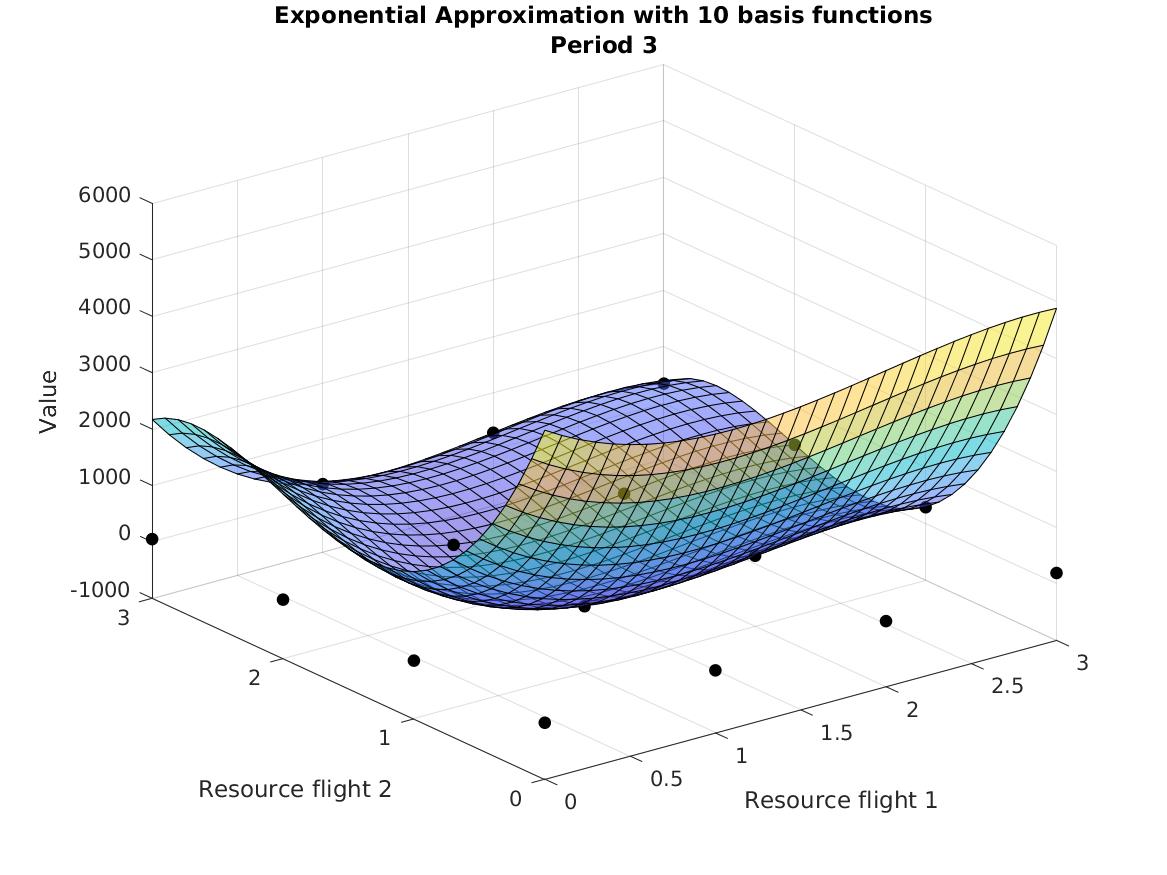} &  \includegraphics[width=170pt]{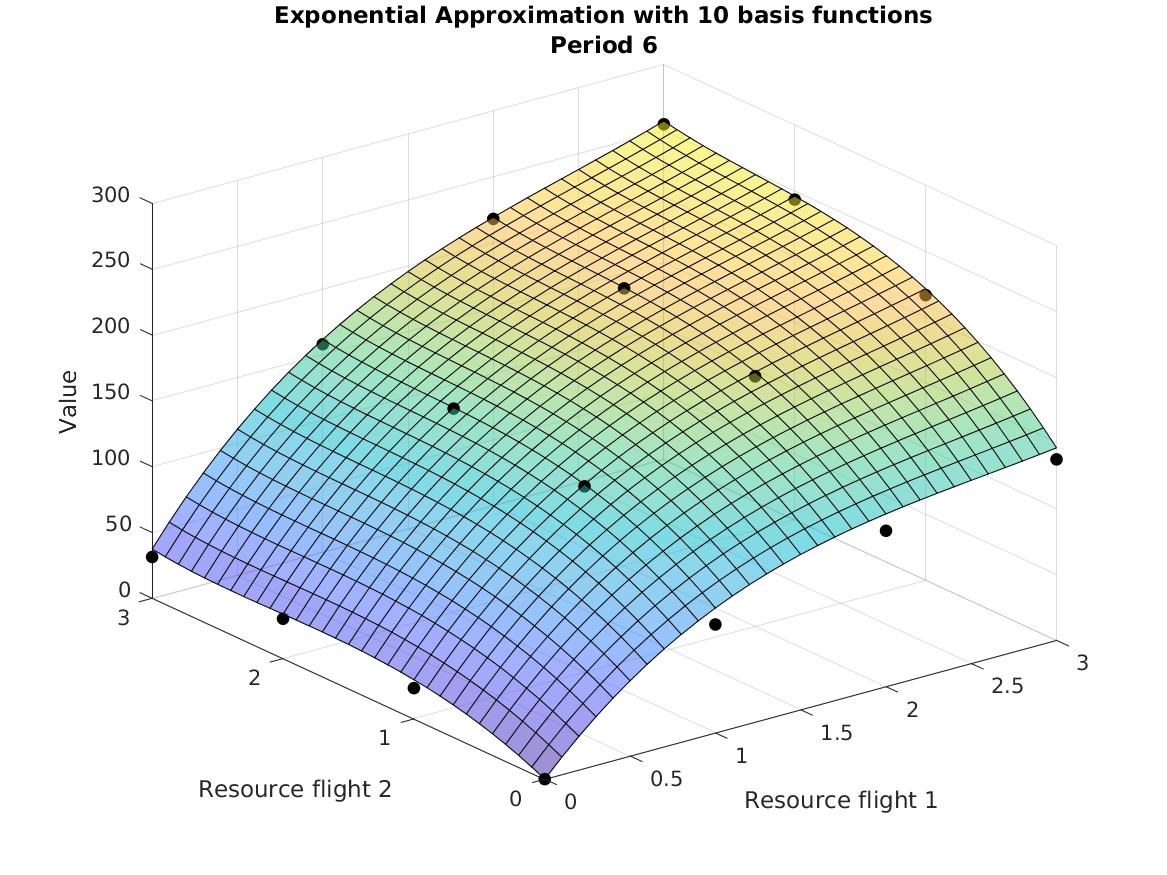} &\includegraphics[width=170pt]{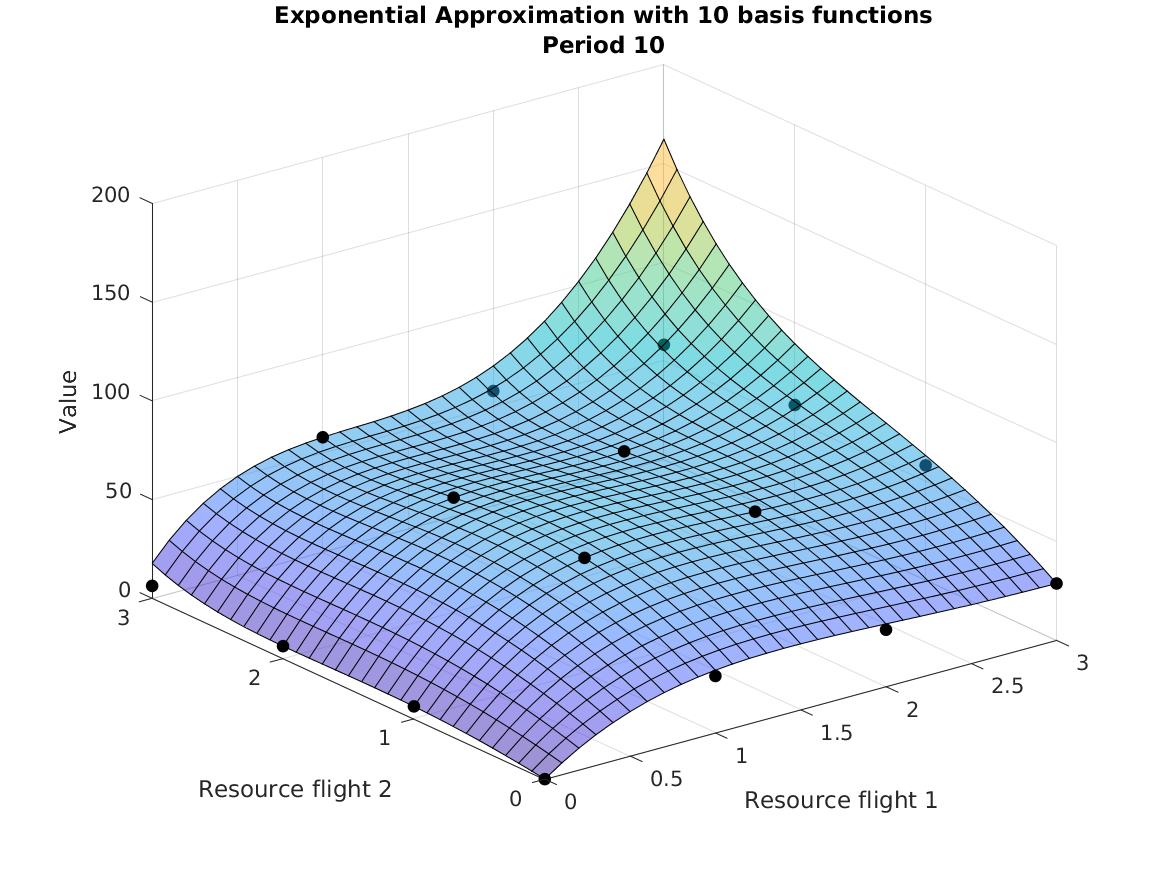} \\
\end{tabular}\\
\caption{\label{fitting_toy}{Approximated value function through linear and exponential basis functions}}
\end{figure}

\subsection{Results from NLIAlg \label{appendix_dc}}

Table \ref{table:toyNLIAlg} displays the upper bound, simulated average reward and CPU time in seconds for the approximation provided by NLIAlg as the number of basis functions increases.

\begin{table}[H]
	\centering	\begin{small}
		\begin{tabular}{l|rrr}
			&\multicolumn{3}{c}{NLIAlg}\\
			$K$ 	& \multicolumn{1}{c}{$\hat{Z}_{\phi}$} &\multicolumn{1}{c}{ $\overline{R}$} & Time	 \\
			\hline	
1	&	418.7	&	385.1	&	1813		\\
2	&	410.9	&	388.9	&	3022		\\
3	&	405.7	&	394.2	&	3627		\\
4	&	401.4	&	396.6	&	4846		\\
5	&	400.4	&	397.1	&	5157		\\
\end{tabular}
\end{small}
\caption{Upper bound on the value function and average simulated revenue for the NLIAlg. \label{table:toyNLIAlg}}
\end{table}

\color{black}

Figure \ref{convergence_DC1} shows the fitting of the value function estimated via the NLIAlg in the toy example introduced in Section \ref{sec:hat_vs_exponential}. The black dots represent the true values $v_t(\vx)$, computed using value iteration. The surface corresponds to the approximated function $\xi_t -\sum_{k=1}^K V_{t,k} \phi_{k}(\vx;\vbeta_{k})$ after $K$ basis functions have been added. The first, second and third rows of the figure show the fitting of the NLIAlg for $K=1,3$ and 5. 

\begin{figure}[h!]
	\centering
	\begin{tabular}{ccc}
		\includegraphics[width=170pt]{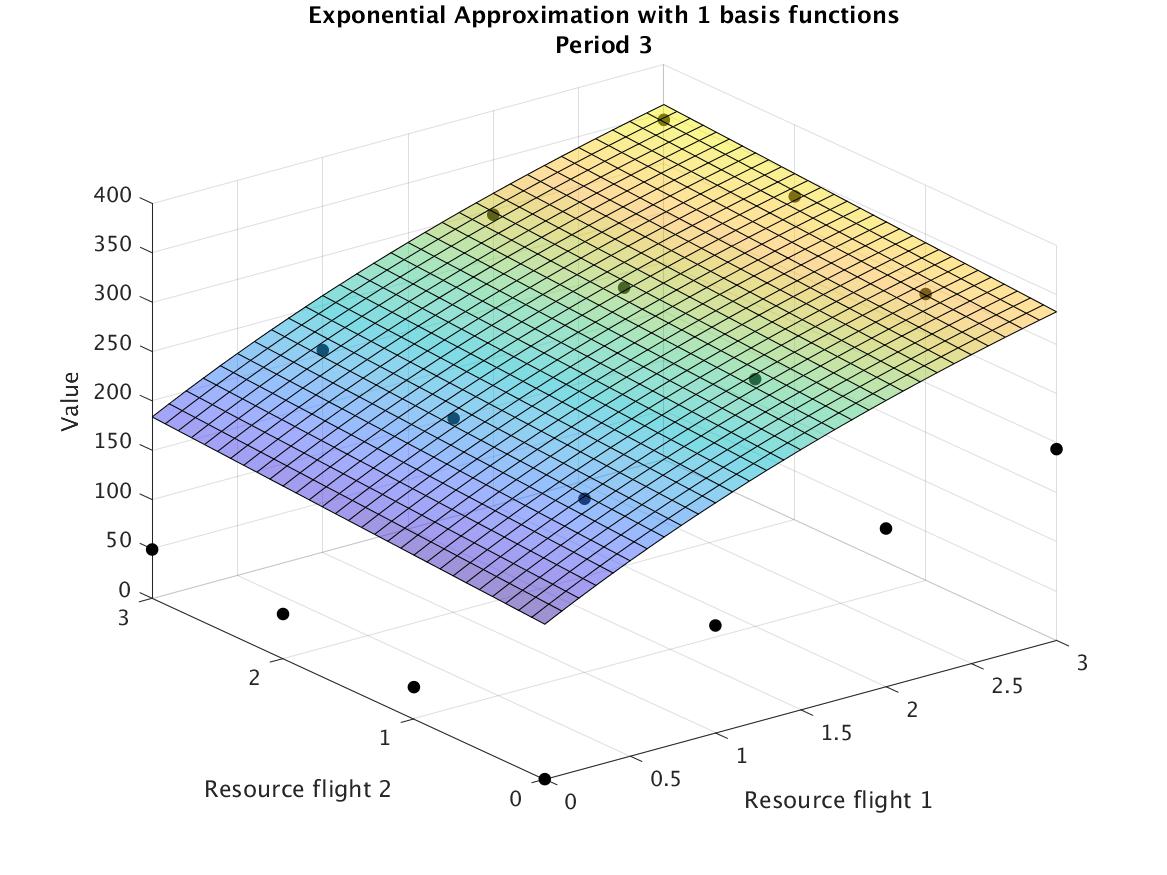} &  \includegraphics[width=170pt]{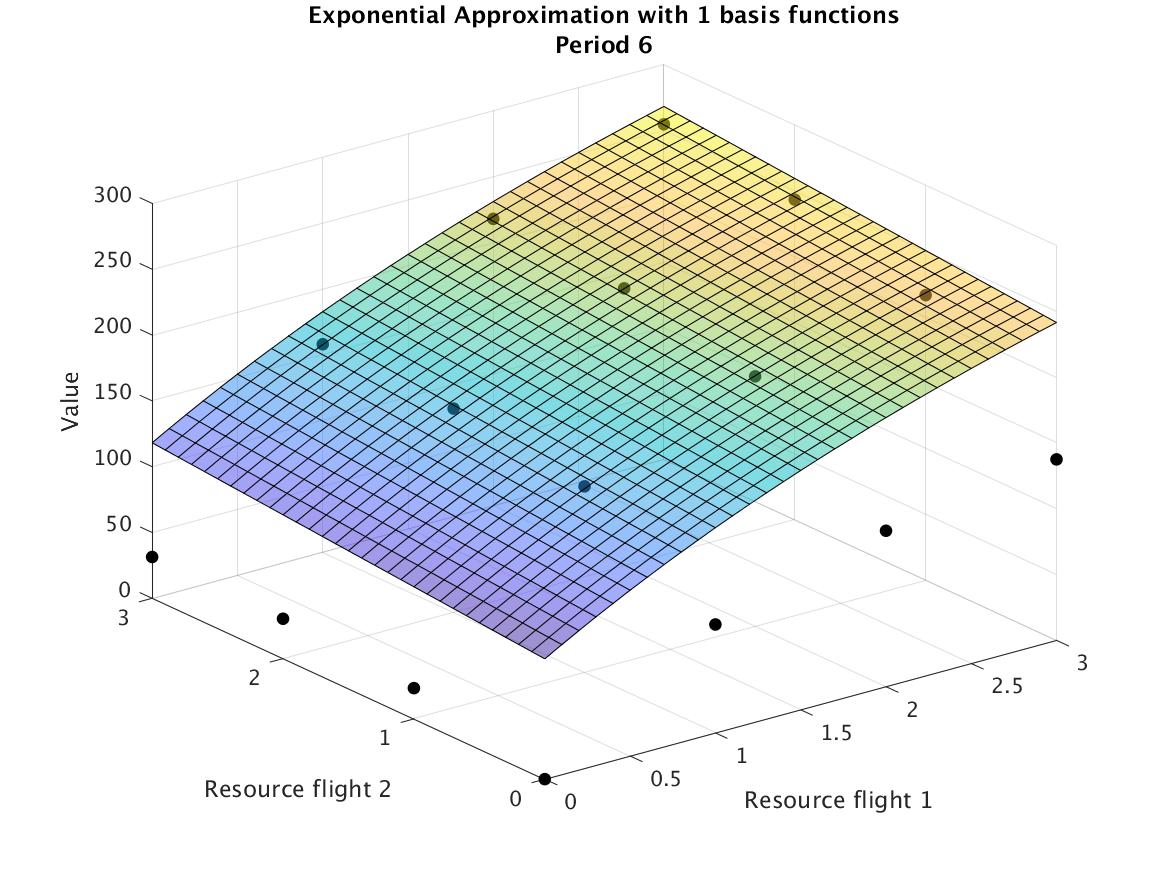} &\includegraphics[width=170pt]{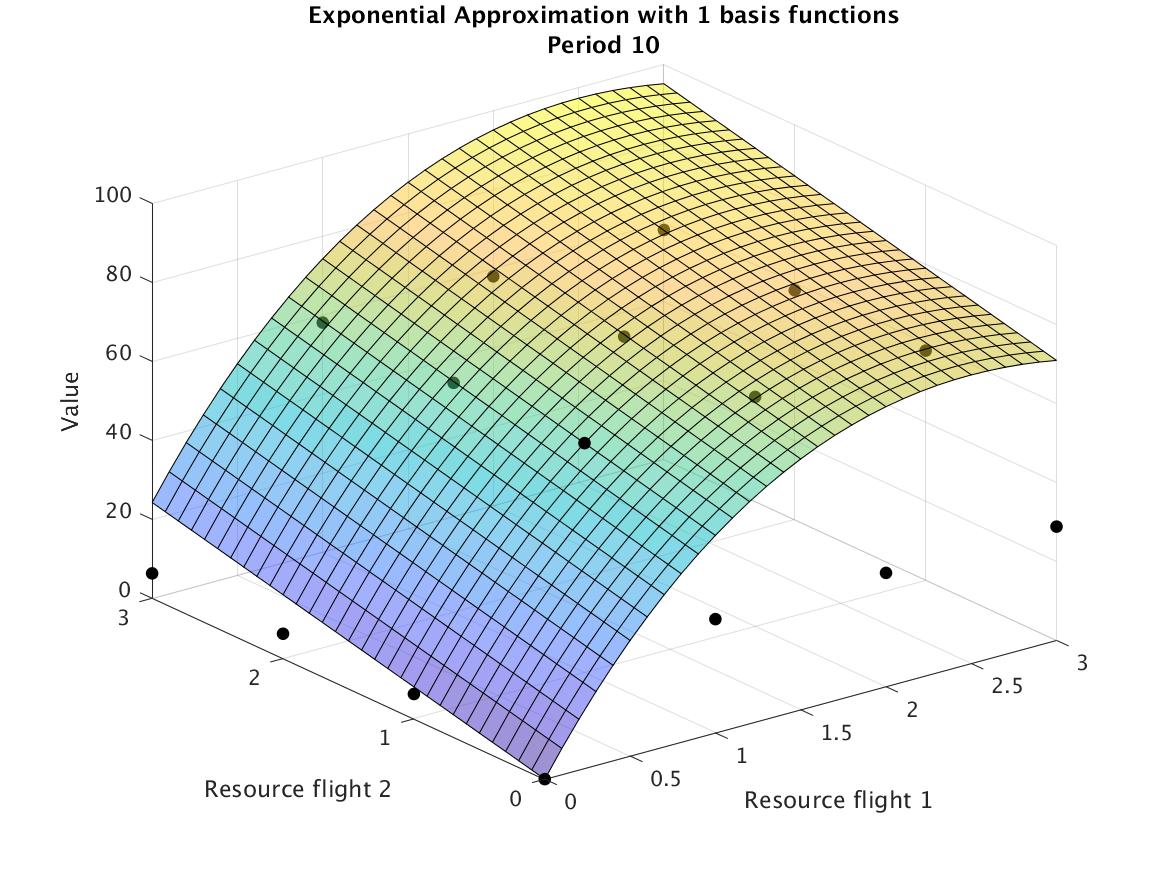} \\
		\includegraphics[width=170pt]{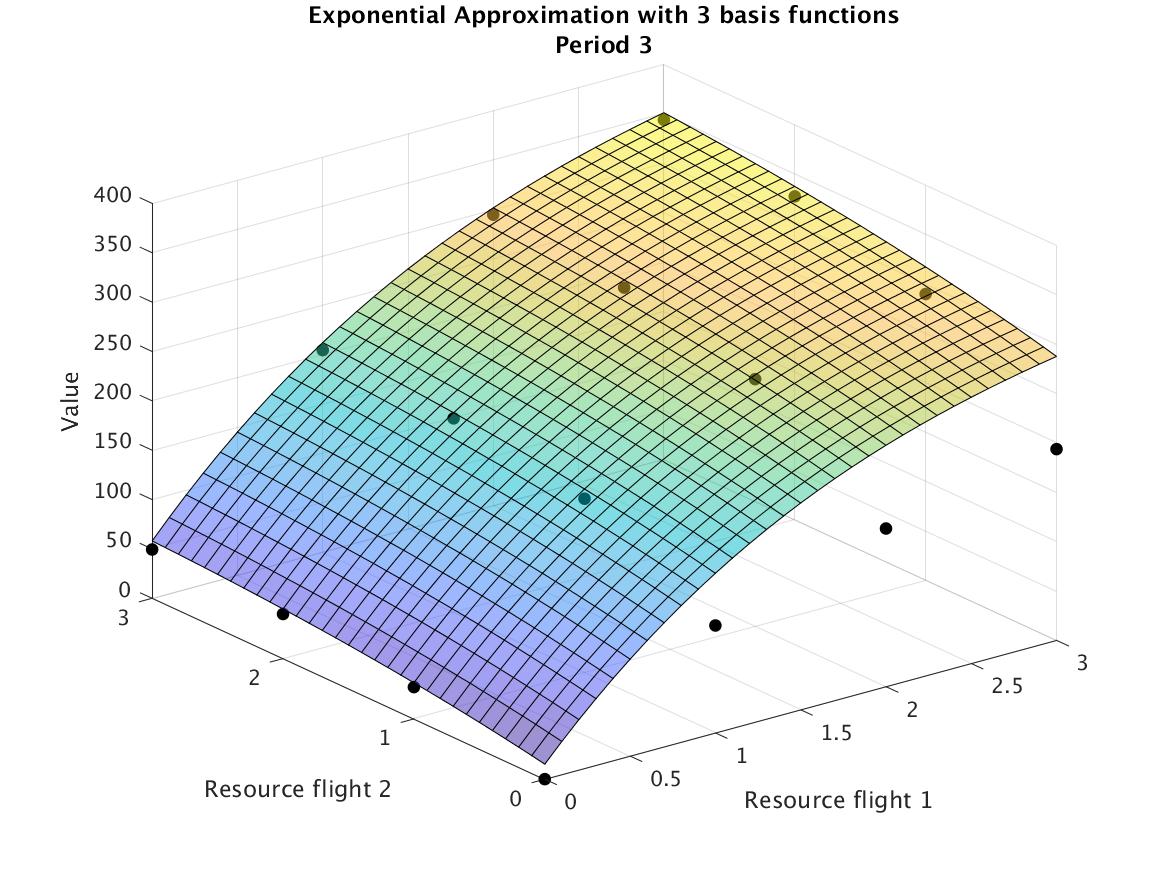} &  \includegraphics[width=170pt]{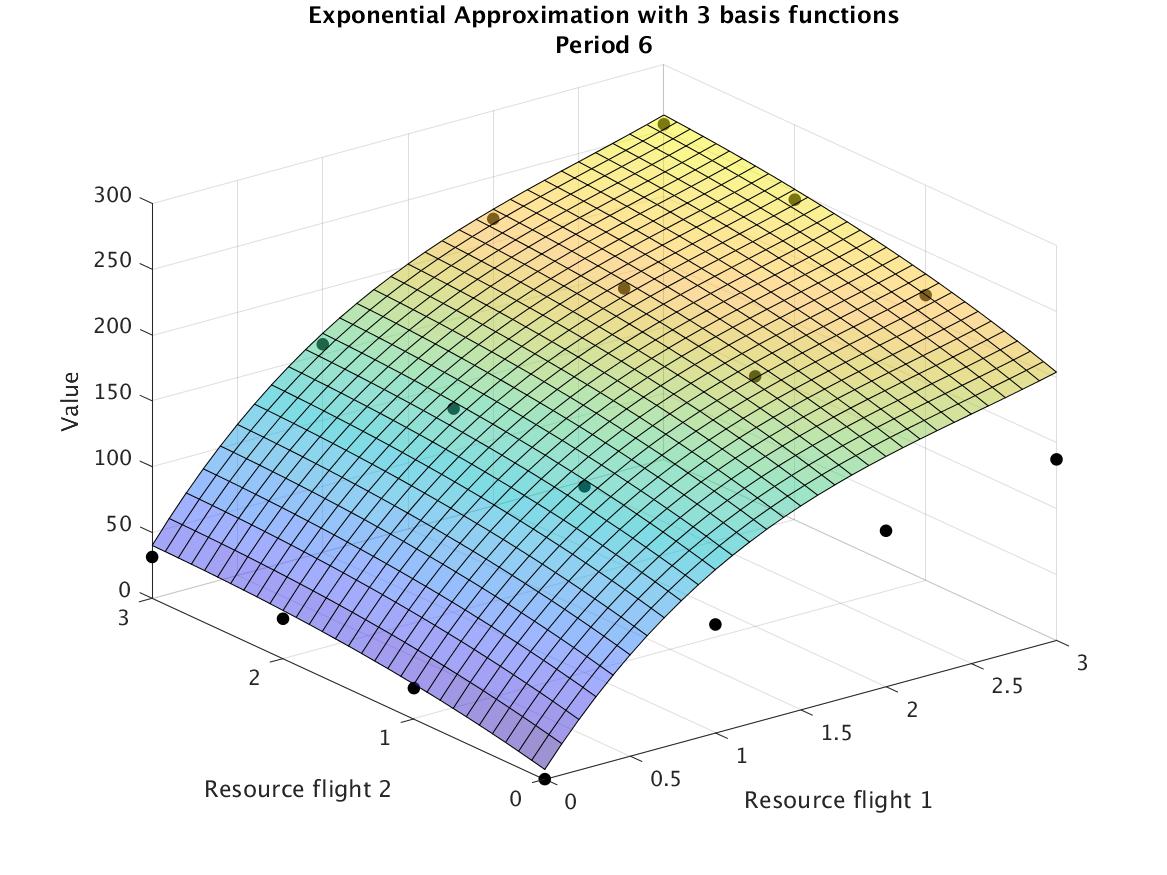} &\includegraphics[width=170pt]{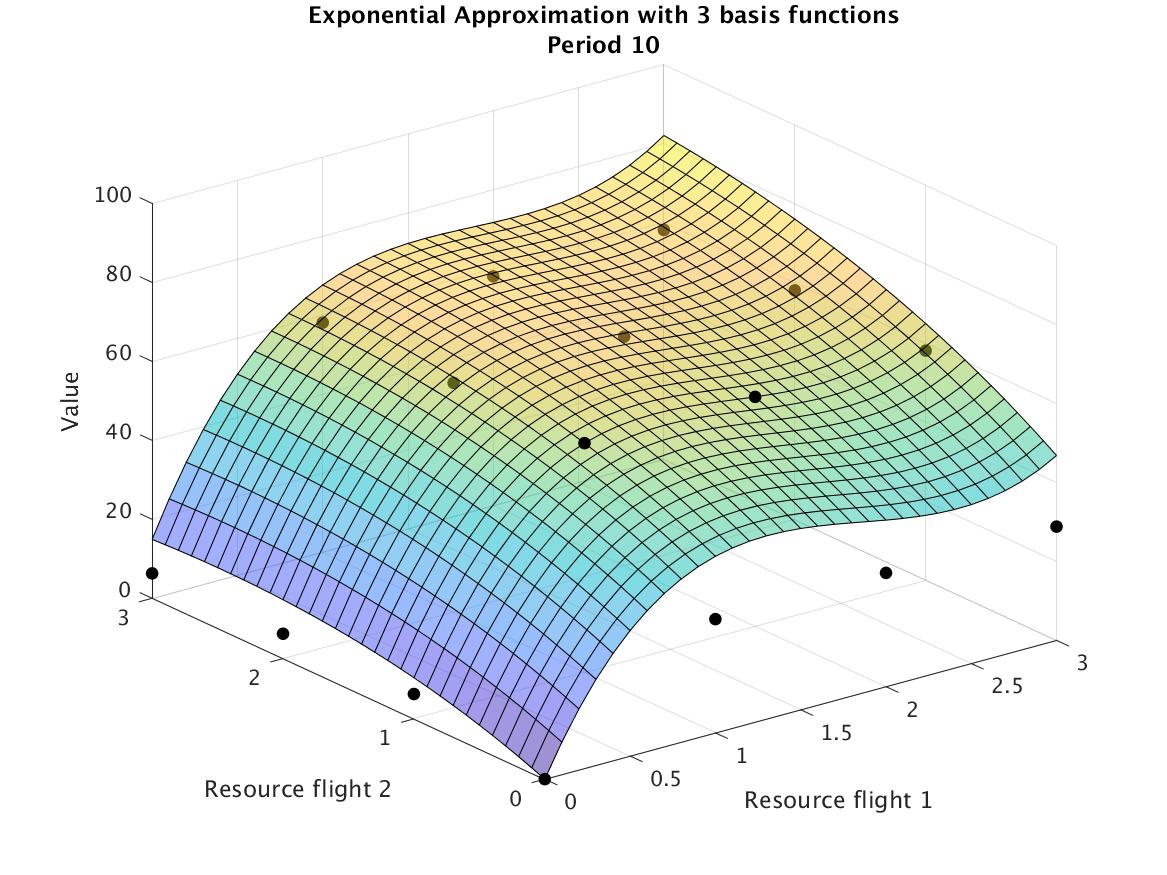} \\
		\includegraphics[width=170pt]{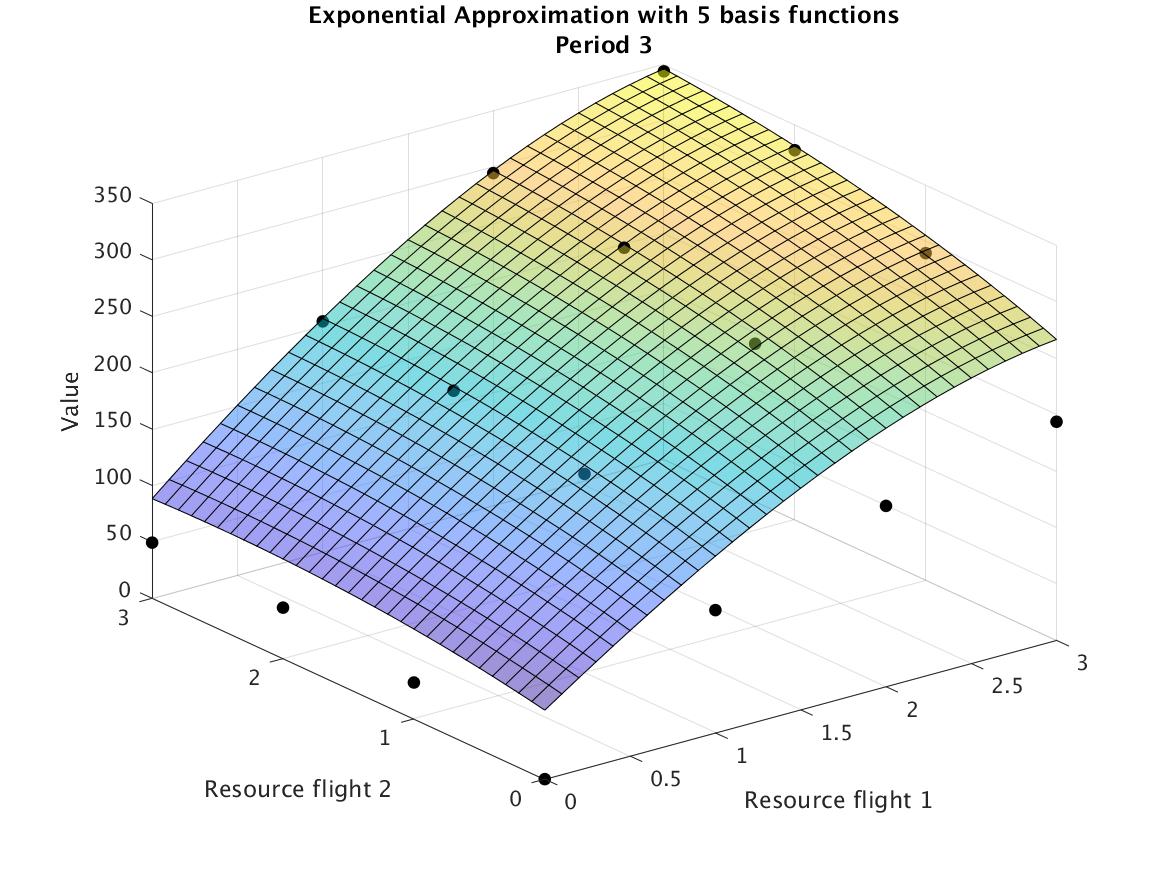} &  \includegraphics[width=170pt]{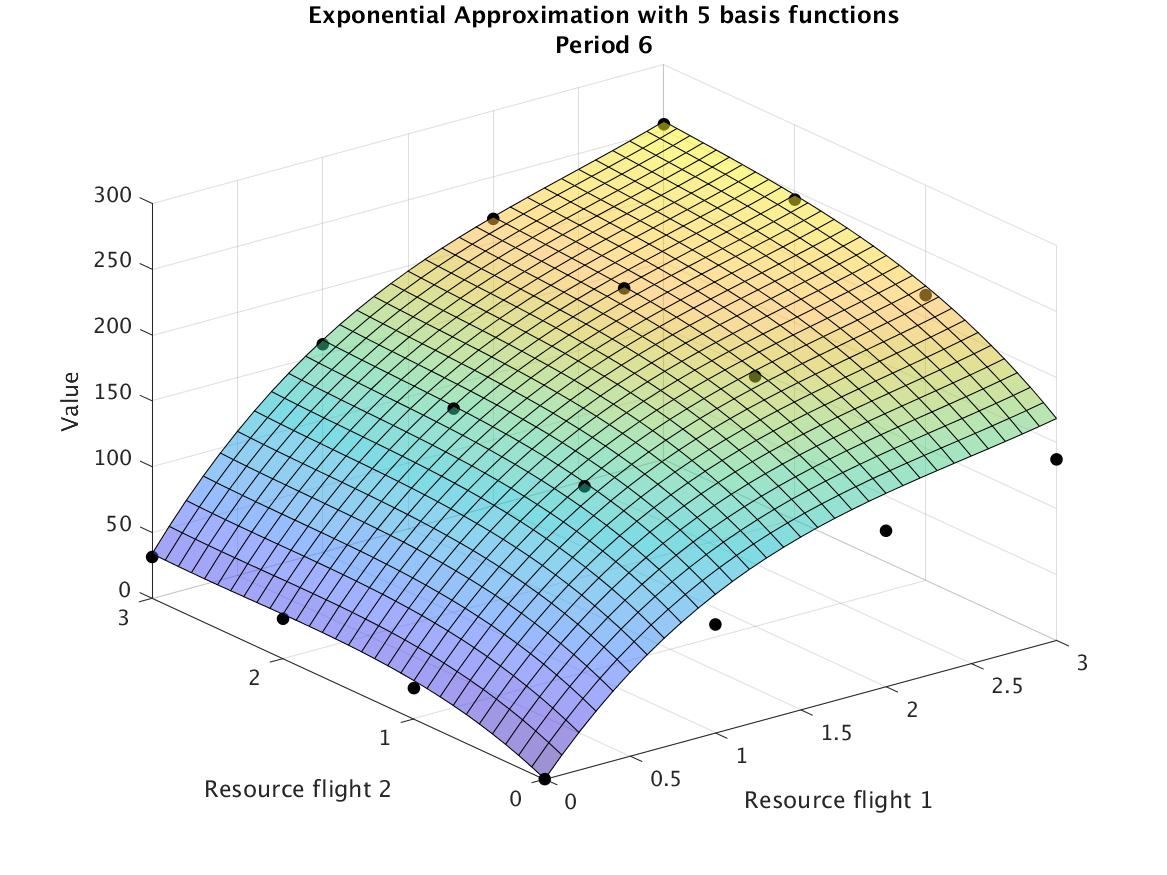} &\includegraphics[width=170pt]{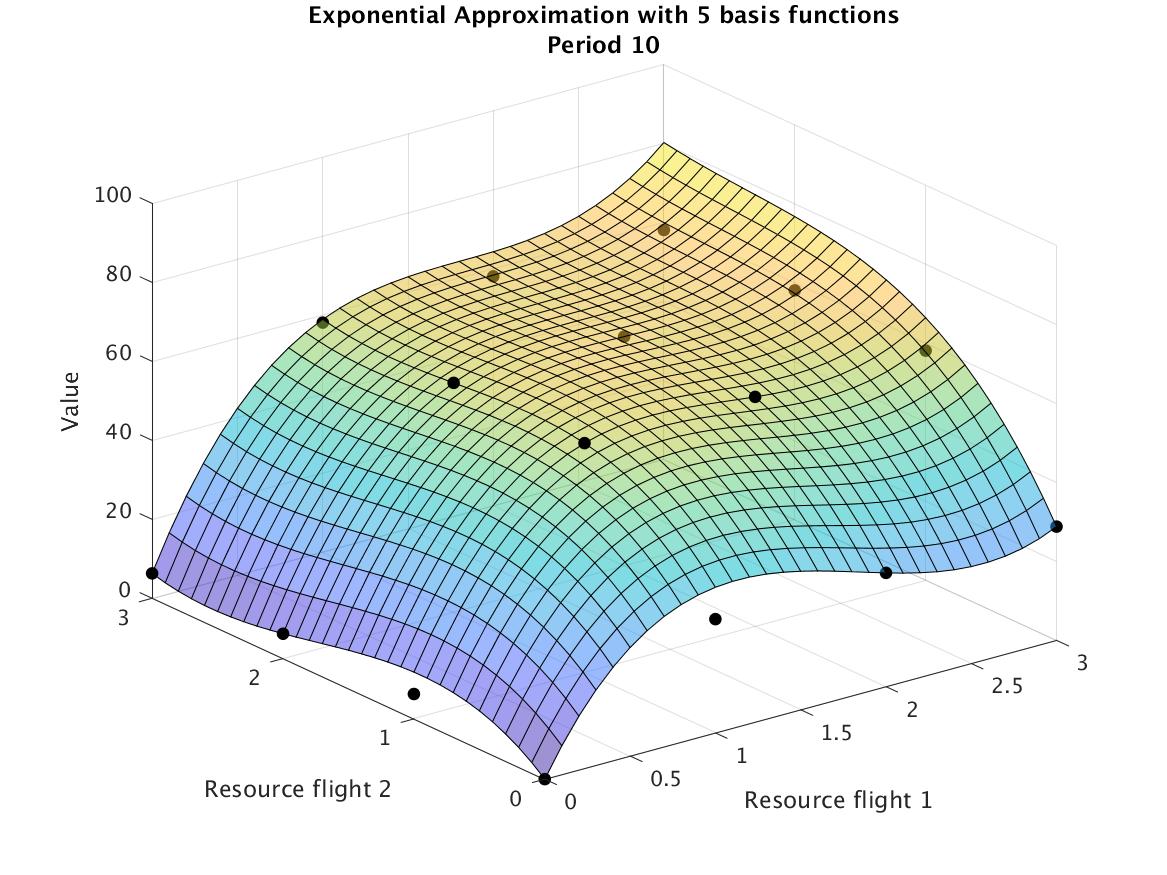} \\
	\end{tabular}\\
	\caption{\label{convergence_DC1}{Approximated value function through linear and exponential basis functions (NLIAlg) in the toy example introduced in Section \ref{sec:hat_vs_exponential}}}
\end{figure}

To further assess the tractability of NLIAlg,
Table \ref{DC_comparison} compares the NLIAlg against the H-2PIAlg for the smaller hub\&spoke instance in the paper. In this case, not only the NLIAlg takes longer than the H-2PIAlg, but the complexity of the DC master problem seems to make it difficult for the global solver to find a good solution within a reasonable time limit, affecting the convergence and the stability of the algorithm. This supports our decision to focus on the H-2PIAlg.

\begin{table}[h!]
	\centering	\begin{small}
		\begin{tabular}{l|ccc|ccc}
			&\multicolumn{3}{c|}{NLIAlg}&\multicolumn{3}{c}{H-2PIAlg}\\
			$K$ 	& $\hat{Z}_{\phi}$ & $\overline{R}$ & Time & $\hat{Z}_{\phi}$ &$\overline{R}$ & Time\\
			\hline									
			$1$	&	524.3	&	746.2	&	1223&	873.9	&	773.5	&	35	\\
			$2$	&	823.6	&	773.6	&	6347&	864.1	&	786.2	&	70	\\
			$3$	&	807.8	&	794.6	&	7873&	860.9	&	792.4	&	106	\\
			$4$	&	786.6	&	785.2	&	8495&	858.6	&	794.0	&	144	\\
			$5$	&&&&	853.1	&	796.2	&	188	\\
			$6$	&&&&	852.7	&	795.8	&	227	\\
			$7$	&&&&	850.5	&	798.7	&	274	\\
			$8$	&&&&	849.4	&	799.6	&	319	\\
			$9$	&&&&	845.6	&	800.1	&	369	\\
			$10$	&&&&	840.8	&	802.8	&	426	\\
			$11$	&&&&	840.6	&	805.5	&	482	\\
			$12$	&&&&	836.1	&	806.0	&	545	\\
			$13$	&&&&	828.8	&	808.6	&	645	\\
			$14$	&&&&	824.1	&	809.1	&	721	\\
			$15$	&&&&	821.1	&	809.2	&	857	\\
			$16$	&&&&	821.0	&	809.4	&	1140	\\
			$17$	&&&&	819.8	&	809.4	&	1514	\\
			$18$	&&&&	819.6	&	809.6	&	1838	\\
		\end{tabular}
	\end{small}
	\caption{\label{DC_comparison} Comparison of the NLIAlg and the H-2PIAlg for the smallest instance in Section \ref{sec:compare_baselines} with no baseline approximation (i.e., $\psi_t(\vx)=0$ for all $t=1,...,\tau$, $\vc \in \mathcal{X}_t$).}
\end{table}


\subsection{Standalone versus Add-on Mode. \label{sec:compare_baselines}}

In this section we compare the two different modes of H-2PIAlg, namely

\begin{itemize}
    \item[(1)] {\bf Standalone mode}: find a suitable value function approximation from scratch; i.e. $\psi_t(\vx)=0$,
    \item[(2)] {\bf Add-on mode}: enhance a given value function approximation; i.e., $\psi_t(\vx)\neq 0$.
\end{itemize}

For the add-on mode we choose the Affine Approximation \eqref{AA} because it exhibits Property (b); i.e., it is very tractable to generate rows since $\psi_t(\vx)$ is linear in $\vx$. In addition, it provides a reasonable approximation of the real value function.

  We use the smallest instances of the experimental setup specified in  \cite{adelman2007}: 2 non-hub locations and 2 fares, low and high, resulting in $4$ single-leg itineraries and $2$ two-leg itineraries. The initial capacity is the same for each flight, i.e. $c_i=c$ for all $i=1,...,I$. The values of  $c$ and $\tau$ are chosen such that the load factor,
i.e.~the ratio of demand to supply, is approximately $1.6$. Both  fares and arrival probabilities are randomly generated. For these small problems with $c=3$ and $c=8$, we can determine the true value function via value iteration, which takes 1,293 CPU seconds and  295,510 CPU seconds for the small and the large instances, respectively.

\begin{figure}[h!]
	\begin{tabular}{cc}
		\includegraphics[width=240pt]{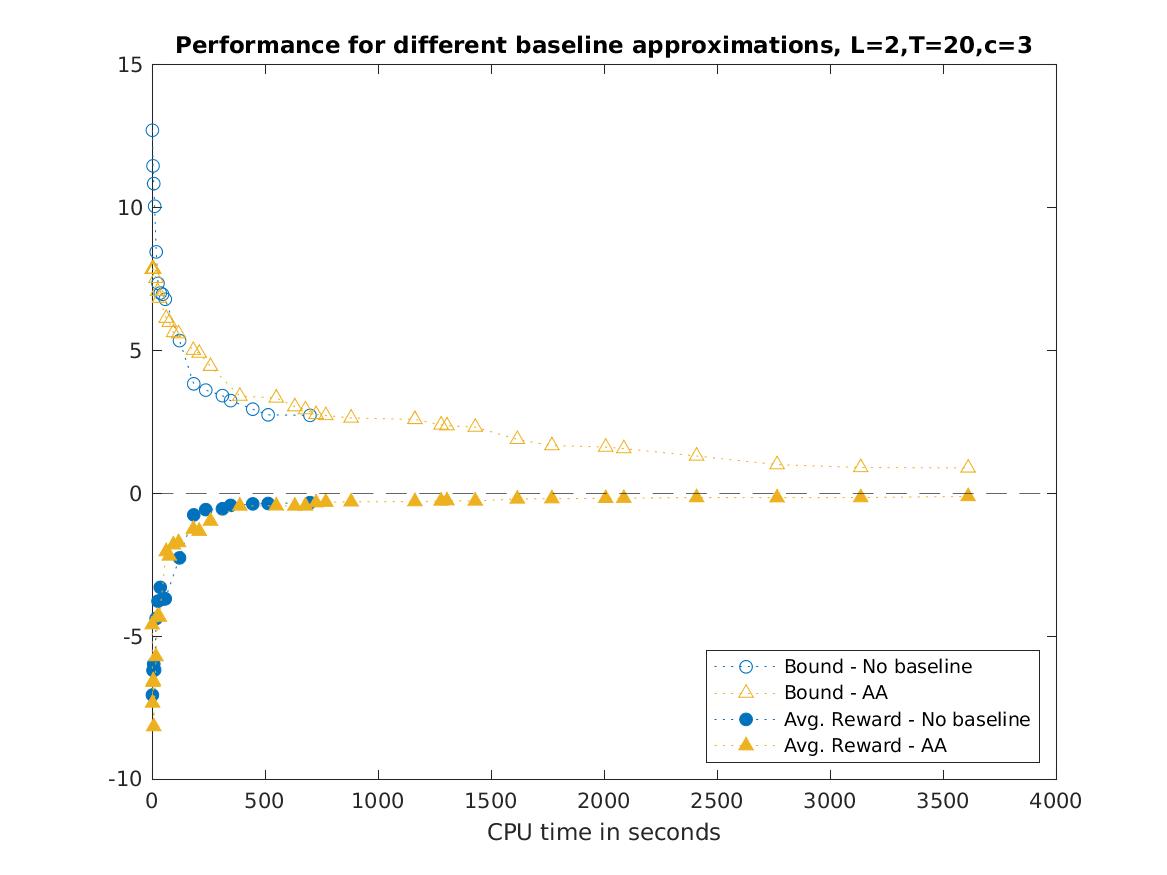} & \includegraphics[width=240pt]{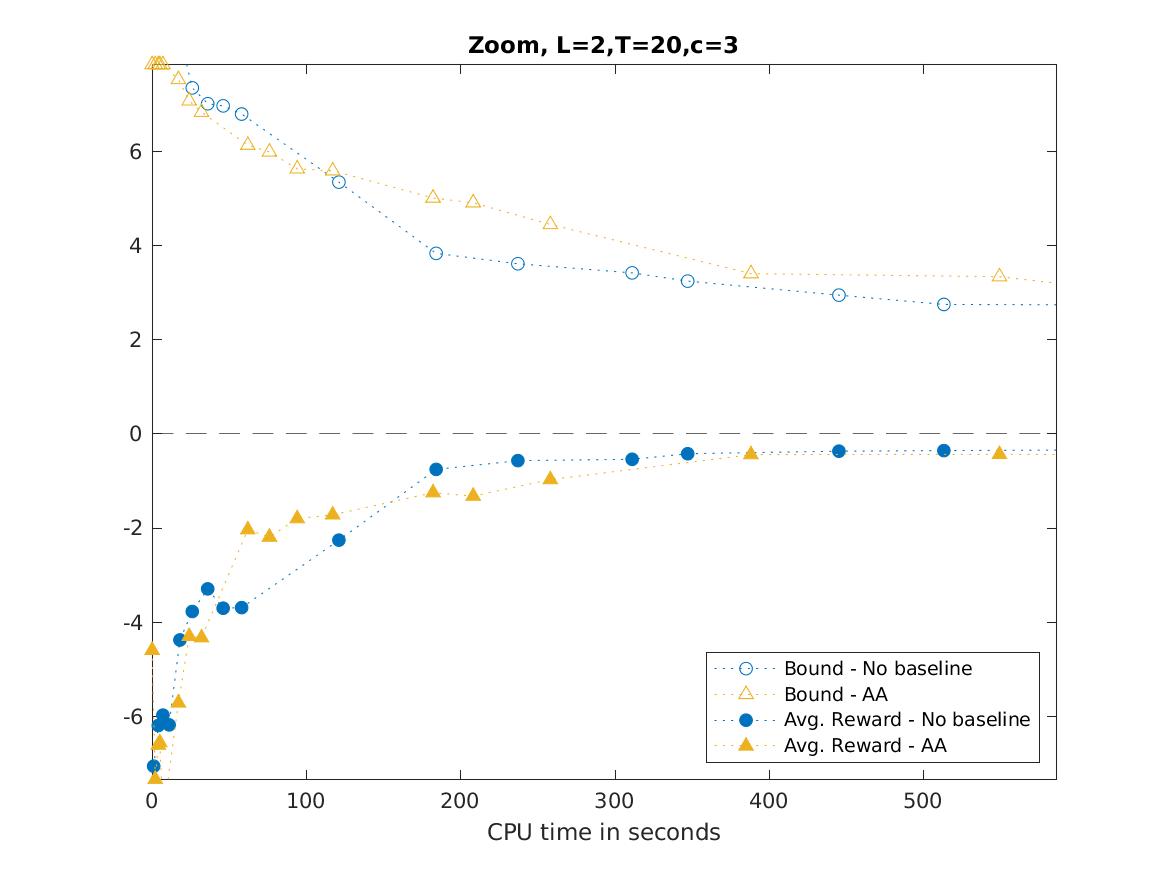}   \\
		\includegraphics[width=240pt]{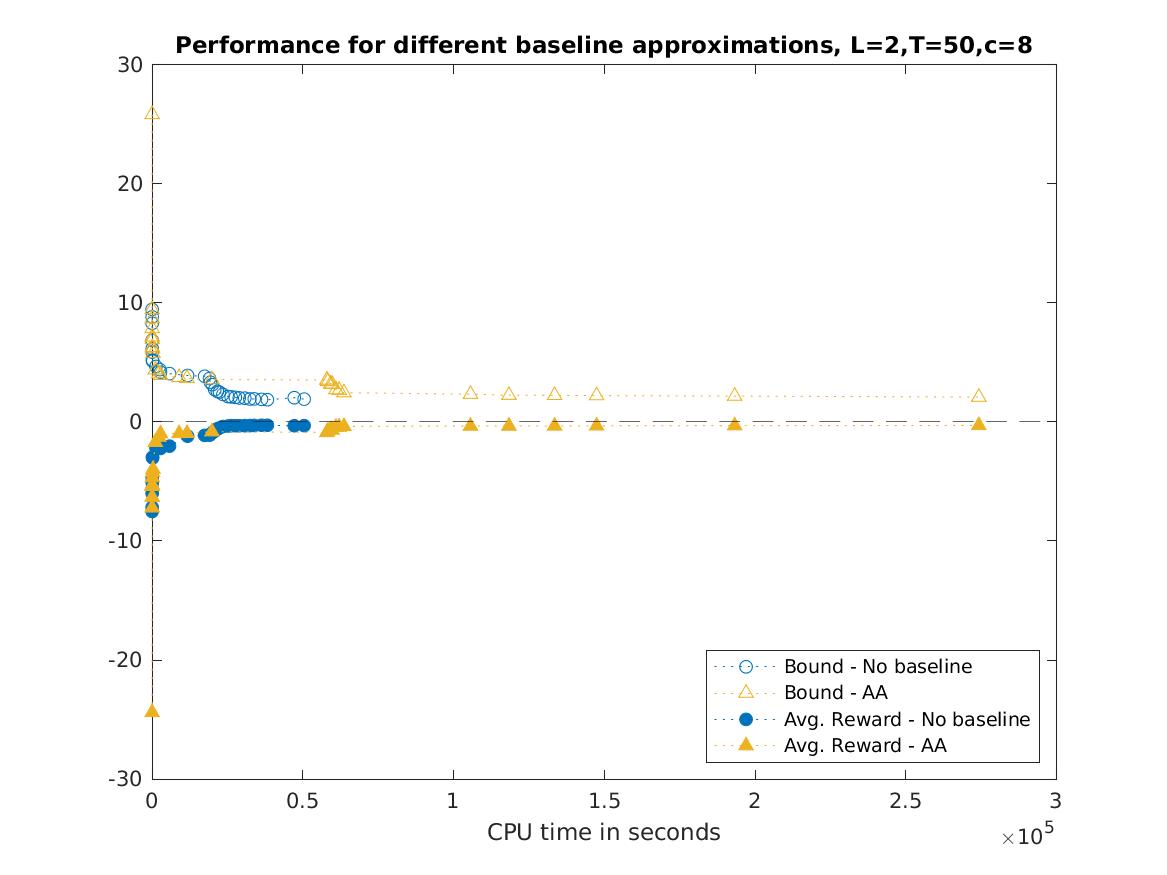} & \includegraphics[width=240pt]{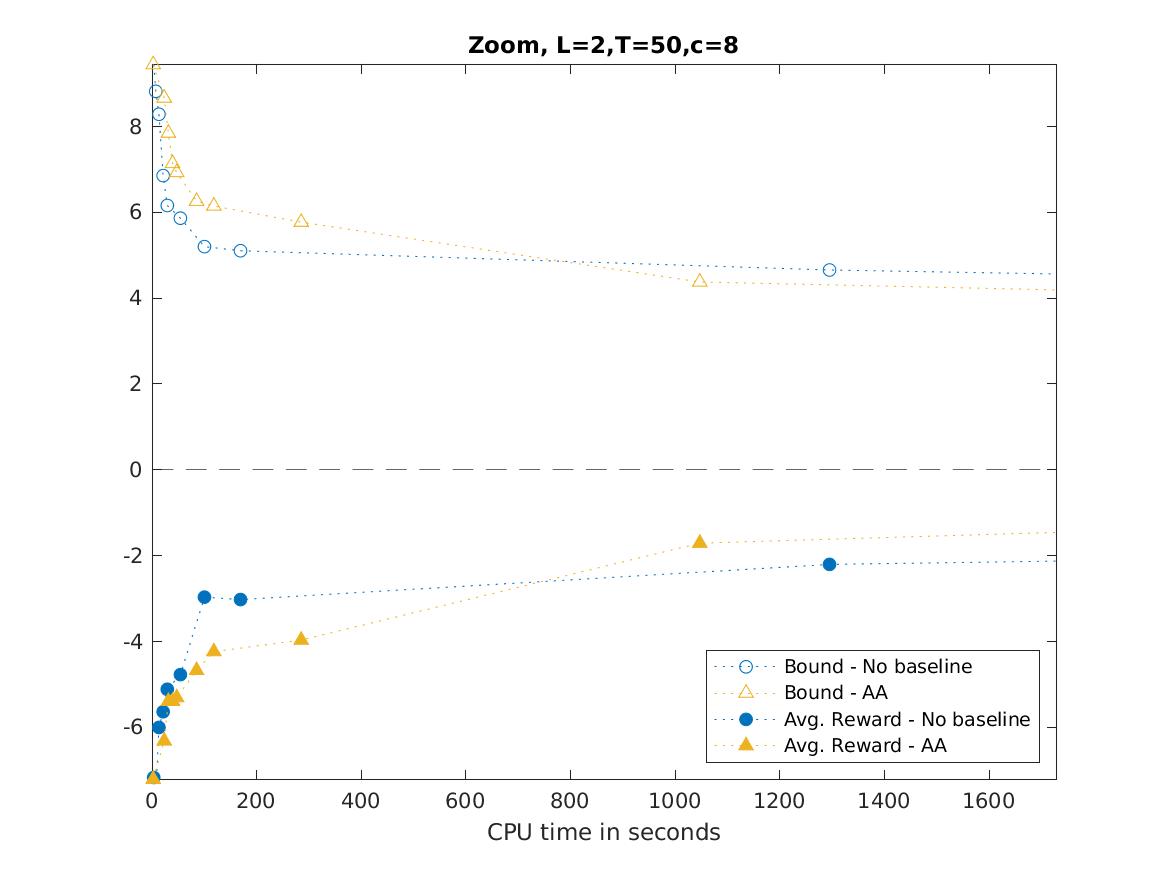}   \\
	\end{tabular}
	\caption{\label{convergence1}{Estimates of the upper bound $\hat{Z}_\phi=Z_{\mathcal{B}}+\sum_t \hat{\pi}_t$ and average revenue $\overline{R}$ from simulated policies for different baseline approximations vs. solving CPU times.}}
\end{figure}

Figure \ref{convergence1} shows both the estimated upper bound and the simulated average revenue generated by the value function approximation of the H-2PIAlg when increasing $K$, which translates in increasing CPU times. We report all bounds and average revenues as the relative difference with respect to the true value $v_1(\vc)$. More specifically, the normalization performed is 
\begin{equation}
\Delta\eta_A=\left(\frac{\eta_A}{v_1(\vc)}-1 \right)\ccdot 100 \%,
\end{equation}
where $\eta_A$ represents the estimated upper bound or the average revenue of an approximation $A$.

Since H-2PIAlg adds new basis functions to the baseline approximations, it is not surprising that they provide tighter estimates of upper bounds even for $K=1$. Average simulated revenues of the improved approximations, i.e., for $K\ge 1$, are larger than the revenues obtained by the policies based on the baseline approximations with $K=0$. The more basis functions are added (i.e., the larger the $K$), the lower the values of the objective function $\hat{Z}_\mathcal{\phi}$ and the higher the simulated average revenues $\overline{R}$. Interestingly, the initial approximation does not seem to strongly affect the convergence of H-2PIAlg. In other words, for different initial approximations $\psi_t(\cdot)$ the H-2PIAlg converges to the same revenue $\overline{R}$. A similar value $\hat{Z}_\phi$ is also achieved regardless of the choice of $\psi_t(\cdot)$. 

\subsection{CPU times and number of basis functions generated \label{times}}

\begin{table}[h!]																	
\centering																	
\begin{small}																	
\begin{tabular}{@{} *{3}{S[table-format=3.0]}|*{5}{S[table-format=6.0]}|*{1}{S[table-format=2.2]}@{}*{1}{S[table-format=2.0]}}														
\multicolumn{3}{c|}{}	&			\multicolumn{5}{c|}{ CPU times}	&			\multicolumn{2}{c}{H-2PIAlg Analysis}			\\
	&		&		&	{SPLA}	& {NSEP} &{AA}	&	{ H-2PIAlg}	&	{ H-2PIAlg+AA}	&  {Gap}	& {K}		\\
\cline{4-10}	
{ $L$}	&	{ $\tau$}	&	 {$c$} & \multicolumn{7}{c}{\bf Hub-and-Spoke instances (H\&S)}\\
\hline\hline
{\multirow{6}{*}{2}}	&	20	&	3	&	15	&	18498	&	66	&	6366	&	6432	&	0.90	&	32	\\
	&	50	&	8	&	30	&	1860	&	58	&	323127	&	323185	&	2.12	&	31	\\
	&	100	&	17	&	52	&	22212	&	122	&	549877	&	549999	&	3.04	&	23	\\
	&	200	&	33	&	98	&	$^\dagger$  	&	116	&	539171	&	539287	&	4.45	&	12	\\
	&	500	&	83	&	508	&	$^\dagger$  	&	184	&	154	&	338	&	4.88	&	2	\\
	&	1000	&	165	&	2465	&	$^\dagger$  	&	185	&	360	&	545	&	3.40	&	2	\\
\hline																			
{\multirow{6}{*}{3}}	&	20	&	2	&	18	&	20980	&	304	&	113962	&	114266	&	1.00	&	43	\\
	&	50	&	6	&	36	&	1737	&	278	&	372511	&	372789	&	3.89	&	37	\\
	&	100	&	12	&	64	&	176	&	395	&	227744	&	228139	&	5.38	&	28	\\
	&	200	&	25	&	142	&	$^\dagger$  	&	371	&	350985	&	351356	&	4.66	&	20	\\
	&	500	&	62	&	841	&	$^\dagger$  	&	501	&	97710	&	98211	&	3.79	&	13	\\
	&	1000	&	124	&	5330	&	$^\dagger$  	&	481	&	496545	&	497026	&	5.59	&	12	\\
\hline																			
{\multirow{6}{*}{5}}	&	20	&	2	&	196	&	17151	&	794	&	130191	&	130985	&	5.10	&	42	\\
	&	50	&	4	&	220	&	4006	&	729	&	584713	&	585442	&	8.26	&	39	\\
	&	100	&	8	&	261	&	56890	&	1045	&	601360	&	602405	&	10.02	&	38	\\
	&	200	&	17	&	402	&	$^\dagger$  	&	975	&	503696	&	504671	&	10.20	&	19	\\
	&	500	&	41	&	5800	&	$^\dagger$  	&	1315	&	434145	&	435460	&	7.07	&	15	\\
	&	1000	&	83	&	$^\dagger$  	&	$^\dagger$  	&	1239	&	534326	&	535565	&	7.97	&	13	\\
\hline																			
{\multirow{6}{*}{10}}	&	20	&	1	&	49	&	$^\dagger$  	&	3090	&	203039	&	206129	&	8.99	&	37	\\
	&	50	&	2	&	97	&	$^\dagger$  	&	2774	&	473233	&	476007	&	26.48	&	15	\\
	&	100	&	5	&	179	&	$^\dagger$  	&	4566	&	380178	&	384744	&	19.99	&	18	\\
	&	200	&	9	&	471	&	$^\dagger$  	&	4270	&	498801	&	503071	&	19.12	&	16	\\
	&	500	&	23	&	7253	&	$^\dagger$  	&	5962	&	489085	&	495047	&	15.54	&	13	\\
	&	1000	&	45	&	$^\dagger$  	&	$^\dagger$  	&	5704	&	312874	&	318578	&	13.78	&	11	\\
\hline																			
{\multirow{6}{*}{20}}	&	20	&	1	&	150	&	$^\dagger$  	&	19263	&	544595	&	563858	&	15.68	&	19	\\
	&	50	&	2	&	327	&	$^\dagger$  	&	19426	&	5208	&	24634	&	34.30	&	1	\\
	&	100	&	2	&	507	&	$^\dagger$  	&	31175	&	5317	&	36492	&	30.43	&	2	\\
	&	200	&	5	&	984	&	$^\dagger$  	&	30654	&	203212	&	233866	&	25.32	&	6	\\
	&	500	&	12	&	$^\dagger$  	&	$^\dagger$  	&	41782	&	18424	&	60206	&	20.92	&	2	\\
	&	1000	&	24	&	$^\dagger$  	&	$^\dagger$  	&	40914	&	359286	&	400200	&	17.68	&	5	\\
\hline
\multicolumn{10}{c}{\bf Bus instances}\\
\hline
\hline
\multicolumn{3}{r|}{Simple (SBL)}					&	43	&	532	&	43	&	290389	&	290432	&	8.20	&	41	\\
\multicolumn{3}{r|}{Consecutive (CBL)}					&	14	&	218	&	18	&	552779	&	552797	&	3.54	&	60	\\
\multicolumn{3}{r|}{Realistic (RBL)}					&	3635	&	$^\dagger$  	&	19	&	57978	&	57997	&	7.10	&	16	\\
\end{tabular}																	
\end{small}																	
\caption{CPU times in seconds of SPLA, NSEP, AA, and  H-2PIAlg with AA baseline for the hub-and-spoke and bus networks.  For H-2PIAlg, we display both the times taken to merely enhance AA, and the times including the computation of the AA baseline (H-2PIAlg+AA column). The last set of columns display the optimality gap (UB vs LB) of H-2PIAlg and the number of exponential ridge basis functions for which the best policy was obtained.  \label{table:times}}	
\end{table}


\bibliographystyle{jf}
\bibliography{references}

@article{zhang2011,
  title={An improved dynamic programming decomposition approach for network revenue management},
  author={Zhang, Dan},
  journal={Manufacturing \& Service Operations Management},
  volume={13},
  number={1},
  pages={35--52},
  year={2011},
  publisher={INFORMS}
}

@article{adelman2004,
  title={A price-directed approach to stochastic inventory/routing},
  author={Adelman, Daniel},
  journal={Operations Research},
  volume={52},
  number={4},
  pages={499--514},
  year={2004},
  publisher={INFORMS}
}

@article{adelmanklabjan2007,
  title={An infinite-dimensional linear programming algorithm for deterministic semi-Markov decision processes on Borel spaces},
  author={Klabjan, Diego and Adelman, Daniel},
  journal={Mathematics of Operations Research},
  volume={32},
  number={3},
  pages={528--550},
  year={2007},
  publisher={INFORMS}
}

@article{adelmanklabjan2012,
  title={Computing near-optimal policies in generalized joint replenishment},
  author={Adelman, Daniel and Klabjan, Diego},
  journal={INFORMS Journal on Computing},
  volume={24},
  number={1},
  pages={148--164},
  year={2012},
  publisher={INFORMS}
}

@article{bellman,
  title={Dynamic programming},
  author={Bellman, Richard},
  journal={Science},
  volume={153},
  number={3731},
  pages={34--37},
  year={1966},
  publisher={American Association for the Advancement of Science}
}

@article{bhat,
  title={Nonparametric Approximate Dynamic Programming via the Kernel Method},
  author={Bhat, Nikhil and Farias, Vivek and Moallemi, Ciamac C and Zheng, Andrew T},
  journal={Stochastic Systems},
 volume = {13},
number = {3},
pages = {321-342},
year = {2023},
doi = {10.1287/stsy.2023.0107},
URL = { 
},
}

@article{fariasroy2003,
  title={The linear programming approach to approximate dynamic programming},
  author={De Farias, Daniela Pucci and Van Roy, Benjamin},
  journal={Operations research},
  volume={51},
  number={6},
  pages={850--865},
  year={2003},
  publisher={INFORMS}
}

@article{fariasroy2004,
  title={On constraint sampling in the linear programming approach to approximate dynamic programming},
  author={De Farias, Daniela Pucci and Van Roy, Benjamin},
  journal={Mathematics of operations research},
  volume={29},
  number={3},
  pages={462--478},
  year={2004},
  publisher={INFORMS}
}

@techreport{fariasroy2007,
  title={An approximate dynamic programming approach to network revenue management},
  author={Farias, Vivek F and Van Roy, Benjamin},
  year={2007},
  institution={Working paper}
}

@article{guestrinkoller2002,
  title={Efficient solution algorithms for factored MDPs},
  author={Guestrin, Carlos and Koller, Daphne and Parr, Ronald and Venkataraman, Shobha},
  journal={Journal of Artificial Intelligence Research},
  volume={19},
  pages={399--468},
  year={2003}
}

@article{kunnumkaltalluri2014,
  title={On a piecewise-linear approximation for network revenue management},
  author={Kunnumkal, Sumit and Talluri, Kalyan},
  journal={Mathematics of Operations Research},
  volume={41},
  number={1},
  pages={72--91},
  year={2016},
  publisher={INFORMS}
}

@article{kunnumkaltalluri2015b,
  title={Choice network revenue management based on new tractable approximations},
  author={Kunnumkal, Sumit and Talluri, Kalyan},
  journal={Transportation Science},
  volume={53},
  number={6},
  pages={1591--1608},
  year={2019},
  publisher={INFORMS}
}

@article{linetal2019,
  title={Revisiting approximate linear programming: Constraint-violation learning with applications to inventory control and energy storage},
  author={Lin, Qihang and Nadarajah, Selvaprabu and Soheili, Negar},
  journal={Management Science},
  volume={66},
  number={4},
  pages={1544--1562},
  year={2020},
  publisher={INFORMS}
}

@techreport{pakimanetal2019,
  title={Self-guided Approximate Linear Programs},
  author={Pakiman, Parshan and Nadarajah, Selvaprabu and Soheili, Negar and Lin, Qihang},
  institution={Working paper},
  year={2020}
}

@book{powell2007,
  title={Approximate Dynamic Programming: Solving the curses of dimensionality},
  author={Powell, Warren B},
  volume={703},
  year={2011},
  publisher={John Wiley \& Sons}
}

@article{schweitzerseidmann1985,
  title={Generalized polynomial approximations in Markovian decision processes},
  author={Schweitzer, Paul J and Seidmann, Abraham},
  journal={Journal of mathematical analysis and applications},
  volume={110},
  number={2},
  pages={568--582},
  year={1985},
  publisher={Elsevier}
}

@article{Simon,
title = {Reductions of non-separable approximate linear programs for network revenue management},
journal = {European Journal of Operational Research},
volume = {309},
number = {1},
pages = {252-270},
year = {2023},
author = {Simon Laumer and Christiane Barz},
}

@article{talluriryzin1998,
  title={An analysis of bid-price controls for network revenue management},
  author={Talluri, Kalyan and Van Ryzin, Garrett},
  journal={Management science},
  volume={44},
  number={11-part-1},
  pages={1577--1593},
  year={1998},
  publisher={INFORMS}
}

@book{talluriryzin2004book,
  title={The theory and practice of revenue management},
  author={Talluri, Kalyan T and Van Ryzin, Garrett and Van Ryzin, Garrett},
  volume={1},
  year={2004},
  publisher={Springer}
}

@article{topaloglu2009,
  title={Using Lagrangian relaxation to compute capacity-dependent bid prices in network revenue management},
  author={Topaloglu, Huseyin},
  journal={Operations Research},
  volume={57},
  number={3},
  pages={637--649},
  year={2009},
  publisher={INFORMS}
}

@article{trickzin1997,
  title={Spline approximations to value functions: linear programming approach},
  author={Trick, Michael A and Zin, Stanley E},
  journal={Macroeconomic Dynamics},
  volume={1},
  number={1},
  pages={255--277},
  year={1997},
  publisher={Cambridge University Press}
}

@article{vossenzhang2014,
  title={A dynamic disaggregation approach to approximate linear programs for network revenue management},
  author={Vossen, Thomas WM and Zhang, Dan},
  journal={Production and Operations Management},
  volume={24},
  number={3},
  pages={469--487},
  year={2015},
  publisher={Wiley Online Library}
}

@article{zhangvossen2015,
  title={Reductions of approximate linear programs for network revenue management},
  author={Vossen, Thomas WM and Zhang, Dan},
  journal={Operations Research},
  volume={63},
  number={6},
  pages={1352--1371},
  year={2015},
  publisher={INFORMS}
}

@article{zhangadelman2009,
  title={An approximate dynamic programming approach to network revenue management with customer choice},
  author={Zhang, Dan and Adelman, Daniel},
  journal={Transportation Science},
  volume={43},
  number={3},
  pages={381--394},
  year={2009},
  publisher={INFORMS}
}

@article{adelman2007,
  title={Dynamic bid prices in revenue management},
  author={Adelman, Daniel},
  journal={Operations Research},
  volume={55},
  number={4},
  pages={647--661},
  year={2007},
  publisher={INFORMS}
}

@article{ridge,
  title={The fundamentality of sets of ridge functions},
  author={Sun, Xingping and Cheney, Elliott Ward},
  journal={aequationes mathematicae},
  volume={44},
  number={2-3},
  pages={226--235},
  year={1992},
  publisher={Springer}
}

@article{topaloglu,
  title={On the approximate linear programming approach for network revenue management problems},
  author={Tong, Chaoxu and Topaloglu, Huseyin},
  journal={INFORMS Journal on Computing},
  volume={26},
  number={1},
  pages={121--134},
  year={2013},
  publisher={INFORMS}
}

@article{DC_overview,
  title={{DC} programming: overview},
  author={Horst, Reiner and Thoai, Nguyen V},
  journal={Journal of Optimization Theory and Applications},
  volume={103},
  number={1},
  pages={1--43},
  year={1999},
  publisher={Springer}
}
\end{document}